\documentclass[10pt,reqno]{amsart}
\usepackage[utf8]{inputenc}
\usepackage[T1]{fontenc}
\usepackage{lmodern}
\usepackage{xspace}
\usepackage{amsfonts,amsmath,amssymb,amsthm}
\usepackage[small]{titlesec}
\usepackage{comment}
\usepackage{mathrsfs}
\usepackage{fancyhdr}
\usepackage{xcolor}
\usepackage{upgreek}
\usepackage{frcursive}
\usepackage{mathrsfs}

\usepackage[a4paper,body={160mm,235mm},centering]{geometry}
\usepackage{hyperref}

\newcommand{\brho}{\boldsymbol{\rho}}

\newcommand{\R}{\mathbb{R}}
\newcommand{\I}{\mathbf{1}}

\newcommand{\ind}{\boldsymbol{1}}

\newcommand{\p}{\partial}
\newcommand{\ww}{{\rm w}}

\newcommand \A[1]{{\bf (#1)}}

\newcommand{\mW}{\mathcal W}
\newcommand{\bmu}{\boldsymbol \mu}

\newcommand{\cv}{c_{\mathbf Y}} 
\newcommand{\cl}{c_{\mathbf L}} 
\newcommand{\ch}{c_{\mathbf{HK}}} 

\newcommand{\Tspkr}{\mathcal{T}_1} 

\def\div{{\rm{div}}}

\renewcommand{\iff}{\Leftrightarrow}

\makeatletter
\newcounter{subhyp}
\let\savedc@hyp\c@hyp
\newcommand{\dotafter}[1]{#1.}
\titleformat{\section}[hang]
{\normalfont\large\bfseries}{\thesection.}{.5em}{\dotafter}[]
\titleformat{\subsection}[runin]
{\normalfont\bfseries}{\thesubsection.}{.4em}{}[.]
\titlespacing*{\subsection}{0pt}{3ex plus 1ex minus .2ex}{1em}
\titleformat{\paragraph}[runin]{\normalfont\bfseries}{\theparagraph.}{.4em}{}[.]

\theoremstyle{plain}
\newtheorem{thm}{Theorem}
\newtheorem{lemme}[thm]{Lemma}
\newtheorem{prop}[thm]{Proposition}

\theoremstyle{definition}

\theoremstyle{remark}
\newtheorem{rem}{Remark}

\renewcommand{\sc}{\mbox{\small{\begin{cursive}s\end{cursive}}}}
\newcommand{\cc}{\mbox{\small{\begin{cursive}c\end{cursive}}}}
\newcommand{\w}{{\rm w}}
\renewcommand{\|}{|}

\begin{document}
\title[Multidimensional stable driven MKV SDEs]{Multidimensional stable driven McKean-Vlasov SDEs with distributional interaction kernel: critical thresholds and related models}

\author{P.-E. Chaudru de Raynal}
\address{Laboratoire de Math\'ematiques Jean Leray, University of Nantes, 2, rue de la Houssini\`ere BP 92208
F-44322 Nantes Cedex 3, France.}
\email{pe.deraynal@univ-nantes.fr}
\author{J.-F. Jabir}
\address{Laboratory of Stochastic Analysis, HSE University, Pokrovsky Blvd, 11, Moscow, Russian Federation.}
\email{jjabir@hse.ru}
\author{S. Menozzi}
\address{LaMME, Universit\'e d'Evry Val d'Essonne, Universit\'e Paris-Saclay, CNRS (UMR 8071), 23 Boulevard de France 91037 Evry, France $\&$ Laboratory of Stochastic Analysis, HSE University, Pokrovsky Blvd, 11, Moscow, Russian Federation.}
\email{stephane.menozzi@univ-evry.fr}
\date{\today}
\maketitle

\begin{abstract}{\color{black}
In this work we continue to investigate well-posedness for stable driven McKean-Vlasov SDEs with distributional interaction kernel following the approach introduced in  \cite{chau:jabi:meno:22-1}. We specifically focus on the impact of the Besov smoothness of the initial condition and quantify how it affects the corresponding density estimates for the SDE. In particular, we manage to attain some \textit{critical} thresholds allowing to revisit/address in a stable noise setting some concrete physical and biological models.
}
\end{abstract}

\\


\section{Introduction and main results}

\subsection{Framework}
The present work is a follow-up to the previous paper \cite{chau:jabi:meno:22-1} where we investigated well-posedness results -in a weak and strong sense - alongside the distributional regularity of the McKean-Vlasov SDE: 
\begin{equation}\label{main}
X_s^{t,\mu} = \xi + \int_t^s \int  b(r,X_r^{t,\mu}-y) \bmu_r^{t,\mu}(dy) dr + (\mW_s-\mW_t),\:\:\:\bmu_{s}^{t,\mu}=\text{Law}(X_s^{t,\mu}),\ {\color{black} 0 \le t \le s < T,}
\end{equation}
 where $T>0$ is some positive time horizon {\color{black}and} 
the characteristic component $b$ corresponds to a singular interaction kernel lying in a Lebesgue-Besov space of the form
\begin{equation}\label{hyp_b}\tag{A}
b \in L^r((t,T\textcolor{black}{)}
,B_{p,q}^\beta(\R^d,\R^d))=:L^r(B_{p,q}^\beta),\qquad \beta \in [-1,0],\, p,q, r\in [1,+\infty].
\end{equation}
{\color{black}We refer to Section \ref{SEC_BESOV} for a precise definition of these function spaces and related properties or notations}.
In our model of interest, $t$ denotes the initial time of the equation, $\xi$ the initial condition which will be assumed to be distributed according to a given probability measure $\mu$ and independent of the symmetric non-degenerate $\alpha $-stable process  $(\mW_s)_{s\ge t} $, with $\alpha\in (1,2] $ (see Assumption \A{UE} below in the non Brownian case $\alpha\in(1,2)$). 

A natural question which arises consists in deriving conditions which relate the \emph{stable exponent} $\alpha$, the \emph{integrability indexes} $r,p,q$, the \emph{regularity index} $\beta$ and the dimension $d$ to obtain either weak or strong well-posedness for the McKean-Vlasov SDE \eqref{main}. In the first work \cite{chau:jabi:meno:22-1}, we developed an approach {\color{black}to answer such question} which is valid for any initial probability law being viewed as an element of a suitable Besov space (see Lemma 5 in \cite{chau:jabi:meno:22-1} and \eqref{lem_proba_in besov} below) {and for any given horizon $T$}.

We basically obtained therein that weak uniqueness holds for \eqref{main} for any initial probability law provided
\begin{equation}\label{cond_gencase}\tag{\textbf{C0}}
 1-\alpha+\frac \alpha{r}+\frac d{p}<\beta, \ {\color{black}\beta\in (-1,0]}.
\end{equation}
We then needed the \textit{strengthened} condition 
\begin{equation}\label{cond_gencase_S}\tag{\textbf{C0}${}_{{\mathbf S}}$}
2-\frac 32\alpha +\frac d{p}+\frac \alpha{r}<\beta, \ {\color{black}\beta\in (-1,0]},
\end{equation}
to guarantee strong well-posedness. Let us point out that in the diffusive case $\alpha=2$, both conditions coincide whereas in the pure jump (strictly stable) case, the condition 
\eqref{cond_gencase_S} is indeed stronger than \eqref{cond_gencase}.

The approach we used in the quoted work to derive those results consisted in considering the Fokker-Planck equation associated with \textcolor{black}{a} suitable mollification of the coefficients in \eqref{main} and in establishing suitable \textit{a priori} estimates that allowed to then pass to the limit. Importantly, see the introduction of \cite{chau:jabi:meno:22-1} and Section \ref{sec_strategy} below, the structure of the non-linearity leads to a quadratic like term in the Fokker-Planck equation. Through convolution estimates in Besov norms (see Lemma 4	in \cite{chau:jabi:meno:22-1}) we used what we called a \textit{dequadrification approach}, which on the one hand allowed to get rid of the aforementioned quadratic dependence and handle any initial probability law, {as well as any time horizon $T$}, but on the other hand did not allow to consider the \textit{critical} thresholds that naturally appear in some related physical models. \textcolor{black}{For e.g. the Burgers, the 2D incompressible Navier-Stokes or the parabolic-elliptic Keller-Segel equations (focusing on the singular part of the associated kernel for this latter), handling the {\color{black}integrability parameters $p$ and $q$ in \eqref{hyp_b} and the} regularity parameter {\color{black}(including $\beta=-1$)} becomes crucial. This could not be done under the previous conditions. We refer to Section \ref{CONNEC_WITH_MODELS} for a thorough discussion related to the indicated models}.

The purpose of the {\color{black}present} 
 work is therefore to quantify how smooth the initial law $\mu $ must be in order \textcolor{black}{to} establish weak/strong well-posedness for \textit{critical drifts}, \textcolor{black}{beyond the thresholds set in \eqref{cond_gencase} and \eqref{cond_gencase_S}, which correspond to concrete but peculiar non-linear models}. This leads to truly handle the quadratic dependence which will lead to \textit{natural} conditions like well-posedness 
in short time or global well-posedness for sufficiently small, in an appropriate Besov norm, initial data. \textcolor{black}{For simplicity reasons we will here focus on the short time setting and discuss in the appendix how the controls {\color{black}can} be extended to handle as well the global well-posedness}. 
Such features are somehow classical in non-linear analysis (see e.g. \cite{Lemarie-16} for the Navier-Stokes equations). We actually manage to provide a \textit{unified} framework to derive\textcolor{black}{/improve}  known results for \textcolor{black}{some} non-linear models \textcolor{black}{established} for $\alpha=2 $ and to extend them to the pure jump case \textcolor{black}{in a systematic way}.

On the other hand, in the special case $\beta=0 $, which roughly says there is \textit{no} smoothness (nor distributional type singularity), we will also investigate how a regularity gain on the initial condition allows to somehow push forward the well-posedness thresholds of the Krylov-R\"ockner type condition, \textcolor{black}{see \cite{kryl:rock:05} for drifts in time-space Lebesgue spaces when $\alpha=2 $ and \cite{xie:zhan:20}, \cite{chau:jabi:meno:22-1} for $\alpha\in (1,2] $, which is precisely given by \eqref{cond_gencase}  taking therein $\beta=0 $}.

\textbf{Organization of the paper.} We state our main results in the next section.
The strategy of the proof is then briefly recalled in Section \ref{sec_strategy}. We state in Section \ref{SEC_BESOV} some useful properties on Besov spaces that will be used for the proof of the main results. 
The framework of Besov interaction kernels allows to revisit the \textcolor{black}{classical} non linear martingale problem approach \textcolor{black}{for McKean-Vlasov SDEs} in a quite systematic way starting from some global density estimates, which are locally stronger than in \cite{chau:jabi:meno:22-1} whenever the measure $\mu$ lies in some appropriate Besov space. Section \ref{ESTI_FOR_FK} is dedicated \textcolor{black}{to those density estimates} focusing on the properties of the associated Fokker-Planck equation, and Section \ref{sec_WP_SDE} to the derivation of the well-posedness results (in a weak and strong sense). As a by-product of the density estimates obtained in Section \ref{ESTI_FOR_FK},  the related propagation of chaos for a suitable particle system \textcolor{black}{could be captured}. This will specifically concern future works. \textcolor{black}{We can mention \cite{hao:jabi:meno:rock:zhan:24} for related results in the kinetic case.}  \textcolor{black}{Section \ref{CONNEC_WITH_MODELS} is dedicated to the connection of our main results with the concrete 
 models mentioned above.
} 

\textcolor{black}{We would also like to mention that after a presentation of X. Zhang at the online seminar ``Non-local operators, probability and singularities'',  \textcolor{black}{a few days prior  the {\color{black}original} preprint of this work {\color{black}w}as released}, we realized that he together with Z. Hao and M. R\"ockner
had a paper in preparation with related results, see \cite{hao:rock:zhan:23}. We exchanged the current versions of our works and can now specify some differences between them. In \cite{hao:rock:zhan:23}, the authors address the (wider) kinetic setting for stable driven McKean-Vlasov SDEs. The approaches to derive and quantify regularization effects are yet rather different, multi-scale Littlewood-Paley analysis in \cite{hao:rock:zhan:23} whereas we focus on global duality techniques for Besov spaces. Eventually, we try to mainly relate our approach to the probabilistic literature/results on those equations and to the extensions we can provide. The paper  \cite{hao:rock:zhan:23} is more connected to PDE results.
}


\subsection{Main results}\,

{
\noindent\textbf{Data.} We recall that we are given a horizon time $T > 0$, an initial time $0< t<T$ and initial law $\mu \in \mathcal P(\R^d)$, a set of parameters $\alpha \in (1,2]$, $\beta \in [-1,0],\, p,q, r\in [1,+\infty]$ and a convolution kernel $b \in L^r(B_{p,q}^\beta)$. Importantly{\color{black}, as mentioned above,} we will {\color{black}often} assume that $T $ is \textit{sufficiently} small. \textcolor{black}{It will be clear from the proofs below how this smallness condition  must be related to the data (i.e. $\mu, \alpha, \|b\|_{L^r(B_{p,q}^\beta)}) $}.
 
We assume, without loss of generality (see {\color{black}E}quation \eqref{lem_proba_in besov} \textcolor{black}{p. \pageref{lem_proba_in besov}} below), that: $\mu\in B^{\beta_0}_{p_0,q_0}$ with
\begin{equation}\label{INIT_DATA}
\beta_0\ge 0, \textcolor{black}{p_0\in [ 1,+\infty]}, \textcolor{black}{q_0\in [1,+\infty]}.\tag{\textbf{C$ {}_{\mathbf I}$}}
\end{equation}
We also define:
\begin{equation}\label{def_zetra}
\zeta_0 := \left(\beta_0 + \frac{d}{p_0'}\right)\left( 1 \wedge \frac{p_0'}{p} \right),
\end{equation}
where $p_0'$ stands for the conjugate exponent of $p_0$ \textcolor{black}{(with the convention $\infty/\infty=1 $)}.
}

\noindent\textbf{Assumptions \A{UE}.} For a point $z\in \R^d$ we write $z=\zeta \rho,\ (\zeta,\rho)\in \mathbb S^{d-1}\times \R_{+}$ its polar coordinates where $\mathbb S^{d-1} $ stands for the unit sphere of $\R^d$. In the pure jump case $\alpha\in \textcolor{black}{(1,2)}$ we assume the following condition holds. The L\'evy measure $\nu$ of $\mW$ is given by the decomposition $\nu(dz)=w(d\zeta)/\rho^{1+\alpha}\I_{\rho>0}$ where $w$ is a symmetric uniformly non-degenerate measure on $\mathbb S^{d-1}$. Namely, $w $ satisfies the  condition:
$$
\kappa^{-1}|\lambda|^\alpha\leq \int_{\mathbb S^{d-1}} |\zeta\cdot\,\lambda|^\alpha\, w(d\zeta)\leq\kappa|\lambda|^\alpha,\,\text{for all }\lambda\in\R^d\,,
$$
for some $\kappa\ge 1$. We point out that this condition allows in particular to consider L\'evy measures that  have a singular spherical part (like e.g. cylindrical processes).

\noindent\textbf{Constraints on the parameters: {\color{black}}assumptions \A{C1} and \A{C2}.} To state our main results we introduce the following assumptions{\color{black}:} \textcolor{black}{Let
	 $b\in L^r((t,T),B_{p,q}^\beta) $ and $\mu \in B_{p_0,q_0}^{\beta_0} $ {\color{black}with $\zeta_0$ as in \eqref{def_zetra}}.}
{\color{black}
\begin{trivlist}
\item[-] We say that condition \textbf{(C1)} holds if  the following conditions are satisfied:
\begin{align}\label{cond_coeff_SPKR}\tag{\textbf{C1}}
\beta \in (-1,0]\quad \ {\rm and}\ \quad 1-\alpha+\frac \alpha r+\textcolor{black}{[ -\beta+\frac d{p}-\zeta_0]_+}<\beta.
\end{align}
\end{trivlist}
\begin{trivlist}
\item[-]We say that condition \textbf{(C2)} holds if  the following conditions are satisfied:
\begin{align}\label{THE_COND_CI}\tag{\textbf{C2}}
\beta =-1,\ \div(b) \in L^{r}(B_{p,q}^{-1})\quad \ {\rm and}\
\quad 1  - \alpha + \frac \alpha r +\textcolor{black}{[1 + \frac dp   -\zeta_0]_+}  < 0. 
\end{align}


\end{trivlist}

\paragraph{Main results}
Our main results read as follows.

\begin{thm}[Weak well-posedness]\label{main_thm_W}
Assume \eqref{INIT_DATA} holds.
Under \A{C1} or \A{C2},  there exists $0< \Tspkr:=\Tspkr(\alpha,d,b,\mu){\color{black}\le T}$ such that for any $S\textcolor{black}{\le} \Tspkr$ the  McKean-Vlasov SDE \eqref{main} admits a weak solution such that its marginal laws $ (\bmu_s^{t,\mu})_{s\in [t,S]}$ have a density $\brho_{t,\mu}(s,\cdot) $ for almost any time $s\in (t,S]$ satisfying
\begin{eqnarray}\label{EST_A_PRIORI_INTRO}
\sup_{s\in (t,S]}[(s-t)^{\theta}\wedge 1]  |\brho_{t,\mu}(s,\cdot)|_{B_{p',1}^{-\beta+ \textcolor{black}{\vartheta \Gamma}}}<+\infty,
\end{eqnarray}
\textcolor{black}{$\vartheta\in (0,1) $} with
\begin{equation*}
\Gamma := {\color{black}\eta\left\{ \alpha-1+\beta-\frac \alpha r + \beta \ind_{\beta >-1}-\frac dp + \zeta_0 \right\},\quad \eta \in (0,1)},
\end{equation*}
{\color{black}$\eta$} being sufficiently close to 1 and 
\begin{equation*}
\theta := \textcolor{black}{\frac 1\alpha \left\{-\beta + \frac dp - \zeta_0 + \left(\frac{1+\eta}{2\eta}\right) \Gamma\right\}\textcolor{black}{>0}}.
\end{equation*}
Moreover, the solution is unique among those satisfying the property \eqref{EST_A_PRIORI_INTRO}.
\\


\end{thm}
}


\begin{thm}[Strong well-posedness]\label{main_thm_S}
Assum{\color{black}ing and that one of the two following conditions hold:}
\begin{itemize}
\item[$\bullet $] \textcolor{black}{If $\beta\in (-1,0] $ and  \eqref{cond_coeff_SPKR} is reinforced into}
\begin{equation}\label{COND_SPKR_STRONG_INTRO}
\Bigg(2-\frac 32\alpha+\frac \alpha r+\Big[-\beta+\frac dp  -\textcolor{black}{\zeta_0} \Big]\Bigg)\vee \Bigg( 1-\alpha+\frac \alpha r+\Big[-\beta+\frac dp-\textcolor{black}{\zeta_0} \Big]_+\Bigg)<\beta\tag{\textbf{C1}${}_{\mathbf S} $},
\end{equation}
\end{itemize}
or
\begin{itemize}
\item[$\bullet $] \textcolor{black}{If $\beta=-1 $ and \eqref{THE_COND_CI} is reinforced into}
\begin{equation}\label{COND_CRITIQUE_INTRO}
\Bigg(2-\frac 32\alpha+\frac \alpha r+\Big[\frac dp  -\zeta_0 \Big]\Bigg)\vee \Bigg(- \alpha + \frac \alpha r +\textcolor{black}{[1 + \frac dp   -\zeta_0]_+}\Bigg)  <\beta=-1\tag{\textbf{C2}${}_{\mathbf S} $},
\end{equation}
\end{itemize}
then{\color{black}, Theorem \ref{main_thm_W} is reinforced from weak to strong well-posedness.}
\end{thm}


\paragraph{Some comments about the results}
 \begin{itemize}
 \item[$\circ$] The setting \eqref{cond_coeff_SPKR} provides an alternative regime  \textcolor{black}{to \eqref{cond_gencase}} {\color{black}from \cite{chau:jabi:meno:22-1}} whose interest consists in specifically specifying how the additional integrability/smoothness of the initial data impacts the previous bounds when $\beta\in (-1,0] $. We see in \eqref{cond_coeff_SPKR} and \eqref{COND_SPKR_STRONG_INTRO} \textcolor{black}{that a key quantity which appears is the quantity $\zeta_0 $ defined in \eqref{def_zetra} that can be viewed as a  \textit{regularity gain factor} associated with the initial condition and the integrability exponent of the singular interaction kernel $b$. Namely, $\zeta_0 $ corresponds to the {\it intrinsic Besov index} $\beta_0+d/p_0'$ of the initial distribution possibly deflated by a scaling factor when $p_0>p' $}.

 \textcolor{black}{We use this quantity instead of the more common differential dimension $\beta_0- d/p_0 $ associated with the initial condition in order to compare more easily the condition \eqref{cond_gencase} of \cite{chau:jabi:meno:22-1} and \eqref{cond_coeff_SPKR}}.
 
 {\color{black}
 \item[$\circ$]  Let us first give some details about the case $\beta=0 $ (Krylov and R\"ockner type framework) in \eqref{cond_coeff_SPKR}. The point is that we actually manage, through the regularity of the initial condition, to weaken the spatial integrability constraint that formerly appeared in \eqref{cond_gencase}. As such, condition \eqref{cond_coeff_SPKR} precisely quantifies this phenomenon (again compared to \eqref{cond_gencase} taking $\beta=0 $).
 It nevertheless appears that we can only benefit from this regularity  up to a factor $d/p$ and obtain at most the constraint $ \alpha/ r<\alpha -1$. In other word, we cannot hope for a better smoothing than $(d/p -\zeta_0)_+$.}

 \item[$\circ$]  If now $\beta\in (-1,0) $, the equilibrium in \eqref{cond_coeff_SPKR} for the positive part depends on the positivity of $-\beta+ d/p-\textcolor{black}{\zeta_0}
 $\footnote{\textcolor{black}{rewriting, \textcolor{black}{when $p'\ge p_0 $}, this quantity as $-\textcolor{black}{(\beta+\frac d{p'})}-(\beta_0\textcolor{black}{-}\frac d{p_0}) $, we see that it actually corresponds to the difference of the differential/dimension indexes, with the terminology of \cite{RunSic-97}, \cite{Sawano-18}, associated respectively with the spaces $B_{p',1}^{-\beta}, B_{p_0,q_0}^{\beta_0} $.}}. The \textcolor{black}{additional term} $-\beta $ here comes from the strategy of the proof we adopt, through the handling of a quadratic term in the related Fokker-Planck equation, see Lemma \ref{lem_unifesti_gencase2_RELOADED} below. It \textcolor{black}{indeed} seems rather natural since, in order to define properly the non-linear drift in \eqref{main}  when $b\in L^r(B_{p,q}^\beta)$, \textcolor{black}{it is necessary} to have estimates on the law \textcolor{black}{$\bmu^{t,\mu} $} in a function space which can be put in duality, \textcolor{black}{at least for the space variable}, with the one of the drift. For technical reasons  the chosen space will be $L^\infty(B_{p',1}^{-\beta}) $ (and even a slightly more demanding one concerning the regularity parameter, \textcolor{black}{see e.g. the estimates \eqref{EST_A_PRIORI_INTRO} in Theorem \ref{main_thm_W}}). The choice of an $L^\infty $ space in time \textcolor{black}{allows} to iterate the estimates in time whereas taking $1 $ \textcolor{black}{for the second Besov integrability index} instead of the more natural $q'$ (standing for the conjugate of $q$) gives more flexibility concerning the product laws in Besov spaces (see Theorem \ref{thm_paraprod_2} below and again the proof of Lemma \ref{lem_unifesti_gencase2_RELOADED}). Anyhow, the corresponding \textit{intrinsic} Besov index \textcolor{black}{for the law} then reads $-\beta+ d/p $. Thus again, if the \textit{regularity gain factor} of the initial condition is greater than \textcolor{black}{the intrinsic Besov index} associated with the function space in which \textcolor{black}{we} will estimate the law of the process, then weak uniqueness holds under the sole condition $\beta>1-\alpha+ \alpha / r $. With respect to the former condition \eqref{cond_gencase}, valid for any initial probability law, this precisely means that the smoothness of $\mu $ makes the spatial integrability condition $d/p$ unnecessary to have weak uniqueness (at least in small time). When 
 $-\beta+d/p-\textcolor{black}{\zeta_0}
 \ge 0 $, i.e. the \textit{intrinsic} Besov index of the law of the process prevails\textcolor{black}{,} the condition for weak uniqueness reads as
 $$ 1-\alpha+\frac \alpha r-\beta+\frac d{p}-\textcolor{black}{\zeta_0}
 <\beta
\iff \frac12 \Big[1-\alpha+\frac \alpha r+\frac d{p}-\textcolor{black}{\zeta_0}
\Big]<\beta.$$
In that case, the threshold for weak existence and uniqueness is relaxed, \textcolor{black}{compared to \eqref{cond_gencase}}, provided that
$$ \frac12 \Big[1-\alpha+\frac \alpha r+\frac d{p}-\textcolor{black}{\zeta_0}\Big] <1-\alpha+\frac \alpha r+\frac d{p}\iff \alpha-1-\frac \alpha r-\frac d{p}<\textcolor{black}{\zeta_0}.$$
 Going to strong uniqueness in this case we see that
 \begin{align*}
\Bigg(2-\frac 32\alpha+\frac \alpha r+\Big[-\beta+\frac dp-\textcolor{black}{\zeta_0} \Big]\Bigg) \vee \Bigg( 1-\alpha+\frac \alpha r+\Big[-\beta+\frac dp-\textcolor{black}{\zeta_0} \Big]_+\Bigg)\\
 =\begin{cases}
2-\frac 32\alpha+\frac \alpha r+\Big[-\beta+\frac dp-\textcolor{black}{\zeta_0}\Big],\ \\
\hspace*{.5cm}{\rm if}\ \Big[-\beta+\frac dp-\textcolor{black}{\zeta_0} \Big]_+\neq 0\ {\rm or }\ \textcolor{black}{-\beta+\frac dp \le \zeta_0\le -\beta+\frac dp+1-\frac \alpha 2}
, \\
 \textcolor{black}{1-\alpha+\frac \alpha r,\ {\rm if}\ \zeta_0\ge -\beta+\frac dp+1-\frac \alpha 2}
 .
 \end{cases}
  \end{align*}
  Hence, when the initial condition is regular enough,  weak and strong uniqueness are implied by the sole condition
  $\beta>1-\alpha+ \alpha/ r $. In the other cases we see that the condition for strong uniqueness \eqref{cond_gencase_S} is relaxed as soon as $-\beta-\textcolor{black}{\zeta_0}
  <0 $. 

\item[\textcolor{black}{$\circ$}] Let us eventually turn to $\beta=-1$. From the additional, and rather strong structure condition that $\div (b)\in L^r(B_{p,q}^\beta) $ we see that the term
 $-\beta=1 $ disappears in the \textcolor{black}{right} hand side of \eqref{THE_COND_CI} and \textcolor{black}{the second term in the left hand side of \eqref{COND_CRITIQUE_INTRO}}. This is precisely because the structure condition allows to perform an integration by parts in the analysis \textcolor{black}{of} Lemma \ref{lem_unifesti_gencase2_RELOADED} which somehow leads for the l.h.s. to the case $\beta=0 $ under \eqref{cond_coeff_SPKR}. This assumption is strong but can be verified in many settings, one can e.g. think about fluid dynamics problems which involve divergence free drifts, or the Keller-Segel model discussed below. In connection with divergence free drift\textcolor{black}{s} we can mention the work \cite{zhan:zhao:21} by Zhang and Zhao who obtained under this additional condition existence for a \textit{linear} SDE beyond the Krylov and R\"ockner condition.
 
 {\color{black}An important class of kernels $b$ entering the setting $\beta=-1$ is given by: 
 	 $$
 	b=b_1+b_2,\qquad b_1\in L^r((t,T),B^{0}_{p,q}(\mathbb R^d;\mathbb R^d)),\quad b_2=\left(\sum_{j=1}^d \partial_{x_j}q_{1,j},\cdots,\sum_{j=1}^d \partial_{x_j}q_{d,j}\right),  
 	$$
 	where the formed matrix field $q=\{q_{i,j}\}_{1\le i,j\le d}$ is assumed to  anti-symmetric ($q_{j,i}=-q_{i,j}$) and each element $q_{i,j} $ to belong to $L^r((t,T),B^{0}_{p,q}(\mathbb R^d;\mathbb R))$. As such,  according to the Besov embedding $B^{0}_{p,q}\hookrightarrow B^{-1}_{p,q}$ (see \eqref{EMBEDDING}) and the lifting property (\eqref{LO}), $b_1,b_2$ both lies in $L^r((t,T),B^{-1}_{p,q}(\mathbb R^d;\mathbb R^d))$, and, in the sense of distributions, $\div(b_2)=0$, and $\div {\color{black}(}b {\color{black})}=\div(b_1)\in L^r(B^{-1}_{p,q})$. 
 	 {\color{black}T}he above decomposition can be understood as a form of Helmholtz decomposition where $b_2$ is the singular divergence free part (carrying the $\beta=-1$ irregularity) and $b_1$ is the "regular" gradient part viewed as $b_1=\nabla g$ (as $\div(b_1(t,\cdot))\in B^{-1}_{p,q}$ with $\triangle g(t,\cdot)\in B^{-1}_{p,q}$ $\Rightarrow$ $ g\in(\triangle)^{-1}B^{-1}_{p,q}\simeq B^{1}_{p,q}$, and so $\nabla g\in\nabla B^{1}_{p,q}\simeq B^{0}_{p,q}$). As it will be discussed in Section  \ref{CONNEC_WITH_MODELS}, this view is particularly natural for the two-dimensional vortex equation.}

\item[\textcolor{black}{$\circ$}]We insist that all the above discussion concerning  the relaxation of the former condition \eqref{cond_gencase} is valid in \textit{small} time or could a priori be extended to an arbitrary final time under some appropriate conditions \textcolor{black}{(see Appendix \ref{EXT_TO_LONG_TIME} for related discussions)}.
\textcolor{black}{Small time is often a \textit{natural} constraint  with non-linear dynamics, we can e.g. to the Navier-Stokes \cite{Lemarie-16} or the Keller-Segel equation in dimension two\cite{bile:19}.This is a specificity of the current approach that differs}  w.r.t. to the one in \cite{chau:jabi:meno:22-1} which yields well-posedness for any initial condition and fixed final \textcolor{black}{time horizon}. However, this highly depends on the current approach, which consists in handling the quadratic term deriving from the related Fokker-Planck equations and can be seen as the price to pay to quantify the \textit{global} impact of a smoother initial condition, which could have only been quantified in small time in \cite{chau:jabi:meno:22-1}.
 \end{itemize}


\subsection{Mollified SDE and strategy of \textcolor{black}{the} proof}\label{sec_strategy}
\textcolor{black}{The principal steps of our procedure mainly follow those of our previous work \cite{chau:jabi:meno:22-1}. We briefly recall it for the sake of completeness}.
The  \textcolor{black}{strategy} consists first in  establishing the existence of a solution to \eqref{main} in terms of a nonlinear martingale problem, through a mollification of the coefficient and a stability argument (the solutions of the non-linear equations with mollified coefficients \textcolor{black}{form} a  Cauchy sequence in a suitable function space). This actually  allows to obtain the well-posedness of  \eqref{main} directly from the construction of its time-marginal distributions as solution to the nonlinear Fokker-Planck equation related to \eqref{main}.
 
The  martingale problem approach to the well-posedness  of McKean-Vlasov SDEs has been successfully used over the past to handle a wide rage of settings from smooth or ``quasi''-smooth to singular interacting kernels. We refer to the  \textcolor{black}{papers} \cite{Oelschlager-84}, \cite{MelRoe-87} and \cite{Jourdain-97},  and again to \cite{Dawson-83}, \cite{Sznitman-86}, \cite{Meleard-00}, \cite{font:03} \textcolor{black}{- among others -} and references therein for more particular cases. \textcolor{black}{This approach has been notably successful to validate numerical particle methods}. \textcolor{black}{In connection with the current paper, we can as well mention the work Issoglio and Russo \cite{isso:russ:23}, who develop the martingale approach, together with a close functional framework (Besov spaces, paraproduct), for non-linear dynamics involving a singular drift in the non convolutional case, which in turn does not provide similar regularization properties,  and the work {Olivera} \textit{et al.} \cite{oliv:rich:toma:23} addressing well-posedness and the particle approximation of non linear Fokker-Planck equation in a Brownian setting with a convolution kernel in Lebesgue spaces.}

For convenience we now introduce for a drift $b$ satisfying condition {\color{black} \A{C1} or \A{C2}} and any measure $\nu$ for which this is meaningful the notation:
 $$\mathcal B_{\nu}(s,\cdot):=b(s,\cdot)\star \nu(\cdot),$$
\textcolor{black}{where $\star$ denotes the spatial convolution}.
 For all $\varepsilon>0 $ consider a time-space mollified drift $b^\varepsilon $, i.e. $b^\varepsilon$ is smooth and bounded in time and space (see Proposition \ref{PROP_APPROX} below for precise properties related to $b^\varepsilon$ and the proof of this result in \cite{chau:jabi:meno:22-1}).
  We now write similarly,
 \begin{equation}\label{DRIFT_NON_LIN_MOLL}
 \mathcal B_{\nu}^\varepsilon(s,\cdot):=b^\varepsilon(s,\cdot)\star \nu(\cdot),
 \end{equation}
 which is  well defined for any $\nu \in \mathcal P(\R^d) $ since $b^\varepsilon $ is smooth and bounded.

The \textit{smoothened} version of \eqref{main} is defined by the family of McKean-Vlasov SDEs
 \begin{equation}\label{main_smoothed}
X_s^{\varepsilon,t,\mu} = \xi + \int_t^s \mathcal B^\varepsilon_{\textcolor{black}{\bmu^{\varepsilon,t,\mu}_r}}(r,X_r^{\varepsilon,t,\mu}) dr + \mW_s-\mW_t,\quad \textcolor{black}{\bmu_s^{\varepsilon,t,\mu}}=\text{Law}(X_s^{\varepsilon,t,\mu}),\, {\color{black} 0\le t \le s < T,\, \varepsilon >0.}
\end{equation}

{\color{black}

\textcolor{black}{For every $\varepsilon >0$, $\alpha\in (1,2] $,
 the SDE \eqref{main_smoothed} with mollified (i.e. smooth and bounded) interaction kernel $b^\varepsilon$ admits a unique weak solution whose time marginal distributions $(\bmu_s^{\varepsilon,t,\mu})_{s\in (s,T]} $ are absolutely continuous \textcolor{black}{w.r.t. the Lebesgue measure of $\R^d $ with density $(\brho_{t,\mu}^\varepsilon (s,\cdot))_{s\in (s,T]}$} (see \cite{chau:jabi:meno:22-1} Section 1.3 for details)}.
Namely, {\color{black} for any $S<T$,}
\begin{equation}\label{relation_density}
\forall A \in {\color{black}\mathscr B([t,S])} \otimes \mathscr B(\R^d),\quad \textcolor{black}{\bmu}^{\varepsilon,t,\mu}(A) %
\textcolor{black}{:=\int_A {\bmu}_r^{\varepsilon,t,\mu}(dy)=}\int_A \brho_{t,\mu}^\varepsilon(r,y) drdy.
\end{equation}

As a consequence of  It\^o's formula, the following Duhamel representation holds: for each $\varepsilon>0$, $\brho^\varepsilon_{t,\mu}(s,\cdot) $ satisfies for  all $s\in {\color{black} (t,S]} $ and all $(x,y)\in (\R^d)^2 $:

\begin{eqnarray}\label{main_MOLL}
\brho^\varepsilon_{t,\mu}(s,y) = p^\alpha_{s-t}\star \mu(y) - \int_t^s dv \Big[\{\mathcal B_{\brho^\varepsilon_{t,\mu}}^\varepsilon(v,\cdot) \brho^\varepsilon_{t,\mu}(v,\cdot)\} \star \nabla p_{s-v}^{\alpha}\Big] (y),
\end{eqnarray}
where $p^\alpha$ stands for the density of the driving process $\mW$, and with a slight abuse of notation w.r.t. \eqref{DRIFT_NON_LIN_MOLL}, $\mathcal B_{\brho^\varepsilon_{t,\mu}}^\varepsilon(v,\cdot)=[b^\varepsilon (v,\cdot)\star  \mathbf \brho^\varepsilon_{t,\mu}(v,\cdot)] $.

\textcolor{black}{Equivalently, see Lemma 3 in \cite{chau:jabi:meno:22-1}}, for $k \ge 1$, $\brho^{\varepsilon_k}_{t,\mu} $ is a \textit{mild} solution of the equation:
\begin{equation}
\label{PDE_EPS}
\begin{cases}
\partial_s\brho^{\varepsilon_k}_{t,\mu}(s,y)+ \div ( \mathcal B_{\brho^{\varepsilon_k}_{t,\mu}}^{\varepsilon_k}(s,y) \brho^{\varepsilon_k}_{t,\mu}(s,y))-L^\alpha \brho^{\varepsilon_k}_{t,\mu}(s,y)=0, \\
\brho^{\varepsilon_k}_{t,\mu}(t,\cdot)=\mu,
\end{cases}
\end{equation}
where 
$L^\alpha$ is the generator of the driving process.

Provided that $\brho^{\varepsilon_{\color{black}k}}_{t,\mu}(s,\cdot)$  admits a limit $\brho_{t,\mu}(s,\cdot)$ in some appropriate function space which precisely allows to take the limit in the Duhamel formulation \eqref{main_MOLL} we derive that the limit satisfies
\begin{eqnarray}\label{main_DUHAMEL}
\brho_{t,\mu}(s,y) = p^\alpha_{s-t}\star \mu (y)  -\int_t^s dv \Big[\{\brho_{t,\mu}(v,\cdot)\mathcal B_{\brho_{t,\mu}}(v,\cdot)\} \star \nabla p_{s-v}^{\alpha}\Big] (y),
\end{eqnarray}
{\color{black} for any $t < s \le S$.} In other words  $\brho_{t,\mu}(s,y)\,dy$ is a (mild) solution to the nonlinear Fokker-Planck equation related to \eqref{main}:
\begin{equation}
\label{NL_PDE_FK}
\begin{cases}
\partial_s\brho_{t,\mu}(s,y)+\div( \brho_{t,\mu}(s,y)\mathcal B_{\brho_{t,\mu}}(s,y))-L^\alpha \brho_{t,\mu}(s,y)=0\,\ (s,y)\in (t, S]\times \R^d,\\
\brho_{t,\mu}(t,\cdot)=\mu(\cdot).
\end{cases}
\end{equation}
\textcolor{black}{In that case, $\brho_{t,\mu}(s,\cdot) $ is also a solution to \eqref{NL_PDE_FK}  in the sense of  distribution, i.e. 
for any $\varphi \in \mathcal D([t,S)\times \R^d)$,
\begin{equation}\label{DEF_DISTR} \tag{$\mathscr D $}
\int_t^S \int \brho_{t,\mu}(s,x)\big(\partial_s\varphi+\mathcal B_{\brho_{t,\mu}}\cdot \nabla \varphi+\textcolor{black}{L^\alpha}(\varphi)\big)(s,x)\,ds\,dx=-	\int \varphi(t,x)\mu(dx),
\end{equation}
\textcolor{black}{where we used here as well that, from the symmetry of the L\'evy measure $\nu $ of $\mathcal W $, the operator $L^\alpha $ is self-adjoint}. 
It will actually be shown in Lemma \ref{lem_ex_duha_SPKRcase_BIS} below that any limit point of $\brho^{\varepsilon}_{t,\mu}(s,\cdot) $ is actually a distributional solution to \eqref{NL_PDE_FK} which  also satisfies the mild formulation \eqref{main_DUHAMEL}}.

This provides a method to construct a solution to \eqref{main} identifying the limit of the martingale problem related to \eqref{main_smoothed}. 
 More precisely, from the solution to \eqref{main_smoothed}, one can consider the probability measure $\textcolor{black}{\mathbf P}_{\textcolor{black}{t}}^\varepsilon$ on the space \textcolor{black}{$\Omega_\alpha $ (corresponding to the space of c\`adl\`ag functions ${\color{black}\mathbb D([t,S];\R^d)}$ if $\alpha\in (1,2) $ and to the space of continuous functions ${\color{black}\mathcal C([t,S];\R^d)} $ if $\alpha=2$)} such that, for $x(s)$, {\color{black}$t\le s\le S$}, the canonical process on $
 \textcolor{black}{\Omega_\alpha}$, and for $\textcolor{black}{{\mathbf P}}^\varepsilon_t(s,dx):=\textcolor{black}{\mathbf P}^\varepsilon_t(x(s)\in dx)$ the family of probability measures induced by $x(s)$, we have: 
 $\textcolor{blue}{\mathbf P}^\varepsilon_t({\color{black}t,\cdot})$
   is equal to $\mu$ a.s. 
 and for all function $\phi$ twice continuously differentiable on $\R^d$, with bounded derivatives at all order, the process
\[
\phi(x(s))-\phi(x(t))-\int_t^s\left\{\mathcal B_{{\textcolor{black}{\mathbf P}}^\varepsilon_t}(v,x(v))\cdot \nabla\phi(x(v))+L^\alpha(\phi(x(v)))\right\}\,dv , {\color{black}\,t\le s\le S},
\]
is a martingale. Provided, again, that the time-marginal distributions $\textcolor{black}{\mathbf P}^\varepsilon_t(s,dx)=\brho_{t,\mu}^\varepsilon(s,x)\,dx$ lie in a\textcolor{black}{n} appropriate space to ensure that $\textcolor{black}{\mathbf P}_t^\varepsilon$ is compact in $\mathcal P(\textcolor{black}{\Omega_\alpha})$ any corresponding limit along a converging subsequence defines naturally a solution to the (nonlinear) martingale problem related to \eqref{main}. {\color{black} From the well-posedness of the limit Fokker-Planck equation one eventually derives uniqueness results for the time-marginal distributions giving in turn the uniqueness of \eqref{main}.}
We also emphasize that we here consider the \textit{classical}  martingale problem, i.e. the integral of the non-linear drift with singular kernel is well defined (through the estimates established in Section \ref{ESTI_FOR_FK}.)
}

\section{Notation and reminders on some fundamental of Besov spaces} 
\label{SEC_BESOV}

In this paragraph, we set definitions/notations and remind technical preliminaries - reviewed or directly established in \cite{chau:jabi:meno:22-1} - that will be used throughout the present paper.
 
 From here on, we denote  by $B^\gamma_{\ell,m}, \ell,m,\gamma\in \R$  the Besov space with regularity index $\gamma$ and integrability \textcolor{black}{parameters} $ \ell,m$ (\textcolor{black}{see e.g.  \cite{athr:butk:mytn:20}, \cite{CdRM-20}, \cite{Lemarie-02}} for some related applications and the dedicated monograph \cite{Triebel-83a} by Triebel). 
 We use the thermic characterization  through the isotropic stable heat kernel for its definition (see e.g. Section 2.6.4 in \cite{Triebel-83a}). 
\textcolor{black}{Namely, denoting by ${\mathcal S'}( \R^d) $ the dual space of the Schwartz class ${\mathcal S}( \R^d) $},
  \[
 B^\gamma_{\ell,m}=\left\{f\in\textcolor{black}{\mathcal S'}( \R^d)\,:\,| f|_{B^\gamma_{\ell,m}}:=|\mathcal F^{-1}(\phi\mathcal F(f))|_{\textcolor{black}{L^\ell}}+\mathcal T_{\ell,m}^\gamma(f)<\infty\right\},
 \]
 \begin{align}
\mathcal T_{\ell,m}^\gamma(f):=&\left\{
\begin{aligned}
&
\left(\int_0^1\,\frac{dv}{v}v^{(n-\gamma/\textcolor{black}{\alpha})m}|\partial^n_v \tilde p_\alpha(v,\cdot)*f|^m_{L^\ell}\right)^{\frac 1m}\,\text{for}\,1\le m<\infty,\\
&\sup_{v\in(0,1]}\left\{v^{(n-\gamma/\textcolor{black}{\alpha})}|\partial^n_v\tilde p_\alpha(v,\cdot)*f|_{L^\ell}\right\}\,\text{for}\,m=\infty,
\end{aligned}
\right. 
\label{HEAT_CAR}
\end{align}
  $n$ being any non-negative integer (strictly) greater than $\gamma/\textcolor{black}{\alpha}$, the function $\phi$ being a $\mathcal C^\infty_0$-function \textcolor{black}{(infinitely differentiable function with compact support)} such that $\phi(0)\neq 0$, and \textcolor{black}{$\tilde p_\alpha(v,\cdot)$ denoting the density function at time $v$ of the $d$-dimensional isotropic stable process}. We can refer e.g. to the discussion in Section 2.6.4. and the general characterization in Section 2.5.1 of \cite{Triebel-83a} for this characterization of Besov spaces. It will be in particular useful to derive some heat kernel estimates on the density of the driving noise in Appendix \ref{HK_APPENDIX}. This is actually the only point in the work for which we explicitly use the 
  thermic characterization since otherwise we only exploit inequalities in Besov norms to establish our estimates. {\color{black}We refer to the dedicated paragraph at the bottom of p. 19 in \cite{CdRM-20} for a related discussion on the choice of the parameter of the stable heat kernel chose in the thermic characterization.}

We now list some properties that we will throughly exploit to establish the density estimates on the solution of \eqref{main}.

$\bullet$ \emph{Embeddings.} 
\begin{itemize}
\item[(i)] Between Lebesgue and $B^0_{\ell,m}$-spaces \textcolor{black}{(\cite[Prop. 2.1]{Sawano-18})}:
\begin{equation}
\forall 1 \le \ell \le \infty,\qquad B^0_{\ell,1}\hookrightarrow L^\ell \hookrightarrow B_{\ell,\infty}^0.\label{EMBEDDING}\tag{$\mathbf E_1$}
\end{equation}
\item[(ii)] Between Besov spaces \textcolor{black}{(see (1.1) in \cite{trie:14} and Proposition 2.2 in \cite{Sawano-18})}: For all $p_1,p_2,q_1,q_2 \in[1,\infty]$ such that $q_1\le q_2$, $p_1\le  p_2$ and $s_1-d/p_1\geq s_2-d/p_2$, 
\begin{equation}\label{BesovEmbedding}\tag{$\mathbf E_2 $}
B^{s_1}_{p_1,q_1}\hookrightarrow B^{s_2}_{p_2,q_2}.  
\end{equation}
\item[(iii)] \textcolor{black}{Inclusion of Probability measures into Besov spaces (e.g. Lemma 5 in \cite{chau:jabi:meno:22-1})}. {\color{black}For  $\mathcal P(\R^d)$, the space of probability measures on $\mathbb R^d$,}
\begin{align}
	&\mathcal P(\R^d)\subset \cap_{\ell \ge 1}B^{-d/\ell'}_{\ell,\infty},\notag\\
	&\mathcal P(\R^d)\subset \cap_{\ell \ge 1}B^{-d/\ell'-\epsilon}_{\ell,m},\:\:\epsilon>0, m\in[1,\infty)\,\text{where}\,\ell^{-1}+(\ell')^{-1}=1.\label{lem_proba_in besov}\tag{$\mathbf E_3$}
\end{align}
\end{itemize}
$\bullet$ \emph{Young/Convolution inequality} \textcolor{black}{(\cite[Theorem 3]{Burenkov-90})}. Let  $\gamma \in \R$, $\ell$ and $m$ in $[1,+\infty]$. Then for any $\delta\in\mathbb R$, $\ell_1,\ell_2\in [1,\infty]$ such that $1+\ell^{-1} = \ell_1^{-1} + \ell_2^{\textcolor{black}{-1}}$ and $m_1,m_2\in (0,\infty]$ such that $m_1^{-1} \ge \max\{m^{-1}-m_2^{-1},0\}$
\begin{equation}\label{YOUNG}
| f\star g|_{B_{\ell,m}^\gamma} \le \cv | f |_{B^{\gamma - \delta}_{\ell_1,m_1}} | g |_{B^{\delta}_{\ell_2,m_2}}\textcolor{black}{,}\tag{$\mathbf Y $}
\end{equation}
for $\cv$ a universal constant depending only on $d$.  \textcolor{black}{We can also mention \cite{KuhSch-21} for recent  convolution inequalities for Besov and Triebel-Lizorkin spaces}.

$\bullet$ \emph{Besov norm of heat kernel}. There exists $\ch:=C(\alpha,\textcolor{black}{\ell},m,\gamma,d)$ s.t. for all multi-index $\textcolor{black}{\mathbf a \in \mathbb N^{d}}$ 
with $ |\mathbf a|\le 1 $, {\color{black} $\gamma \in \R$, with $\gamma\neq -d(1-1/\ell) $ for $\gamma<0 $} and $0<v<s<\infty$: {\color{black}
\begin{equation}\label{SING_STABLE_HK}
\big| \p^{\mathbf a} p_{s-v}^{\alpha} \big|_{B^{\gamma}_{\ell,m}}\le \frac{\ch}{[(s-v)\wedge 1]^{\textcolor{black}{\frac {|\mathbf a|}\alpha+[\frac \gamma\alpha+\frac d{\alpha}(1-\frac 1\ell)]_+}}[(s-v)\vee 1]^{\frac d{\alpha}(1-\frac 1\ell) +\frac {|\mathbf a|}\alpha}}. \tag{$\mathbf H\mathbf K $}
\end{equation}
\textcolor{black}{The proof for $\gamma\ge 0 $ can e.g. be found in \cite[Lemma 11, 12 and (3.19)]{CdRM-20}. For $\gamma<0 $ we refer to Appendix \ref{HK_APPENDIX}.}\\
}
\noindent
$\bullet$ \emph{Lift operator.} For any $\gamma \in \R$, $\ell,m \in [1,\infty]$, there exists $\cl>0$ such that
\begin{equation}\label{LO}
|\nabla f|_{B^{\gamma-1}_{\ell,m}} \le \cl |f|_{B_{\ell,m}^{\gamma}}. \tag{$\mathbf L$}
\end{equation}

%
$\bullet$ Smooth approximation of the interaction kernel and associated uniform-control properties:
\begin{prop}\cite[Proposition 2]{chau:jabi:meno:22-1}\label{PROP_APPROX}
Let  $b\in L^r((t,T],B_{p,q}^\beta)$ and {\color{black} $\beta\in [-1,0] $}, $1\le p,q \le \infty $.
There exists a sequence of \textcolor{black}{time-space} smooth bounded functions $(b^\varepsilon)_{\varepsilon >0} $ s.t.
\begin{equation}\label{SMOOTH_APP_GEN}
|b-b^\varepsilon|_{L^{\bar r}((t,T],B_{p,q}^{\tilde \beta})} \underset{\varepsilon \rightarrow 0}{\longrightarrow} 0,\quad \forall \tilde \beta<\beta,
\end{equation}
with $\bar r=r $ if $r<+\infty $ and for any $\bar r<+\infty $ if $r=+\infty$. Moreover, there exists $ \cc\ge 1$, $\displaystyle \sup_{\varepsilon>0}|b^\varepsilon|_{L^{\bar r}((t,T],B_{p,q}^{\beta})}\le \textcolor{black}{\cc} |b|_{L^{\bar r}((t,T],B_{p,q}^{\beta})}$.

If $p,q,r<+\infty$ it then also holds, see e.g. \cite{Lemarie-02}
, that 
\begin{equation}\label{SMOOTH_APPR_FINITE}
|b-b^\varepsilon|_{L^{ r}((t,T],B_{p,q}^{\beta})} \underset{\varepsilon \rightarrow 0}{\longrightarrow} 0.
\end{equation}
\end{prop}
{\color{black} This proposition has been proved in \cite{chau:jabi:meno:22-1} for $\beta$ in $(-1,0]$ ($-1$ being excluded) as this was the set of assumptions therein. However, it can be readily checked that the proof still work with $\beta=-1$, as for this regularity index, only the difference $ \beta-\tilde \beta >0$ really matters.\\}

$\bullet$ {\color{black}\textit{Products of Besov spaces and related embeddings} 

 Let $f\in B^{\gamma_1}_{\ell_1,m_1}$, $g\in B^{\gamma_2}_{\ell_2,m_2}$.
We recall the following definition for the product, see \cite{RunSic-97}, Sections 4.2 and 4.3, 
\[
f\cdot g=\lim_{j\rightarrow \infty} \mathcal F^{-1}\Big(\textcolor{black}{\psi}(2^{-j}\xi) \mathcal F(f)(\xi)\Big) \times  \mathcal F^{-1}\Big(\textcolor{black}{\psi}(2^{-j}\xi)\mathcal F(g)(\xi)\Big),
\]  
\textcolor{black}{where $\psi\in \mathcal C_0^\infty $ s.t. $\psi(x)=\begin{cases}1,\ |x|\le 1\\
0,\ |x|\ge \frac 32\end{cases} $},
whenever the limit exists in $\mathcal S'$. Then, from Theorem $2$, p. $177$ in \cite{RunSic-97}, the following embeddings  hold:

\begin{thm}\label{thm_paraprod_2}

Let $\gamma>0 $, $\ell,\ell_1,\ell_2\in [1,+\infty] $ s.t. $\ell^{-1}=(\ell_1)^{-1}+(\ell_2)^{-1} $. Then
\begin{equation}
\label{PROD1}
B^{\gamma}_{\ell_1,\infty}\cdot B^{\gamma}_{\ell_2,1}\hookrightarrow B^{\gamma}_{\ell,\infty} \tag{$\textbf{Prod}$}.
\end{equation}
In particular there exists $C$ s.t. for all $f\in B_{\ell_1,\infty}^{\gamma}$, $g\in B_{\ell_2,1}^\gamma $,
 \begin{eqnarray*}
 \Big|f \cdot  g \Big|_{B^{\gamma}_{\ell,\infty}} 
\le C\Big|f\Big|_{B^{ \gamma}_{ \ell_1,\infty}} \Big|g\Big|_{B^{ \gamma}_{ \ell_2,\textcolor{black}{1}}}.
 \end{eqnarray*}
 \end{thm}
 The previous theorem will be extensively used for estimates on \eqref{main_MOLL} later on. 

}

{\color{black}
\paragraph{Weighted Lebesgue-Besov spaces} As a preliminary step before stating our main results, we introduce a characteristic class of weighted iterated 
\textcolor{black}{\textit{Lebesgue-Besov}} function spaces. The solutions to the Fokker-Planck equation related to \eqref{main} and some associated \textit{a priori} estimates will be sought in those, Bochner type, spaces. Introduce for 
$\sc, \ell,m\in [1,+\infty], \gamma, \theta\in \R $, {\color{black}$t\le S\textcolor{black}{\le} T$,}
\begin{align*}
&L_{\w_\theta}^{\sc}((t,S],B_{\ell,m}^\gamma):=\Bigg\{f:s\in[t,S]\mapsto f(s,\cdot)\in B^{ \gamma}_{ \ell, m}\,\text{measurable, s.t.}\, \int_t^S  | f(s,\cdot)|_{B^{ \gamma}_{ \ell, m}}^{\textcolor{black}{\sc}} 
(s-t)^\theta ds<+\infty \Bigg\},
\end{align*}
if $\sc<+\infty $ and
\begin{align}
&L_{{\text w}_\theta}^\infty((t,S],B_{\ell,m}^\gamma)\notag\\
:=&\Bigg\{f:s\in[t,S]\mapsto f(s,\cdot)\in B^{ \gamma}_{ \ell, m}\,\text{measurable, s.t.}\, \text{ess sup}_{s\in(t,S]}  \Big(
(s-t)^\theta| f(s,\cdot)|_{B^{ \gamma}_{ \ell, m}}\Big)<+\infty \Bigg\}.\label{DEF_NORM_W}
\end{align}
}
\begin{rem}
We emphasize that one of the main differences with  \cite{chau:jabi:meno:22-1} in the definition of the weighted spaces is that we here consider a weight which involves the initial time $t$ (or \textit{backward} weight), whereas we previously considered in   \cite{chau:jabi:meno:22-1} a weight involving the final time (or \textit{forward} weight). This is mainly due to the fact that we here want to absorb the potentially insufficient  smoothing effects of the initial measure in order to address the current critical case. In \cite{chau:jabi:meno:22-1}, the forward weights were chosen in order to equilibrate the higher singularities  of the gradient of the heat kernel in the Duhamel representation. The approach will be here different since we will not rely exclusively on the heat kernel to absorb the singularities induced by the regularity estimates, \textcolor{black}{but also on the initial condition}. 
\end{rem}


Endowed with the metric
\begin{equation*}
|f|_{L_{{\w}_{\theta}}^{\textcolor{black}{\sc}}((t,S],B^\gamma_{\ell,m})}=\Bigg( \int_t^S   | f(s,\cdot)|_{B^{\gamma}_{\ell,m}}^{\textcolor{black}{\sc}}(s-t)^\theta ds\Bigg)^{\frac 1{\textcolor{black}{\sc}}} ,
\end{equation*}
with the usual modification if $ \textcolor{black}{\sc}=+\infty $, and recalling that $(B^{\gamma}_{\ell,m},|\cdot|_{B^{\gamma}_{\ell,m}})$ is a Banach space, the normed space $(L_{{\w}_{\theta}}^{\textcolor{black}{\sc}}((t,S],B_{\ell,m}^\gamma) ,|\cdot|_{L_{{\w}_{\theta}}^{\textcolor{black}{\sc}}((t,S],B^\gamma_{\ell,m})})$ is also a Banach space (see e.g. \cite[Chapter 1]{HyNeVeWe-16}). In the case $\theta=0$, $L_{{\w}_{\theta}}^{\textcolor{black}{\sc}}((t,S],B^\gamma_{\ell,m})$ reduces to
$L^{\textcolor{black}{\sc}}\big((t,S],B^{ \gamma}_{ \ell, m})$.

\section{Estimates on the Fokker-Planck equation \eqref{NL_PDE_FK}}
\label{ESTI_FOR_FK}
%
\begin{prop}\label{WP_FK} Let $\beta \in [-1,0]$. Assume that the \emph{parameters} are such that \eqref{cond_coeff_SPKR} holds  for $\beta\in (-1,0] $ (resp. \eqref{THE_COND_CI} if $\beta=-1$). Then, for any  $(t,\mu)$  in $[0,\textcolor{black}{T)} \times \mathcal P(\R^d)\cap B_{p_0,q_0}^{\beta_0}$, the non-linear Fokker-Planck equation \eqref{NL_PDE_FK} admits a solution  
which is unique among all the distributional solutions lying in $ L_{{\rm w}_\theta}^{\infty}((t,S],B^{-\beta+\vartheta \Gamma}_{p',\textcolor{black}{1}}),\ \textcolor{black}{S\le\mathcal T_1}$ with $\textcolor{black}{\mathcal T_1\le T}, \theta,\Gamma $ as in Theorem \ref{main_thm_W} and $\vartheta\in (0,1)$.

Moreover, \textcolor{black}{for all $s\in [t,S], \brho_{t,\mu}(s,\cdot)$ belongs to $\mathcal P(\R^d)$}. \textcolor{black}{Eventually,  for a.e. $s$ in $(t,\textcolor{black}{S}]$, $\brho_{t,\mu}(s,\cdot)$ is absolutely continuous w.r.t. the Lebesgue measure  and satisfies the Duhamel representation \eqref{main_DUHAMEL}}.
\end{prop}

 Let us sum up the strategy to derive weak and strong well posedness in our non-linear setting under \eqref{cond_gencase} and \eqref{cond_coeff_SPKR}, \eqref{THE_COND_CI} respectively \textcolor{black}{and in \textit{short time}}.
 In any case we need to establish \textit{a priori} estimates on the Fokker-Planck equation \eqref{PDE_EPS} associated with mollified \textcolor{black}{kernels} through its Duhamel representation \eqref{main_MOLL}.
 \begin{itemize}
\item[1.] In \cite{chau:jabi:meno:22-1}, under \eqref{cond_gencase}, we used the so-called \textit{de-quadrification} technique. In this case we cannot in some sense benefit from the smoothness of the initial condition since the whole regularity index associated with the Besov norm for which we are estimating the (mollified) density, is felt by  the gradient of the stable heat kernel \textcolor{black}{(see as well the related discussion in the proof of Lemma \ref{lem_unifesti_gencase2_RELOADED} below)}.
\item[2.] In the current work, under \eqref{cond_coeff_SPKR}-\eqref{THE_COND_CI}, we use techniques that are more \textit{common} in non-linear analysis and make a quadratic term appear. We will see below that in that case this approach allows to have an extra-margin,  what we already called \textcolor{black}{in the comments following Theorems \ref{main_thm_W} and \ref{main_thm_S}} the  \textcolor{black}{regularity gain factor} associated with the initial condition $\mu \in B_{p_0,q_0}^{\beta_0} $ \textcolor{black}{and the kernel integrability index $p$}, with respect to the former threshold appearing in \eqref{cond_gencase}.
We point out that, for $\beta=-1 $ we anyhow require some additional smoothness properties for the drift to handle this \textit{critical case}.
 \end{itemize}
To derive weak uniqueness, our approach consists in exploiting appropriate estimates on the density so that we can prove that the drift
$$\mathcal B_{\textcolor{black}{\brho}_{t,\mu}}(s,\cdot):=\int_{\R^d} b(t,y) \brho_{t,\mu}(r,\textcolor{black}{\cdot-}y)dy $$
belongs to \textcolor{black}{the time-space Lebesgue space} $L^{\sc}-L^\infty$ for an index $\sc$ which then allows to enter the framework of \cite{CdRM-20} where, in the linear setting, a parabolic bootstrap result was established for singular drifts.

We insist on the fact that appealing to product rules in Besov spaces precisely allows to derive the highest regularity order, with respect to what could have e.g. been done through easier embeddings, i.e. going back to Lebsegue spaces and then deteriorating the smoothness indexes. 
The point is then of course to derive the maximum regularity index for which the estimates on the equation with mollified coefficients work. We refer to the proof of Lemma \ref{lem_unifesti_gencase2_RELOADED} below for further details.

{\color{black} 

As previously mentioned, we provide in Appendix \ref{EXT_TO_LONG_TIME} an extension of the analysis below to the long time setting for   \textit{sufficiently small} initial data, in the considered setting of Besov spaces.
}

\textbf{About constants.} 
Introduce the parameter set
\begin{equation}\label{DEF_GRAND_THETA}
\Theta:=\begin{cases}\{d,\textcolor{black}{\kappa},r,p,q,\beta,|b|_{L^r(B_{p,q}^\beta)},p_0,q_0,\beta_0, |\mu|_{B_{p_0,q_0}^{\beta_0}}\},\ {\rm if} \beta\in (-1,0],\\
\{d,\textcolor{black}{\kappa},r,p,q,\beta,|b|_{L^r(B_{p,q}^\beta)},|\div(b)|_{L^r(B_{p,q}^\beta)},p_0,q_0,\beta_0, |\mu|_{B_{p_0,q_0}^{\beta_0}}\},\ {\rm if} \beta=-1.
\end{cases}
\end{equation}
Namely $\Theta $ gathers the various parameters appearing in Assumption \A{UE} and conditions {\color{black}\eqref{cond_coeff_SPKR}, \eqref{THE_COND_CI}} depending on the considered value of $\beta $.
In what follows we denote by $C:=C(\Theta)$ a generic constant depending on the parameter set $\Theta$ that may change from line to line. Other possible dependencies will be explicitly specified. \textcolor{black}{Importantly, $\Theta $ does not depend on time}.\\

We first begin with a Lemma giving a control of the Besov norm of the initial condition in the mollified Fokker-Planck equation \eqref{PDE_EPS}.

{\color{black}
\begin{lemme}[Besov controls for the convolution of the initial condition and the stable heat kernel]\label{controle_CI}
Define:
\begin{equation}\label{DEF_INDEXES_CI}
 p_1=\min(p_0,p'),\ q_1=q_0(1\vee \frac{p}{p_0'}),\ \beta_1=\beta_0(1\wedge \frac{p_0'}p),
\end{equation}
\textcolor{black}{again with the convention that $\infty/\infty=1 $}.
Then, for any $s \ge t$ satifying $s-t$ is small enough, it holds that for any ${\color{black}\gamma>0}$:
 \begin{align}
 \Big|\mu \star p^\alpha_{s-t}\Big|_{B^{ -\beta+{\color{black}\gamma}}_{p',1}} \le 
 C|\mu|_{B_{p_1,q_1}^{\beta_1}} (s-t)^{- \textcolor{black}{\frac{1}{\alpha}\left[{\color{black}\gamma} - \beta + \frac dp - \zeta_0\right]_+}}
 ,
 \label{NEW_CTR_CI}
\end{align} 
where $\zeta_0=\left(\beta_0 + \frac{d}{p_0'}\right)\left( 1 \wedge \frac{p_0'}{p} \right)$ is the \textit{regularity gain from the initial condition} defined in \eqref{def_zetra}.
\end{lemme}
\begin{proof}[Proof of Lemma \ref{controle_CI}] {\color{black} {\color{black}T}o obtain the desired result, we aim at applying Young's inequality \eqref{YOUNG}. When doing so, we need to distinguish between two cases depending on the position of $p'$ (dual integrability index of the singular convolution kernel $b$) w.r.t $p_0$ (integrability index of the initial condition).}  Recall that $p_1 = \min(p_0,p')$ and let us define
\begin{align*}
[\mathfrak p(p_1,p')]^{-1} &= 1 +  (p')^{-1} - (p_1)^{-1}=\left\lbrace\begin{array}{ll}
\displaystyle 1+(p')^{-1} - (p_0)^{-1} \text{ if } p_0 \le p', \\
 \displaystyle 1 \text{ if }p'<p_0.
\end{array}\right.\\
 \end{align*}
We claim that:
 \begin{equation}\label{HK_PARTIAL_1}
  \Big|\mu \star p^\alpha_{s-t}\Big|_{B^{ -\beta+{\color{black}\gamma}}_{p',1}} \le 
 C|\mu|_{B_{p_1,q_1}^{\beta_1}} |p_{s-t}^\alpha|_{B_{\mathfrak p(p_1,p'),1}^{-\beta+{\color{black}\gamma}-\beta_1}}.
 \end{equation}

\begin{itemize}
\item[$\bullet $]Assume first that $ p'\ge p_0$ ($\iff p\le p_0' $). {\color{black}We can then directly apply \eqref{YOUNG} which yields}
 \begin{eqnarray*}
\Big|\mu \star p^\alpha_{s-t}\Big|_{B^{ -\beta+{\color{black}\gamma}}_{p',1}} &\le& \cv|\mu|_{B^{\beta_0}_{p_0,q_0}} | p^\alpha_{s-t}|_{B^{ -\beta+{\color{black}\gamma}-\beta_0}_{\mathfrak p(p_0,p'), 1}}
 \end{eqnarray*}
so that \eqref{HK_PARTIAL_1} holds.
 
\item[$\bullet$] Assume now that $ p'< p_0$ ($\iff p>p_0' $). In this case, we must slightly reduce the contribution of the initial condition in order to apply Young's inequality. Since $\mu \in B_{1,\infty}^0$ as a probability law, see  \eqref{lem_proba_in besov}, we derive by interpolation that, see e.g. Theorem 4.29 in \cite{Sawano-18}, it also belongs to \textcolor{black}{$B_{p',q_0 p/p_0'}^{\beta_0 p_0'/p}=B_{p_1,q_1}^{\beta_1} $} in the current case, with $p_1,q_1,\beta_1 $ as in \eqref{DEF_INDEXES_CI}. Indeed, for $\lambda \in (0,1) $, since $\beta_0\ge 0 $,
$$[B_{1,\infty}^0,B_{p_0,q_0}^{\beta_0}]_\lambda=B_{\tilde p,\tilde q}^{\lambda \beta_0}, \ \frac{1}{\tilde p}=\lambda \frac{1}{p_0}+(1-\lambda), \frac 1{\tilde q}=\lambda\frac 1{q_0}.$$ 
Choosing $\tilde p=p'\in [1,p_0) $\footnote{\textcolor{black}{This choice is somehow natural since it induces no additional singularity on the heat kernel norm associated with the integrability exponent.}} yields $\lambda(1- 1/p_0)=1-(1/p')\iff \lambda=[(1/p)/(1/p_0')]=p_0'/p$ that 
\textcolor{black}{$\mu \in B_{p_1,q_1}^{\beta_1}$}. Applying \eqref{YOUNG}  we get in that case:
\begin{equation}\label{NEW_CTR_CI_ST}
 \Big|\mu \star p^\alpha_{s-t}\Big|_{B^{ -\beta+\gamma}_{p',1}}\le c_{\mathbf Y}|\mu|_{\textcolor{black}{B_{p_1,q_1}^{\beta_1}}} | p^\alpha_{s-t}|_{B^{ -\beta+\gamma-\textcolor{black}{\beta_1}}_{1, 1}},
 \end{equation}
which gain yields that \eqref{HK_PARTIAL_1} holds.\\
\end{itemize}

 {\color{black}Starting from \eqref{HK_PARTIAL_1}, we note that up to a possible slight modification of $\gamma$ when $p_1=p_0$, we can assume w.l.o.g. that, whenever $ -\beta+\gamma-\beta_0<0$, $-\beta+\gamma-\beta_0 \neq -d(1-1/\mathfrak p(p_0,p'))$ (the latter being equal to 0 when $p_1=p'$). We can therefore apply the stable-kernel estimate \eqref{SING_STABLE_HK} to obtain:}
 
 \begin{eqnarray*}
  \Big|\mu \star p^\alpha_{s-t}\Big|_{B^{ -\beta+{\color{black}\gamma}}_{p',1}} \le 
 C|\mu|_{B_{p_1,q_1}^{\beta_1}}(s-t)^{-\frac 1\alpha\left[-\beta+\gamma-\beta_1+d\left(1-\frac 1{\mathfrak p(p_1,p')}\right)\right]_+}.
 \end{eqnarray*}
 From the very definition of $\beta_1 = \beta_0(1\wedge p_0'/p)$ and $\mathfrak p(p_1,p')=1 +  (p')^{-1} - (p_1)^{-1}$, $p_1=p_0\wedge p'$, we have
 \begin{eqnarray*}
 -\beta+\gamma-\beta_1+d\left(1-\frac 1{\mathfrak p(p_1,p')}\right) &=& -\beta + \gamma + \frac dp -\beta_0\left(1\wedge \frac{p_0'}{p} \right)+ d\left(1-\frac 1{p_1}\right)\\
 &=&-\beta + \gamma + \frac dp -\beta_0\left(1\wedge \frac{p_0'}{p} \right)+ \frac d{p_0'}\left(1\wedge \frac{p_0'}{p}\right),
 \end{eqnarray*}
since $1-1/p_1 = (1-1/p_0)\mathbf 1_{p_0 \le p'} + (1-1/p')\mathbf 1_{p_0 >p'} = [1/p_0'] \mathbf 1_{p_0' \ge p} + [1/p]\mathbf 1_{p_0' <p} = [1/p_0'](1\wedge [p_0'/p])$.

Recalling now from \eqref{def_zetra} that $\zeta_0 = \left(\beta_0 + d/(p_0')\right)\left( 1 \wedge [p_0'/p] \right),$ we conclude
that for any $s>t$,
 \begin{align*}
 \Big|\mu \star p^\alpha_{s-t}\Big|_{B^{ -\beta+\Gamma}_{p',1}} \le 
 C|\mu|_{B_{p_1,q_1}^{\beta_1}} \textcolor{black}{(s-t)^{- \frac{1}{\alpha}\left[\Gamma - \beta + \frac dp - \zeta_0\right]_+}}.
\end{align*} 
\textit{i.e.} \eqref{NEW_CTR_CI} holds
. This concludes the proof.
\end{proof}

We now state a Lemma giving a control of the Besov norm \textcolor{black}{of the mollified Fokker-Planck equation \eqref{PDE_EPS}}.

\begin{lemme}[A priori estimates on the mollified density]\label{lem_unifesti_gencase2_RELOADED} 
Under 
\textbf{(C1)} or \textbf{(C2)}, 
setting
\begin{equation}\label{def_gamma}
\Gamma := \textcolor{black}{
\eta\left\{ \alpha-1+\beta-\frac \alpha r + \beta \ind_{\beta >-1}-\frac dp + \zeta_0 \right\},\quad \eta \in (0,1)},
\end{equation}
and 
\begin{equation}\label{def_thetas}
\theta =\textcolor{black}{ 
\frac 1\alpha\left\{-\beta + \frac dp - \zeta_0 + \left(\frac{1+\eta}{2\eta}\right) \Gamma\right\}}>0,
\end{equation}
for $\eta$ sufficiently close to 1, there exists {\color{black}$C_{\ref{lem_unifesti_gencase2_RELOADED}} := C(\Theta)>0$} such that for all {\color{black}$S \le T$},
\begin{eqnarray}
&&\sup_{s\in (t,S]}\Big\{ (s-t)^{\theta} |\brho_{t,\mu}^\varepsilon(s,\cdot)|_{B_{p',1}^{-\beta+ \Gamma}}\Big\} \le  \textcolor{black}{C_{\ref{lem_unifesti_gencase2_RELOADED}}\Bigg\{} \textcolor{black}{|\mu|_{B^{\beta_1}_{p_1,q_1}}}  (S-t)^{
\frac{1-\eta}{2\eta} \frac \Gamma{\alpha}
}
\notag\\
&&\qquad  +\left(|b|_{L^{r}(B^{\beta}_{p,q})} + |\div(b)|_{L^{r}(B^{\beta}_{p,q})}\ind_{\beta =-1} \right)  (S-t)^{
\frac{1-\eta}{2\eta}\frac \Gamma{\alpha}
}\bigg(\sup_{s\in (t,S]}\Big\{ (s-t)^{\theta} |\brho_{t,\mu}^\varepsilon(s,\cdot)|_{B_{p',1}^{-\beta+ \Gamma}}\Big\} \bigg)^2\textcolor{black}{\Bigg\}}.
\notag
\end{eqnarray}
\end{lemme}

\textcolor{black}{Write explicitly for the previous choice of $\Gamma,\theta $:
\begin{align*}
\theta =& \textcolor{black}{\frac 1\alpha}\left\{-\beta + \frac dp - \zeta_0 + \left(\frac{1+\eta}{2\eta}\right) \Gamma\right\}\\
=& \textcolor{black}{\frac 1\alpha}\left\{-\beta + \frac dp - \zeta_0 + \left(\frac{1+\eta}{2}\right) \left\{ \alpha-1+\beta-\frac \alpha r + \beta \ind_{\beta >-1}-\frac dp + \zeta_0 \right\}\right\}\\
=& \textcolor{black}{\frac 1\alpha}\left\{(-\beta + \frac dp - \zeta_0)\frac{1-\eta}2+ \left(\frac{1+\eta}{2}\right) \underbrace{\left\{ \alpha-1-\frac \alpha r + \beta \ind_{\beta >-1} \right\}}_{> 0}\right\},
\end{align*}
which can be made non-negative provided $\eta $ is sufficiently close to 1. Similarly observe that:
\begin{align}
\Gamma-\beta+\frac dp-\zeta_0=\eta\left\{ \alpha-1+\beta-\frac \alpha r + \beta \ind_{\beta >-1}-\frac dp + \zeta_0 \right\}-\beta+\frac dp-\zeta_0\notag\\
=\eta \underbrace{(\alpha-1-\frac \alpha r+\beta\ind_{\beta >-1})}_{> 0}+(\eta -1)(\beta-\frac dp+\zeta_0).\label{EXPR_GAMMA_MIN_GAMMA_0}
\end{align}
Hence, both quantities can be made positive, even when $\zeta_0 $ is large, provided $\eta $ is sufficiently close to 1. 
}

\begin{proof}[Proof under \textbf{(C1)}] \textcolor{black}{Recalling that we have assumed to work in small time, we can suppose w.l.o.g. that $T\le 1$}.
We start from the Duhamel formulation \eqref{main_MOLL} and apply the norm $|\cdot|_{B^{ -\beta+\Gamma}_{p',1}}$ for $\Gamma> 0 $ to be specified later on, \textcolor{black}{i.e. for the moment we take a generic $\Gamma $ and specify along the proof the constraints that are needed for the analysis to work keeping in mind that we want to take $\Gamma $ as large as possible}. For $s\in (t,T] $, it follows
\begin{eqnarray}
|\brho_{t,\mu}^\varepsilon(s,\cdot)|_{B^{-\beta+\Gamma}_{p',1}} \le  \Big|\mu \star p^\alpha_{s-t}\Big|_{B^{ -\beta+\Gamma}_{p',1}} +
\int_t^s dv \bigg|\Big(\mathcal B_{\brho_{t,\mu}^\varepsilon}^\varepsilon(v,\cdot) \cdot\brho_{t,\mu}^\varepsilon(v,\cdot)\Big) \star \nabla p^{\alpha}_{s-v}\Bigg|_{B^{-\beta+ \Gamma}_{p',1}}.\label{TRIANG_KR}
\end{eqnarray}
Applying successively \eqref{YOUNG} (with $m_1=1,m_2=\infty$), \eqref{PROD1}, \textcolor{black}{\eqref{YOUNG} again} and finally \eqref{BesovEmbedding} yields
\begin{eqnarray}
 &&	\Big|\Big(\mathcal B_{\brho_{t,\mu}^\varepsilon}^\varepsilon(v,\cdot) \cdot  \brho_{t,\mu}^\varepsilon(v,\cdot)\Big) \star \nabla p^{\alpha}_{s-v}\Big|_{B^{-\beta+\Gamma}_{p',1}} \notag\\
 &\le& C \Big|\Big(\mathcal B_{\brho_{t,\mu}^\varepsilon}^\varepsilon(v,\cdot) \cdot  \brho_{t,\mu}^\varepsilon(v,\cdot)\Big) \Big|_{B^{\Gamma}_{p',\infty}} \Big|\nabla p^{\alpha}_{s-v}\Big|_{B^{-\beta}_{1,1}}\notag\\
 &\le & C\Big|\mathcal B_{\brho_{t,\mu}^\varepsilon}^\varepsilon(v,\cdot)\Big|_{B^{ \Gamma}_{\infty,\infty}} \Big|\brho_{t,\mu}^\varepsilon(v,\cdot)\Big|_{B^{ \Gamma}_{p',1}} \Big|\nabla p^{\alpha}_{s-v}\Big|_{B^{-\beta}_{1,1}}\notag\\
 &\le & C|b^\varepsilon(v,\cdot)|_{B^{\beta}_{p,q}}
 \Big|\brho_{t,\mu}^\varepsilon(v,\cdot)\Big|_{B^{ -\beta+\Gamma}_{p',q'}} \Big|\brho_{t,\mu}^\varepsilon(v,\cdot)\Big|_{B^{ \Gamma}_{p',1}} \Big|\nabla p^{\alpha}_{s-v}\Big|_{B^{-\beta}_{1,1}}\notag\\
 &\le& C|b^\varepsilon(v,\cdot)|_{B^{\beta}_{p,q}}\Big|\brho_{t,\mu}^\varepsilon(v,\cdot)\Big|_{B^{ -\beta+\Gamma}_{p',1}}^2 \Big|\nabla p^{\alpha}_{s-v}\Big|_{B^{-\beta}_{1,1}}.\label{first_ctr_lem6_c1} 
\end{eqnarray}

  Coming back to \eqref{TRIANG_KR}, we derive from \eqref{first_ctr_lem6_c1} and the estimate on the initial condition \eqref{NEW_CTR_CI} from Lemma \ref{controle_CI}, using  \eqref{SING_STABLE_HK}, that
\begin{eqnarray} 
|\brho_{t,\mu}^\varepsilon(s,\cdot)|_{B^{-\beta+\Gamma}_{p',1}}
&\le&  C\textcolor{black}{|\mu|_{B^{\beta_1}_{p_1,q_1}}} (s-t)^{- \frac{1}{\alpha}\textcolor{black}{\left[\Gamma - \beta + \frac dp - \zeta_0\right]_+}}
\notag\\
&& +C\int_t^s \frac{dv}{(s-v)^{\frac {-\beta+1}\alpha}
} 
| b^\varepsilon(v,\cdot)|_{B^{\beta}_{p,q}} | \brho_{t,\mu}^\varepsilon(v,\cdot)|_{B^{ -\beta+\Gamma}_{p',1}}^2.\label{BEFORE_NOR}
\end{eqnarray}
 Applying next the $L^1:L^r-L^{r'}$-H\"older inequality in time we get: 
\begin{eqnarray}
|\brho_{t,\mu}^\varepsilon(s,\cdot)|_{B^{-\beta+ \Gamma}_{p',1}} &\le&  \textcolor{black}{C \Bigg\{} \textcolor{black}{|\mu|_{B^{\beta_1}_{p_1,q_1}}}  (s-t)^{- \frac{1}{\alpha}\textcolor{black}{\left[\Gamma - \beta + \frac dp - \zeta_0\right]_+}}
\notag\\
&& +|b|_{L^{r}(B^{\beta}_{p,q})}  \Bigg(\int_t^s \frac{dv}{(s-v)^{\frac {-\beta+1}\alpha r'} 
}  | \brho_{t,\mu}^\varepsilon(v,\cdot)|_{B^{-\beta+\Gamma}_{p', 1}}^{2r'} \Bigg)^{\frac 1{r'}}\textcolor{black}{\Bigg\}}\label{CI_EPS_AVANT_GR_QUADRA_PROOF_2}.
\end{eqnarray}
In order to have time integrable singularities in the former integral we assume that:
\begin{equation}\label{C1ST_a}\tag{a}
\frac 1{r'}-\frac{1-\beta}\alpha >0.
\end{equation}
\textcolor{black}{Let now  $\theta$  be a non negative parameter to be calibrated. We have} 
from \eqref{CI_EPS_AVANT_GR_QUADRA_PROOF_2}:
\begin{eqnarray}
&&\sup_{s\in (t,S]}\Big\{(s-t)^{\theta} |\brho_{t,\mu}^\varepsilon(s,\cdot)|_{B_{p',1}^{-\beta+ \Gamma}}\Big\} \le \textcolor{black}{C\Bigg\{} \textcolor{black}{|\mu|_{B^{\beta_1}_{p_1,q_1}}} \max_{s\in(t,S]} \Big\{(s-t)^{\theta- \frac{1}{\alpha}\textcolor{black}{\left[\Gamma - \beta + \frac dp - \zeta_0\right]_+}}\Big\} \label{CI_EPS_AVANT_GR_QUADRA_PROOF_2_BIS}\\
&&\qquad  +|b|_{L^{r}(B^{\beta}_{p,q})}  \max_{s\in(t,S]}\Bigg\{(s-t)^{\theta}\Bigg(\int_t^s \frac{dv}{(s-v)^{\frac{-\beta+1}{\alpha} r'}} | \brho_{t,\mu}^\varepsilon(v,\cdot)|_{B^{ -\beta+\Gamma}_{p', 1}}^{2r'}  \Bigg)^{\frac 1{r'}}\Bigg\}\textcolor{black}{\Bigg\}}.\notag
\end{eqnarray}
To obtain homogeneous quantities, we now have to singularize the integrand in the above r.h.s. To do so, we need the exponent $\theta$ to be so that $(\textcolor{black}{v}-t)^{-2\theta r'}$ is integrable around $t$. Thus, 
\begin{equation}\label{C1ST_b}\tag{b}
\theta< \frac 1{2r'}. 
\end{equation}
We now write
\begin{eqnarray*}
&&\int_t^s \frac{dv}{(s-v)^{\frac{-\beta+1}{\alpha} r'}}  | \brho_{t,\mu}^\varepsilon(v,\cdot)|_{B^{ -\beta+\Gamma}_{p', 1}}^{2r'}\\
&=&\int_t^s \frac{dv}{(s-v)^{\frac{-\beta+1}{\alpha} r'} (v-t)^{2\theta r'}} \Big((v-t)^{\theta} | \brho_{t,\mu}^\varepsilon(v,\cdot)|_{B^{ -\beta+\Gamma}_{p', 1}}\Big)^{2r'}\\
&\le& \sup_{v\in (t,s]}\Big\{(v-t)^{2\theta r'}|\brho_{t,\mu}^\varepsilon(v,\cdot)|^{2r'}_{B^{-\beta+ \Gamma}_{p',1}}\Big\} \int_t^s \frac{dv}{(s-v)^{\frac{-\beta+1}{\alpha} r'} \,  (v-t)^{2\theta r'}}.
\end{eqnarray*}
\textcolor{black}{Under the conditions \eqref{C1ST_a} and \eqref{C1ST_b} we have assumed, the above time singularities are integrable \textcolor{black}{and a change of variable (which makes the Beta function appear)
 yields}
 :
 }
\begin{eqnarray*}
&&\int_t^s \frac{dv}{(s-v)^{\frac{-\beta+1}{\alpha} r'}}  | \brho_{t,\mu}^\varepsilon(v,\cdot)|_{B^{ -\beta+\Gamma}_{p', 1}}^{2r'}\\
&\le&{\color{black}C} \sup_{v\in [t,s)}\Big\{ (s-t)^{2{\theta}r'} |\brho_{t,\mu}^\varepsilon(v,\cdot)|^{2r'}_{B^{-\beta+\Gamma}_{p',1}}\Big\}(s-t)^{1-\frac{1-\beta}{\alpha}r'-2\theta r'} 
\end{eqnarray*}

Plugging the above estimate \textcolor{black}{into} \eqref{CI_EPS_AVANT_GR_QUADRA_PROOF_2_BIS} yields
\begin{eqnarray}
&&\sup_{s\in (t,S]}\Big\{ (s-t)^{\theta} |\brho_{t,\mu}^\varepsilon(s,\cdot)|_{B_{p',1}^{-\beta+ \Gamma}}\Big\} \le  \textcolor{black}{C\Bigg\{} \textcolor{black}{|\mu|_{B^{\beta_1}_{p_1,q_1}}} \sup_{s\in (t,S]} \Big\{(s-t)^{\theta- \frac{1}{\alpha}\textcolor{black}{\left[\Gamma - \beta + \frac dp - \zeta_0\right]_+}}\Big\} \label{AVANT_EQUILIBRE}\\
&&\qquad  +|b|_{L^{r}(B^{\beta}_{p,q})}   \sup_{s\in (t,S]} \left\{(s-t)^{\frac 1{r'}-\frac{1-\beta}{\alpha}-\theta}\right\}\bigg(\sup_{s\in (t,S]}\Big\{(s-t)^{\theta} |\brho_{t,\mu}^\varepsilon(s,\cdot)|_{B_{p',1}^{-\beta+ \Gamma}}\Big\} \bigg)^2\textcolor{black}{\Bigg\}}.\notag
\end{eqnarray}
Our objective now consists in equilibrating and removing the singularities \textcolor{black}{possibly associated with the initial condition, $[\Gamma - \beta +  d/p - \zeta_0]_+/\alpha$}, and of the integral term, $1/r'-(1-\beta)/\alpha$, in the above in the chosen \textit{small time} setting. 
We could now consider two cases depending on the sign of $-\beta+d/p-\zeta_0 $. Once again, when it is non-negative,  
it means that the exponent which can be absorbed by the initial condition is not greater than the intrinsic regularity needed to give a meaning to convoluted drift by duality), whereas when $-\beta+d/p-\zeta_0 < 0 $, i.e. when the regularity of the initial condition somehow dominates, it somehow means that taking $\Gamma=\zeta_0+\beta-d/p$ we could a priori get rid of  a normalizing factor in time to estimate the density. 

Keeping in mind that we anyhow want to derive estimates for the biggest possible $\Gamma $ (in connection with weak/strong well-posedness for the corresponding SDE) we anyhow will use a time normalization in that setting too, which in fact leads to a choice of $\Gamma $ which does not depend on the sign of $-\beta+d/p-\zeta_0 $.

\textcolor{black}{For the moment we simply assume that $\Gamma \ge [- \beta +  d/p- \zeta_0]_+ $. Since  $\Gamma>0 $, this actually only adds a constraint when $[- \beta +  d/p- \zeta_0]_+\neq 0\iff \zeta_0<  - \beta +  d/p$ (small initial regularity gain)}.
The choice of $\Gamma $ as indicated in the Lemma will actually follow from the various constraints we need to satisfy.

 \begin{trivlist}
\item[$\bullet$] Let us consider the case $-\beta+d/p-\zeta_0 \ge 0   $. In order to balance the singularities for the term associated with the initial condition and to keep a positive exponent in time for the integrated term in \eqref{AVANT_EQUILIBRE}, we need to have:
\end{trivlist}
$$\exists \Gamma >0,\, \theta \ge 0,\, {\rm s.t. }\quad \eqref{C1ST_a}:\, \frac 1{r'}- \frac{1-\beta}\alpha >0,\quad \eqref{C1ST_b}:\, \theta < \frac{1}{2r'}\, \text{ and }\frac{\Gamma - \beta + \frac dp- \zeta_0}{\alpha} < \theta \textcolor{black}{<} \frac{1}{r'}-\frac{(1-\beta)}{\alpha}.$$
We first note that, as for any $\beta\in (-1,0]$, $\alpha \in (1,2]$ and $r\ge 1$, $[1/r'-(1-\beta)/\alpha] \le 1/(2r') \iff \alpha(1-1/r) \le 2(1-\beta)$ hold, one may rewrite
\begin{equation}\label{BD_FOR_GAMMA}
\exists \Gamma >0,\, \theta \ge 0,\, {\rm s.t. }\quad \frac{\Gamma - \beta + \frac dp- \zeta_0}{\alpha} < \theta < \frac{1}{r'}-\frac{(1-\beta)}{\alpha}.
\end{equation}
We thus need
$$ \left\{\frac 1\alpha\left[- \beta + \frac dp- \zeta_0\right] <  \frac1{r'}-\frac{1-\beta}{\alpha}  \iff 1-\alpha + \frac \alpha r + \frac dp   -\zeta_0 -\beta  < \beta \text{ and }\textcolor{black}{1-\alpha + \frac \alpha r  < \beta}\right\} \Leftrightarrow \eqref{cond_coeff_SPKR}.$$
This allows to choose $\Gamma/\alpha$ as \textcolor{black}{fraction} of the distance between the exponent related to the initial condition, \textcolor{black}{$(- \beta + d/p- \zeta_0)/\alpha$} and of the distance to the singularity of the norm of the heat kernel, $1/r'-(1-\beta)/\alpha$. Then, to equilibrate and avoid the explosion of the time contribution in \eqref{AVANT_EQUILIBRE}, a reasonable choice for $\theta$ is the middle point between the two exponents. This gives, for $\eta \in (0,1)$:
$$ \Gamma/\alpha := \eta \left( [1/r'-(1-\beta)/\alpha] - [- \beta + d/p- \zeta_0]/\alpha]\right),\quad \theta :=  \left[- \beta + d/p- \zeta_0\right]/\alpha  + \frac{1+\eta}{2\eta} \Gamma/\alpha,$$
which equivalently rewrites:
\begin{equation}\label{def_gamma_theta1}
\textcolor{black}{\Gamma = \eta \left\{\alpha-1+ 2\beta - \frac \alpha r - \frac dp + \zeta_0 \right\},\quad \alpha \theta = \left\{-\beta + \frac dp -\zeta_0 \right\} + \frac 1{2\eta}\left(1+\eta\right) \Gamma}.
\end{equation}

We can now plug the parameters $\Gamma$ and $\theta$  we chose in \eqref{AVANT_EQUILIBRE} to obtain the following estimate:
\begin{eqnarray}
&&\sup_{s\in (t,S]}\Big\{ (s-t)^{\theta} |\brho_{t,\mu}^\varepsilon(s,\cdot)|_{B_{p',1}^{-\beta+ \Gamma}}\Big\} \le  \textcolor{black}{C\Bigg\{} \textcolor{black}{|\mu|_{B^{\beta_1}_{p_1,q_1}}} \sup_{s\in (t,S]} \Big\{(s-t)^{\frac{1-\eta}{2\eta}\frac \Gamma{\alpha}}\Big\} \label{APRES_EQUILIBRE}\\
&&\qquad  +C_1|b|_{L^{r}(B^{\beta}_{p,q})}   \sup_{s\in (t,S]} \left\{(s-t)^{\frac{1-\eta}{2\eta}\frac \Gamma{\alpha}}\right\}\bigg(\sup_{s\in (t,S]}\Big\{ (s-t)^{\theta} |\brho_{t,\mu}^\varepsilon(s,\cdot)|_{B_{p',1}^{-\beta+ \Gamma}}\Big\} \bigg)^2\textcolor{black}{\Bigg\}}.\notag
\end{eqnarray}

\begin{trivlist}
\item[$\bullet$] Let us consider the case $-\beta+d/p-\zeta_0 <0   $. Actually, keeping in mind that we want to obtain the biggest admissible $\Gamma$ in the estimates, in connection with weak/strong uniqueness, we can keep the previous choice of $\Gamma $ in \eqref{def_gamma_theta1} which almost saturates the inequality in \eqref{BD_FOR_GAMMA} (when $\eta $ is close to 1).

Hence, in the considered case, \eqref{APRES_EQUILIBRE} still holds under the sole condition
$$\frac 1{r'}- \frac{1-\beta}\alpha >0\iff  1-\alpha + \frac \alpha r  < \beta$$
which is precisely in that case the condition appearing in \eqref{cond_coeff_SPKR}.

\end{trivlist}

This concludes the proof under \A{C1}.\\
\end{proof}

{\color{black}
\begin{proof}[Proof under \textbf{(C2)}]

Let us now restart from the Duhamel formulation \eqref{main_MOLL} under \A{C2}. In the current short time setting, we precisely rebalance the gradient through an integration by parts to alleviate the time singularity on the heat kernel.
We get,
\begin{eqnarray}
|\brho_{t,\mu}^\varepsilon(s,\cdot)|_{B^{1+\Gamma}_{p',1}} &\le&   |\mu \star p^\alpha_{s-t}|_{B^{1+\Gamma}_{p',1}}
+ \int_{t}^s dv\Big|\Big(\div\big(\mathcal B_{\brho_{t,\mu}^\varepsilon}^\varepsilon(v,\cdot)\big) \brho_{t,\mu}^\varepsilon(v,\cdot)\Big) \star  p^{\alpha}_{s-v}\Big|_{B^{1+\Gamma}_{p',1}}\notag\\
&&+ \int_{t}^s dv \Big|\Big(\mathcal B_{\brho_{t,\mu}^\varepsilon}^\varepsilon(v,\cdot) \cdot \nabla \brho_{t,\mu}^\varepsilon(v,\cdot)\Big) \star   p^{\alpha}_{s-v}(\cdot)\Big|_{B^{1+\Gamma}_{p',1}}\label{norm_IPP}.
\end{eqnarray}

Applying successively \eqref{YOUNG} (with $m_1=1,m_2=\infty$), \eqref{PROD1}, \textcolor{black}{\eqref{YOUNG} again} and finally \eqref{BesovEmbedding} yields
\begin{eqnarray*}
 &&	\Big|\textcolor{black}{\Big(}\div\Big(\mathcal B_{\brho_{t,\mu}^\varepsilon}^\varepsilon(v,\cdot) \Big)   \brho_{t,\mu}^\varepsilon(v,\cdot)\textcolor{black}{\Big)} \star  p^{\alpha}_{s-v}\Big|_{B^{1+\Gamma}_{p',1}} \\
 &\le& C \Big|\div\Big(\mathcal B_{\brho_{t,\mu}^\varepsilon}^\varepsilon(v,\cdot)\Big)   \brho_{t,\mu}^\varepsilon(v,\cdot) \Big|_{B^{\Gamma}_{p',\infty}} \Big| p^{\alpha}_{s-v}\Big|_{B^{1}_{1,1}}\\
 &\le & C\Big|\div \Big(\mathcal B_{\brho_{t,\mu}^\varepsilon}^\varepsilon(v,\cdot)\Big)\Big|_{B^{ \Gamma}_{\infty,\infty}} \Big|\brho_{t,\mu}^\varepsilon(v,\cdot)\Big|_{B^{ \Gamma}_{p',1}} \Big| p^{\alpha}_{s-v}\Big|_{B^{1}_{1,1}}\\
 &\le & C|\div(b^\varepsilon(v,\cdot))|_{B^{-1}_{p,q}}
 \Big|\brho_{t,\mu}^\varepsilon(v,\cdot)\Big|_{B^{ 1+\Gamma}_{p',q'}} \Big|\brho_{t,\mu}^\varepsilon(v,\cdot)\Big|_{B^{ 1+\Gamma}_{p',1}} \Big| p^{\alpha}_{s-v}\Big|_{B^{1}_{1,1}}\le C|\div(b^\varepsilon(v,\cdot))|_{B^{-1}_{p,q}}\Big|\brho_{t,\mu}^\varepsilon(v,\cdot)\Big|_{B^{ 1+\Gamma}_{p',1}}^2 \Big| p^{\alpha}_{s-v}\Big|_{B^{1}_{1,1}}.\nonumber 
\end{eqnarray*}
Similarly,
\begin{eqnarray*}
 &&	\Big|\mathcal B_{\brho_{t,\mu}^\varepsilon}^\varepsilon(v,\cdot) \cdot\nabla   \brho_{t,\mu}^\varepsilon(v,\cdot) \star  p^{\alpha}_{s-v}\Big|_{B^{1+\Gamma}_{p',1}} \\
 &\le& C \Big|\mathcal B_{\brho_{t,\mu}^\varepsilon}^\varepsilon(v,\cdot)\cdot \nabla   \brho_{t,\mu}^\varepsilon(v,\cdot) \Big|_{B^{\Gamma}_{p',\infty}} \Big| p^{\alpha}_{s-v}\Big|_{B^{1}_{1,1}}\\
 &\le & C\Big|\mathcal B_{\brho_{t,\mu}^\varepsilon}^\varepsilon(v,\cdot)\Big|_{B^{ \Gamma}_{\infty,\infty}} \Big|\nabla \brho_{t,\mu}^\varepsilon(v,\cdot)\Big|_{B^{ \Gamma}_{p',1}} \Big| p^{\alpha}_{s-v}\Big|_{B^{1}_{1,1}}\\
 &\underset{\eqref{LO}}{\le} & C|b^\varepsilon\textcolor{black}{(v,\cdot)}|_{B^{-1}_{p,q}}
 \Big|\brho_{t,\mu}^\varepsilon(v,\cdot)\Big|_{B^{ 1+\Gamma}_{p',q'}} \Big|\brho_{t,\mu}^\varepsilon(v,\cdot)\Big|_{B^{ 1+\Gamma}_{p',1}} \Big| p^{\alpha}_{s-v}\Big|_{B^{1}_{1,1}}\le C|b^\varepsilon(v,\cdot)|_{B^{-1}_{p,q}}\Big|\brho_{t,\mu}^\varepsilon(v,\cdot)\Big|_{B^{ 1+\Gamma}_{p',1}}^2 \Big| p^{\alpha}_{s-v}\Big|_{B^{1}_{1,1}}.\nonumber 
\end{eqnarray*}
\textcolor{black}{Note from the above bounds that the terms $b^\varepsilon, {\rm div}(b^\varepsilon)$ naturally appear with same norm.} \textcolor{black}{Using again \eqref{NEW_CTR_CI} for the initial condition}, the two above estimates and \eqref{BEFORE_NOR} 
with $\beta=-1 $ in \eqref{norm_IPP},  \textcolor{black}{we obtain} thanks to \eqref{SING_STABLE_HK} that
\begin{eqnarray*}
&&| \brho_{t,\mu}^\varepsilon(s,\cdot)|_{B^{1+\Gamma}_{p',1}} \notag\\
&\le&  \textcolor{black}{C\Bigg\{}|\mu|_{\textcolor{black}{B^{\beta_1}_{p_1,q_1}}} (s-t)^{- \frac{1}{\alpha}\left[\Gamma + 1 + \frac dp - \zeta_0\right]}
+\int_{\textcolor{black}{t}}^s \frac{dv}{(s-v)^{\frac 1 \alpha}} \left(|  b^\varepsilon(v,\cdot) |_{B^{-1}_{p,q}} +|  \div(b^\varepsilon(v,\cdot)) |_{B^{-1}_{p,q}}\right)  |\brho_{t,\mu}^\varepsilon(v,\cdot) |_{B^{1+\Gamma}_{p', 1}}^2 \textcolor{black}{\Bigg\}}.
\end{eqnarray*}
\textcolor{black}{Applying} the $L^r-L^{r'}$ H\"older inequality in time in the above equation, we get
\begin{eqnarray*}
|\brho_{t,\mu}^\varepsilon(s,\cdot)|_{B^{1+\Gamma}_{p',1}} &\le& \textcolor{black}{C\Bigg\{}|\mu|_{\textcolor{black}{B^{\beta_1}_{p_1,q_1}}} (s-t)^{- \textcolor{black}{\frac{1}{\alpha}\left[\Gamma + 1 + \frac dp - \zeta_0\right]_+}} \\
&&+\Big(\big| b^\varepsilon \big|_{L^{r}(B^{-1}_{p,q})} +|  \div(b^\varepsilon) |_{L^r(B^{-1}_{p,q})}\Big) \Bigg(\int_t^s \frac{dv}{\textcolor{black}{(s-v)^{\frac {r'}\alpha}}}  | \brho_{t,\mu}^\varepsilon(v,\cdot)|_{B^{1+\Gamma}_{p', 1}}^{2r'} \Bigg)^{\frac 1{r'}}\textcolor{black}{\Bigg\}}. \label{PREAL_MULT_C2}
\end{eqnarray*}
We now multiply both sides by $(s-t)^{\theta}$, argue as we did to pass from \eqref{CI_EPS_AVANT_GR_QUADRA_PROOF_2} to \eqref{AVANT_EQUILIBRE} to deduce that
\begin{eqnarray*}
&&\sup_{s \in (t,S]}\Big\{(s-t)^{\theta}|\brho_{t,\mu}^\varepsilon(s,\cdot)|_{B^{1+\Gamma}_{p',1}}\Big\} \le  \textcolor{black}{C\Bigg\{}|\mu|_{\textcolor{black}{B^{\beta_1}_{p_1,q_1}}} \max_{s\in(t,S]} \Big\{(s-t)^{\theta- \textcolor{black}{\frac{1}{\alpha}\left[\Gamma + 1 + \frac dp - \zeta_0\right]_+}}\Big\}\notag\\
&&+\Big(\big|  b \big|_{L^{r}(B^{-1}_{p,q})}+|{\div}(b)|_{L^r(B_{\textcolor{black}{p,q}}^{-1})}\Big)\sup_{s\in (t,S]} \left\{(s-t)^{\frac 1{r'}-\frac{1}{\alpha}-\theta}\right\}\bigg(\sup_{s\in (t,S]}\Big\{ (s-t)^{\theta} |\brho_{t,\mu}^\varepsilon(s,\cdot)|_{B_{p',1}^{1 + \Gamma}}\Big\} \bigg)^2,\notag 
\end{eqnarray*}
where we implicitly assumed that the following conditions hold to have time integrable singularities:
$$2\theta r'<1\iff \theta <\frac {1}{2r'}, \text{ and } \frac 1{r'} - \frac 1\alpha >0.$$
\textcolor{black}{We now distinguish as above two cases depending on the sign of $1+d/p-\zeta_0 $.}

\begin{trivlist}
\item[$\bullet$] Let us consider the case $1+d/p-\zeta_0 \ge 0   $. In order to balance the singularities for the term associated with the initial condition and to keep a positive exponent in time for the integrated term in the current \textit{small time} regime, we need to have:
\end{trivlist}
\begin{equation}\label{INEQ_FOR_GAMMA_C2}
\exists \Gamma >0,\, \theta \ge 0,\, {\rm s.t. }\quad \frac{\Gamma + 1+  \frac dp- \zeta_0}{\alpha} < \theta < \frac{1}{r'}-\frac{1}{\alpha}.
\end{equation}
As $1/r' - 1/\alpha \le 1/(2r') \Leftrightarrow \alpha(1-1/r) \le 2$ is always satisfied, conditions reduce to
$$\left\{\exists \Gamma >0,\, \theta \ge 0,\, {\rm s.t. }\quad \frac{\Gamma + 1+  \frac dp- \zeta_0}{\alpha} < \theta < \frac{1}{r'}-\frac{1}{\alpha} \text{ and }  \frac 1{r'} - \frac 1\alpha >0\right\} \Leftrightarrow \eqref{THE_COND_CI}.$$
Reasoning as we did for the proof under \A{C1} gives
\begin{equation}\label{def_gamma_theta1_C2}
{\color{black}\Gamma = \eta \left\{ \alpha - 2 - \frac \alpha r - \frac dp   +\zeta_0 \right\},\quad \alpha \theta = \left\{1  + \frac dp -\zeta_0 \right\} + \frac {1+\eta}{2\eta} \Gamma,\ \eta \in(0,1)}.
\end{equation}
%
It then follows that:
\begin{eqnarray}
&&\sup_{s\in (t,S]}\Big\{ (s-t)^{\theta} |\brho_{t,\mu}^\varepsilon(s,\cdot)|_{B_{p',1}^{-\beta+ \Gamma}}\Big\} \le  \textcolor{black}{C\Bigg\{} \textcolor{black}{|\mu|_{B^{\beta_1}_{p_1,q_1}}}  (S-t)^{\frac{1-\eta}{2\eta}\frac \Gamma{\alpha}}\notag 
\\
&&\qquad  +C_1\Big(|b|_{L^{r}(B^{-1}_{p,q})}+|\div(b)|_{L^{r}(B^{-1}_{p,q})}\Big)   
\textcolor{black}{(S-t)^{\frac{1-\eta}{2\eta} \frac \Gamma{\alpha}}}\bigg(\sup_{s\in (t,S]}\Big\{(s-t)^{\theta} |\brho_{t,\mu}^\varepsilon(s,\cdot)|_{B_{p',1}^{-\beta+ \Gamma}}\Big\} \bigg)^2\textcolor{black}{\Bigg\}},\label{APRES_EQUILIBRE_C2}
\end{eqnarray}
provided again that $ 1/r'-1/\alpha<1/(2r')\iff 1/r'<2\alpha\iff \alpha(1-\frac 1r)<2$.

\begin{trivlist}
\item[$\bullet $] Let us consider the case $1+d/p-\zeta_0 <0   $. We then again choose $\Gamma $ as in \eqref{def_gamma_theta1_C2}  which almost saturates the inequality in \eqref{INEQ_FOR_GAMMA_C2} (for $\eta $ close to 1). We then again get \eqref{APRES_EQUILIBRE_C2}
under the sole condition
$$\frac 1{r'}- \frac{1}\alpha >0\iff  1-\alpha + \frac \alpha r  < 0,$$
which is precisely in that case the condition appearing in \eqref{THE_COND_CI}.
\end{trivlist}

\end{proof}
}


\begin{lemme}[A priori control through a Gronwall type inequality with quadratic growth]\label{lem_quad_gronv_C1} Under \textbf{(C1)} or \textbf{(C2)} 
and for $\theta, \Gamma$ as in Lemma \ref{lem_unifesti_gencase2_RELOADED} and Theorem \ref{main_thm_W}, 
we have that there exists $\mathcal T_1\le T $ and \textcolor{black}{$C_{\ref{lem_quad_gronv_C1}}$ s.t. uniformly in $\varepsilon $}, for any $\textcolor{black}{S\le  \mathcal T_1}$:
\begin{eqnarray}\label{Unifestim_quadrgronwall_C1}
 \sup_{s\in (t,S]}[(s-t)^{\theta} |\brho_{t,\mu}^\varepsilon(s,\cdot)|_{B^{ \textcolor{black}{-\beta}+\Gamma}_{p',{\color{black}1}}}]
\le 
{\color{black}C_{\ref{lem_quad_gronv_C1}}}.
\end{eqnarray} 
\end{lemme}
}

\begin{proof} 
 {\color{black} From the previous Lemma {\color{black}\ref{lem_unifesti_gencase2_RELOADED}}  we have that for any $\textcolor{black}{S\le T}$, the mapping $$f_t^{\textcolor{black}{\varepsilon}}:s\in \textcolor{black}{(}t,S]\mapsto f_t^{\textcolor{black}{\varepsilon}}(s):=\sup_{v\in (t,s]}(v-t)^{\theta} |\brho_{t,\mu}^\varepsilon(v,\cdot)|_{B^{ -\beta+\Gamma}_{p',{\color{black}1}}}$$ satisfies an inequality of the form:
\begin{equation}\label{PREAL_TRINOME}
0 \le  a_t(s) - f_t^{\textcolor{black}{\varepsilon}}(s)+c_t(s)(f_t^{\textcolor{black}{\varepsilon}}(s))^2,\,t< s\le S,
\end{equation}
where, 
\begin{eqnarray*}
a_t(s)&=& C_{\ref{lem_unifesti_gencase2_RELOADED}} |\mu|_{\textcolor{black}{B^{\beta_1}_{p_1,q_1}}} 
(s-t)^{\frac{1-\eta}{2\eta} \frac \Gamma{\alpha}} 
=  C_{\ref{lem_unifesti_gencase2_RELOADED}}  c_0  (\textcolor{black}{s}-t)^{\frac{1-\eta}{2\eta} \frac \Gamma{\alpha}
},\\
c_t(s)&=& C_{\ref{lem_unifesti_gencase2_RELOADED}} \textcolor{black}{\Big(}| b|_{L^{r}(B^{\beta}_{p,q})}+\textcolor{black}{| {\rm div}(b)|_{L^{r}(B^{\beta}_{p,q})}\ind_{\beta=-1}\Big)} 
(s-t)^{\frac{1-\eta}{2\eta} \frac \Gamma{\alpha}}
= : C_{\ref{lem_unifesti_gencase2_RELOADED}} c_b  (\textcolor{black}{s}-t)^{\frac{1-\eta}{2\eta} \frac \Gamma{\alpha}}
.
\end{eqnarray*}

Define, for $\mathcal T_1\in (0,T] $ to specify, the polynomial $P_{\mathcal T_1}(z) = a_t(\mathcal T_1) -z+c_t(\mathcal T_1)z^2$. \textcolor{black}{Since the time dependent coefficients $s\in (t,S]\mapsto a_t(s),c_t(s) $ are increasing,  we thus have from \eqref{PREAL_TRINOME} that} $P_{\mathcal T_1}(f_t^{\textcolor{black}{\varepsilon}}(s)) \ge 0$, $t< s\le S\le \mathcal T_1$. Moreover, as soon as $c_t(\mathcal T_1)a_t(\mathcal T_1) < 1/4$, \textcolor{black}{which will always be the case provided that $T$ is small enough}, this polynomial admits two positive roots and since from Lemma \ref{lem_cont0} we have that for every fixed $\varepsilon >0$, $s \mapsto f_t^\varepsilon (s)$ is continuous and $ f_t^\varepsilon (s) \to 0$ as $s \to t$, we obtain that $f_t^\varepsilon(s)$ is bounded by the smaller root of the polynomial, namely,
$$\forall t< s \le S,\, f_t^\varepsilon(s) \le \textcolor{black}{\frac{1-\sqrt{1-4c_t(\mathcal T_1)a_t(\mathcal T_1)}}{2c_t(\mathcal T_1)}}.$$
Setting then
$$\mathcal T_1 :=( t + \left[8C_{\ref{lem_unifesti_gencase2_RELOADED}}^2 c_0c_b\right]^
{-[\eta/(1-\eta)]\, [\alpha/\Gamma]
})\wedge T \Longrightarrow 4 c_t(\mathcal T_1)a_t(\mathcal T_1) \le \frac 12,
$$
we obtain from Taylor's formula that 
\begin{align*}
f_t^\varepsilon(s)&\le a_t(\mathcal T_1)+\frac 1{8c_t(\mathcal T_1)}\int_0^1 d\lambda y^2(1-y\lambda)^{-\frac 32}|_{y=4c_t(\mathcal T_1)a_t(\mathcal T_1)}\le a_t(\mathcal T_1)+\frac{16(c_t(\mathcal T_1)a_t(\mathcal T_1))^2}{2^{\frac 32}c_t(\mathcal T_1)}\\
&\le a_t(\mathcal T_1)(1+2^{\frac 52}c_t(\mathcal T_1)a_t(\mathcal T_1)):=C_{\ref{lem_quad_gronv_C1}},
\end{align*}
which by the previous choice of $\mathcal T_1 $ indeed only depends on the parameters. This gives the claim.
}
\end{proof}

{\color{black}
\begin{lemme}[Convergence of the mollified densities]\label{SECOND_STAB_BIS} 
{\color{black}Assume \textbf{(C1)} or  \textbf{(C2)} holds.  
For all $S$ smaller than the horizon time $\Tspkr$ given in Lemma \ref{lem_quad_gronv_C1}, for any decreasing sequence $(\varepsilon_k)_{k\ge 1} $ s.t. $\varepsilon_k \underset{k}{\longrightarrow} 0$, $\Big(\brho_{t,\mu}^{\varepsilon_k}\Big)_{k\ge 1} $ is a Cauchy sequence in $L_{{\rm w}_{\theta}}^{\infty}((t,S], B^{ -\beta+\vartheta\Gamma}_{p',1})$ (defined in \eqref{DEF_NORM_W}) with} $\theta,\Gamma$ as in Lemma \ref{lem_unifesti_gencase2_RELOADED} (see \eqref{def_gamma} and \eqref{def_thetas}), $\vartheta=1 $ if $p,q,r<+\infty $ and any $\vartheta\in (0,1) $ if $p\vee q\vee r=+\infty $.
In particular, there exists $\brho_{t,\mu}\in L_{{\ww}_{\theta}}^{\infty}((t,S], B^{-\textcolor{black}{\beta}+ \vartheta\Gamma}_{p',1}) $ s.t.
\begin{equation}\label{STRONG_CONV_2}
 \sup_{s\in (t,S]}\textcolor{black}{(s-t)^{\theta}}|(\brho^{\varepsilon_k}_{t,\mu}-\brho_{t,\mu})(s,\cdot)|_{B^{-\beta+\vartheta \Gamma}_{p',1}} +\sup_{s\in(t,S]}|\textcolor{black}{(\brho^{\varepsilon_k}_{t,\mu}-\brho_{t,\mu})(s,\cdot)}|_{L^1(\mathbb R^d)}\underset{k}{\longrightarrow} 0.
\end{equation}
\end{lemme}
\begin{proof}
\textcolor{black}{Fix $k,j\in \mathbb N $}. 
Assume w.l.o.g. that $k\ge j $. We have from the Duhamel representation \eqref{main_MOLL}
\begin{eqnarray*}
 \brho^{\varepsilon_k}_{t,\mu}(s,y) -\brho^{\varepsilon_j}_{t,\mu}(s,y) &=&-\int_t^s dv \Big[\{{\mathcal B}_{\brho^{\varepsilon_k}_{t,\mu}}^{\varepsilon_k}(v,\cdot) \brho^{\varepsilon_k}_{t,\mu}(v,\cdot)-{\mathcal B}_{\brho^{\varepsilon_j}_{t,\mu}}^{\varepsilon_j}(v,\cdot) \brho^{\varepsilon_j}_{t,\mu}(v,\cdot)\} \star \nabla p_{s-v}^{\alpha}\Big] (y).
\end{eqnarray*}
\textcolor{black}{Fix $\vartheta\in (0,1) $ that can be chosen arbitrarily close to 1}. Applying successively \eqref{YOUNG} (with $m_1=1,m_2=\infty$), \eqref{PROD1}
 , \textcolor{black}{\eqref{YOUNG} again} and finally \eqref{BesovEmbedding} yields
\begin{eqnarray}
 &&	\Big|\Big(\mathcal B_{\brho_{t,\mu}^{\varepsilon_k}}^{\varepsilon_k}(v,\cdot)   \brho_{t,\mu}^{\varepsilon_k}(v,\cdot)-\mathcal B_{\brho_{t,\mu}^{\varepsilon_j}}^{\varepsilon_j}(v,\cdot)   \brho_{t,\mu}^{\varepsilon_j}(v,\cdot)\Big) \star \nabla p^{\alpha}_{s-v}\Big|_{B^{-\beta+\textcolor{black}{\vartheta}\Gamma}_{p',1}} \notag\\
 &\le& C \Big|\Big(\mathcal B_{\brho_{t,\mu}^{\varepsilon_k}}^{\varepsilon_k}(v,\cdot)  \brho_{t,\mu}^{\varepsilon_k}(v,\cdot)-\mathcal B_{\brho_{t,\mu}^{\varepsilon_j}}^{\varepsilon_j}(v,\cdot)  \brho_{t,\mu}^{\varepsilon_j}(v,\cdot)\Big) \Big|_{B^{\textcolor{black}{\vartheta}\Gamma}_{p',\infty}} \Big|\nabla p^{\alpha}_{s-v}\Big|_{B^{-\beta}_{1,1}}\notag\\
 &\le & C\left( \Big|\mathcal B_{\brho_{t,\mu}^{\varepsilon_k}}^{\varepsilon_k}(v,\cdot)-\mathcal B_{\brho_{t,\mu}^{\varepsilon_j}}^{\varepsilon_j}(v,\cdot)\Big|_{B^{ \textcolor{black}{\vartheta}\Gamma}_{\infty,\infty}} \Big|\brho_{t,\mu}^{\varepsilon_k}(v,\cdot)\Big|_{B^{ \textcolor{black}{\vartheta}\Gamma}_{p',1}} +\Big|\mathcal B_{\brho_{t,\mu}^{\varepsilon_j}}^{\varepsilon_j}(v,\cdot)\Big|_{B^{\textcolor{black}{\vartheta} \Gamma}_{\infty,\infty}} \Big|\brho_{t,\mu}^{\varepsilon_k}(v,\cdot)-\brho_{t,\mu}^{\varepsilon_j}(v,\cdot)\Big|_{B^{ \textcolor{black}{\vartheta}\Gamma}_{p',1}}\right)\notag\\
 &&\times \Big|\nabla p^{\alpha}_{s-v}\Big|_{B^{-\beta}_{1,1}}\notag\\
 &\le & C\left( \Big(|(b^{\varepsilon_k}-b^{\varepsilon_j})(v,\cdot)|_{B_{p,q}^{\beta+\textcolor{black}{\Gamma(\vartheta-1)}}}\Big|\brho_{t,\mu}^{\varepsilon_k}(v,\cdot)\Big|_{B^{ -\beta+\Gamma}_{p',1}} +|b^{\varepsilon_j}(v,\cdot)|_{B_{p,q}^\beta}\Big|(\brho_{t,\mu}^{\varepsilon_k}-\brho_{t,\mu}^{\varepsilon_j})(v,\cdot)\Big|_{B^{ -\beta+\textcolor{black}{\vartheta}\Gamma}_{p',1}}\Big)\Big|\brho_{t,\mu}^{\varepsilon_k}(v,\cdot)\Big|_{B^{ \textcolor{black}{\vartheta}\Gamma}_{p',1}}\right.\notag\\
 &&\left.+|b^{\varepsilon_j}(v,\cdot)|_{B^{\beta}_{p,q}}
 \Big|\brho_{t,\mu}^{\varepsilon_j}(v,\cdot)\Big|_{B^{ -\beta+\textcolor{black}{\vartheta}\Gamma}_{p',q'}} \Big|(\brho_{t,\mu}^{\varepsilon_k}-\brho_{t,\mu}^{\varepsilon_j})(v,\cdot)\Big|_{B^{ \textcolor{black}{\vartheta}\Gamma}_{p',1}} \right)\Big|\nabla p^{\alpha}_{s-v}\Big|_{B^{-\beta}_{1,1}}\notag\\
 &\le & C \Big(|(b^{\varepsilon_k}-b^{\varepsilon_j})(v,\cdot)|_{B_{p,q}^{\beta+\textcolor{black}{\Gamma(\vartheta-1)}}}\Big|\brho_{t,\mu}^{\varepsilon_k}(v,\cdot)\Big|_{B^{ -\beta+\Gamma}_{p',1}}^2 +|b^{\varepsilon_j}(v,\cdot)|_{B_{p,q}^\beta}\Big|(\brho_{t,\mu}^{\varepsilon_k}-\brho_{t,\mu}^{\varepsilon_j})(v,\cdot)\Big|_{B^{ -\beta+\textcolor{black}{\vartheta}\Gamma}_{p',1}}\Big|\brho_{t,\mu}^{\varepsilon_k}(v,\cdot)\Big|_{B^{-\beta+ \Gamma}_{p',1}}\Big)\notag\\
 &&
\times \Big|\nabla p^{\alpha}_{s-v}\Big|_{B^{-\beta}_{1,1}}
 .\label{PREAL_STAB_CAS_FACILE}
\end{eqnarray}
Exploiting now Lemma \ref{lem_quad_gronv_C1} and \eqref{SING_STABLE_HK} we get:
\begin{eqnarray*}
 &&	\Big|\Big(\mathcal B_{\brho_{t,\mu}^{\varepsilon_k}}^{\varepsilon_k}(v,\cdot)   \brho_{t,\mu}^{\varepsilon_k}(v,\cdot)-\mathcal B_{\brho_{t,\mu}^{\varepsilon_j}}^{\varepsilon_j}(v,\cdot)   \brho_{t,\mu}^{\varepsilon_j}(v,\cdot)\Big) \star \nabla p^{\alpha}_{s-v}\Big|_{B^{-\beta+\textcolor{black}{\vartheta}\Gamma}_{p',1}} \\
 &\le & \ch 
 \Big(|(b^{\varepsilon_k}-b^{\varepsilon_j})(v,\cdot)|_{B_{p,q}^{\beta+\textcolor{black}{\Gamma(\vartheta-1)}}}\big((v-t)^{-\theta}C_{\ref{lem_quad_gronv_C1}}\big)^2\\
 &&+ |b^{\varepsilon_j}(v,\cdot)|_{B^{\beta}_{p,q}}\textcolor{black}{(v-t)^{{-\theta}}}C_{\ref{lem_quad_gronv_C1}} \Big|(\brho_{t,\mu}^{\varepsilon_k}-\brho_{t,\mu}^{\varepsilon_j})(v,\cdot)\Big|_{B^{ -\beta+\textcolor{black}{\vartheta}\Gamma}_{p',1}}\Big)(s-v)^{-\frac{ 1-\beta}\alpha}.
\end{eqnarray*}
Hence, we \textcolor{black}{derive} from the H\"older inequality that:
\begin{eqnarray*}
&&|(\brho_{t,\mu}^{\varepsilon_k}-\brho_{t,\mu}^{\varepsilon_j})(s,\cdot)|_{B_{p',1}^{-\beta+\textcolor{black}{\vartheta\Gamma }}} \\
&\le & C |b^{\varepsilon_k}-b^{\varepsilon_j}|_{L^{\textcolor{black}{{\bar r}}}(B_{p,q}^{\beta+\textcolor{black}{\Gamma (\vartheta-1)}})}
\left(\int_t^s \frac{dv}{(v-t)^{2\textcolor{black}{\bar r'}\theta}(s-v)^{\textcolor{black}{\bar r'}\frac{ 1-\beta}\alpha}}\right)^{\frac 1{\textcolor{black}{\bar r'}}}\\
&&+ C |b^{\varepsilon_j}|_{L^{\textcolor{black}{\bar r}}(B_{p,q}^\beta)} \left(\int_t^s \frac{dv}{(v-t)^{\textcolor{black}{2}\textcolor{black}{\bar r'}\theta}(s-v)^{\textcolor{black}{\bar r'}\frac{ 1-\beta}\alpha}}\left(\sup_{r\in (t,v]}(\textcolor{black}{r}-t)^{\theta}|(\brho_{t,\mu}^{\varepsilon_k}-\brho_{t,\mu}^{\varepsilon_j})(\textcolor{black}{r},\cdot)|_{B_{p',1}^{-\beta+\textcolor{black}{\vartheta}\Gamma}}\right)^{\bar r'}\right)^{\frac 1{\textcolor{black}{\bar r'}}},
\end{eqnarray*}
\textcolor{black}{with $\bar r=r $ if $r<+\infty $ and any $\bar r<+\infty $ otherwise}. We recall that, similarly to the proof of Lemma \ref{lem_unifesti_gencase2_RELOADED}, under \eqref{cond_coeff_SPKR}, we have $2\theta \textcolor{black}{\bar r'}<1$, and $\textcolor{black}{\bar r'}(1-\beta)/\alpha<1 $. 

\textcolor{black}{Taking $\vartheta\in (0,1) $ is  sufficient to have a negative exponent $\Gamma(\vartheta-1) $ in order to invoke \eqref{SMOOTH_APP_GEN} in Proposition 
\ref{PROP_APPROX} for the convergence of the first term in the above r.h.s. for any integrability parameters $p,q,r $. Note that for $p,q,r<+\infty $ we can also take $\vartheta=1 $, since in that case the Schwartz class is dense in the corresponding space (see \cite[Theorem 2.4 and p. 95]{Sawano-18})}.

Recalling that
$1/\textcolor{black}{\bar  r'} -(1-\beta)/\alpha-\theta>0
$,
we deduce that for $T$ small enough,
\begin{eqnarray*}
\sup_{s\in (t,S]}(s-t)^{\theta}|(\brho_{t,\mu}^{\varepsilon_k}-\brho_{t,\mu}^{\varepsilon_j})(s,\cdot)|_{B_{p',1}^{-\beta+\textcolor{black}{\vartheta}\Gamma}} 
\le C|b^{\varepsilon_k}-b^{\varepsilon_j}|_{L^{\textcolor{black}{\bar r}}(B_{p,q}^{\beta+\textcolor{black}{\Gamma(\vartheta-1)}})}\textcolor{black}{(S-t)^{{\frac 1{\textcolor{black}{\bar r'}} -\frac{1-\beta}{\alpha}-\theta}}}.
\end{eqnarray*}
\textcolor{black}{Since from Proposition \ref{PROP_APPROX}, $|b^{\varepsilon_k}-b^{\varepsilon_j}|_{L^{\bar r}(B_{p,q}^{\beta+\Gamma(\vartheta-1)})} \underset{j,k}{\rightarrow} 0$, we thus derive that $\brho_{t,\mu}^{\varepsilon_k} $ is  a Cauchy sequence in $L_{{\rm w}_{\theta}}^\infty ((t,S],B_{p',1}^{-\beta+\Gamma})$}. 

We have thus  established that $\big(\brho_{t,\mu}^{\varepsilon_k}\big)_k $ is a Cauchy sequence in $L_{{\rm w}_{\theta}}^\infty ((t,S],B_{p',1}^{-\beta+\vartheta\Gamma})$ under \eqref{cond_coeff_SPKR}.
Let us turn to the $L^1$ norm. 
From the embedding \eqref{EMBEDDING} we will actually focus on the $B_{1,1}^0$ norm of the difference and get similarly to \eqref{PREAL_STAB_CAS_FACILE} \textcolor{black}{(using as well that, for $v\in (t,T],\ \brho_{t,\mu}^{\varepsilon_k}(v,\cdot)$ is a probability density)}:
\begin{eqnarray}
&&  \Big|\Big(\mathcal B_{\brho_{t,\mu}^{\varepsilon_k}}^{\varepsilon_k}(v,\cdot)   \brho_{t,\mu}^{\varepsilon_k}(v,\cdot)-\mathcal B_{\brho_{t,\mu}^{\varepsilon_j}}^{\varepsilon_j}(v,\cdot)   \brho_{t,\mu}^{\varepsilon_j}(v,\cdot)\Big) \star \nabla p^{\alpha}_{s-v}\Big|_{B^{0}_{1,1}} \notag\\
 &\underset{\eqref{YOUNG}}{\le}& C \Big|\Big(\mathcal B_{\brho_{t,\mu}^{\varepsilon_k}}^{\varepsilon_k}(v,\cdot)  \brho_{t,\mu}^{\varepsilon_k}(v,\cdot)-\mathcal B_{\brho_{t,\mu}^{\varepsilon_j}}^{\varepsilon_j}(v,\cdot)  \brho_{t,\mu}^{\varepsilon_j}(v,\cdot)\Big) \Big|_{B^{0}_{1,\infty}} \Big|\nabla p^{\alpha}_{s-v}\Big|_{B^{0}_{1,1}}\notag\\
 &\underset{\eqref{EMBEDDING},\textcolor{black}{\eqref{PROD1}}}{\le} & C\left( \Big|\mathcal B_{\brho_{t,\mu}^{\varepsilon_k}}^{\varepsilon_k}(v,\cdot)-\mathcal B_{\brho_{t,\mu}^{\varepsilon_j}}^{\varepsilon_j}(v,\cdot)\Big|_{L^\infty} \underbrace{\Big|\brho_{t,\mu}^{\varepsilon_k}(v,\cdot)\Big|_{L^1}}_{=1} +\Big|\mathcal B_{\brho_{t,\mu}^{\varepsilon_j}}^{\varepsilon_j}(v,\cdot)\Big|_{L^{\infty}} \Big|\brho_{t,\mu}^{\varepsilon_k}(v,\cdot)-\brho_{t,\mu}^{\varepsilon_j}(v,\cdot)\Big|_{B^{0}_{1,\textcolor{black}{1}}}\right)\notag\\
 &&\times \Big|\nabla p^{\alpha}_{s-v}\Big|_{B^{0}_{1,1}}\notag\\
 &\underset{ \textcolor{black}{\eqref{EMBEDDING}, \eqref{YOUNG}}
 , \eqref{BesovEmbedding}}{\le} & C\left( \Big(|(b^{\varepsilon_k}-b^{\varepsilon_j})(v,\cdot)|_{B_{p,q}^{\beta-\textcolor{black}{\vartheta \Gamma}}}\Big|\brho_{t,\mu}^{\varepsilon_k}(v,\cdot)\Big|_{B^{-\beta+ \Gamma}_{p',1}} +|b^{\varepsilon_j}(v,\cdot)|_{B_{p,q}^\beta}\Big|(\brho_{t,\mu}^{\varepsilon_k}-\brho_{t,\mu}^{\varepsilon_j})(v,\cdot)\Big|_{B^{-\beta+\textcolor{black}{\vartheta}\Gamma}_{p',1}}\Big)\right.\notag\\
 &&\left.+|b^{\varepsilon_j}(v,\cdot)|_{B^{\beta}_{p,q}}
 \Big|\brho_{t,\mu}^{\varepsilon_j}(v,\cdot)\Big|_{B^{-\beta+ \Gamma}_{p',1}} \Big|(\brho_{t,\mu}^{\varepsilon_k}-\brho_{t,\mu}^{\varepsilon_j})(v,\cdot)\Big|_{B^{0}_{1,1}} \right)\Big|\nabla p^{\alpha}_{s-v}\Big|_{B^{0}_{1,1}}\notag\\
 &\underset{\eqref{SING_STABLE_HK}, \eqref{Unifestim_quadrgronwall_C1}}{\le} & C\Big( |(b^{\varepsilon_k}-b^{\varepsilon_j})(v,\cdot)|_{B_{p,q}^{\beta-\textcolor{black}{\vartheta\Gamma}}}+|b^{\varepsilon_j}(v,\cdot)|_{B^{\beta}_{p,q}}\Big[
  \textcolor{black}{(v-t)^{\theta}}\Big|(\brho_{t,\mu}^{\varepsilon_k}-\brho_{t,\mu}^{\varepsilon_j})(v,\cdot)\Big|_{B^{-\beta+\textcolor{black}{\vartheta}\Gamma}_{p',1}}\notag\\
  &&\textcolor{black}{+ \Big|(\brho_{t,\mu}^{\varepsilon_k}-\brho_{t,\mu}^{\varepsilon_j})(v,\cdot)\Big|_{B^{0}_{1,1}}\Big]\Big)}\times \textcolor{black}{(v-t)^{-\theta}(s-v)^{-\frac 1 \alpha}}.\label{PREAL_STAB_CAS_FACILE_L1}
\end{eqnarray}
The fact that we made the $L^1$ norm of the mollified densities appear spares us a normalization as in the previous computations. \textcolor{black}{Namely, the time normalized estimate in Besov norm of Lemma \ref{lem_unifesti_gencase2_RELOADED} would have induced additional time singularities}.
\textcolor{black}{We can now write} from \eqref{PREAL_STAB_CAS_FACILE_L1} and the H\"older inequality \textcolor{black}{(for $\bar r$ as above)}:
\begin{align*}
&|(\brho_{t,\mu}^{\varepsilon_k}-\brho_{t,\mu}^{\varepsilon_j})(s,\cdot)|_{B_{1,1}^0} \\
\le & C \left\{\Big(|b^{\varepsilon_k}-b^{\varepsilon_j}|_{L^{\textcolor{black}{\bar r}}(B_{p,q}^{\beta-\textcolor{black}{\vartheta\Gamma}})}
    + |b^{\varepsilon_j}|_{L^{\textcolor{black}{\bar r}}(B_{p,q}^\beta)}[\textcolor{black}{\|\brho_{t,\mu}^{\varepsilon_k}-\brho_{t,\mu}^{\varepsilon_j}\|_{L_{{\rm w}_{\theta}}^\infty(\textcolor{black}{(t,S]},B_{p',1}^{-\beta+\textcolor{black}{\vartheta}\Gamma})}}\Big)
    \left(\int_t^s \frac{dv}{(v-t)^{\bar r' \theta} 
 (s-v)^{\frac {\textcolor{black}{\bar r'}}\alpha}}\right)^{\frac 1{\bar r'}}\right.\\
 &\left.+ |b^{\varepsilon_j}|_{L^{\textcolor{black}{\bar r}}(B_{p,q}^\beta)}\left(\int_t^s \frac{dv}{(v-t)^{\bar r' \theta} 
 (s-v)^{\frac {\textcolor{black}{\bar r'}}\alpha}} \textcolor{black}{\Big(\sup_{r\in (t,v]}|(\brho_{t,\mu}^{\varepsilon_k}-\brho_{t,\mu}^{\varepsilon_j})(r,\cdot)|_{B_{1,1}^0}\Big)^{\bar r'}}\right)^{\frac 1{\textcolor{black}{\bar r'}}}\right\}.
\end{align*}

We can now again invoke the Gronwall-Volterra Lemma and Lemma \ref{lem_gestion_sing} to derive:
\begin{eqnarray*}
\sup_{s\in (t,S]}|(\brho_{t,\mu}^{\varepsilon_k}-\brho_{t,\mu}^{\varepsilon_j})(s,\cdot)|_{B_{1,1}^0} \le C\Big(|b^{\varepsilon_k}-b^{\varepsilon_j}|_{L^r(B_{p,q}^\beta)}+\|\brho_{t,\mu}^{\varepsilon_k}-\brho_{t,\mu}^{\varepsilon_j}\|_{L_{{\rm w}_{\theta}}^\infty(B_{p',1}^{-\beta+\Gamma})}\Big)(S-t)^{\frac{1}{\textcolor{black}{\bar r'}}-(\theta+\frac 1\alpha)},
\end{eqnarray*}
which indeed gives, from Proposition \ref{PROP_APPROX} and the previous part of the proof, that $ (\brho_{t,\mu}^{\varepsilon_k})_k$ is a Cauchy sequence in $L^\infty((t,S],L^1)$. We then derive \eqref{STRONG_CONV_2} by completeness. 

Under \eqref{THE_COND_CI}, i.e. for $\beta=-1 $, let us emphasize that using the structure condition ${\div }(b)\in L^r(B_{p,q}^{-1}) $, one could reproduce the previous arguments writing, 
instead of \eqref{PREAL_STAB_CAS_FACILE}
\begin{eqnarray*}
 &&	\Big|\Big(\mathcal B_{\brho_{t,\mu}^{\varepsilon_k}}^{\varepsilon_k}(v,\cdot)   \brho_{t,\mu}^{\varepsilon_k}(v,\cdot)-\mathcal B_{\brho_{t,\mu}^{\varepsilon_j}}^{\varepsilon_j}(v,\cdot)   \brho_{t,\mu}^{\varepsilon_j}(v,\cdot)\Big) \star \nabla p^{\alpha}_{s-v}\Big|_{B^{-\beta+\textcolor{black}{\vartheta}\Gamma}_{p',1}} \notag\\
 &\le & C \Big(\big[|(b^{\varepsilon_k}-b^{\varepsilon_j})(v,\cdot)|_{B_{p,q}^{\beta+\textcolor{black}{\Gamma(\vartheta-1)}}}+|(\div(b^{\varepsilon_k})-\div(b^{\textcolor{black}{\varepsilon_j}})(v,\cdot)|_{B_{p,q}^{\beta+\textcolor{black}{\Gamma(\vartheta-1)}}}\big]\Big|\brho_{t,\mu}^{\varepsilon_k}(v,\cdot)\Big|_{B^{ -\beta+\Gamma}_{p',1}}^2 \\
 &&+\big[|b^{\varepsilon_j}(v,\cdot)|_{B_{p,q}^\beta}+|\div(b^{\varepsilon_j})(v,\cdot)|_{B_{p,q}^\beta}\big]\Big|(\brho_{t,\mu}^{\varepsilon_k}-\brho_{t,\mu}^{\varepsilon_j})(v,\cdot)\Big|_{B^{ -\beta+\textcolor{black}{\vartheta}\Gamma}_{p',1}}\Big|\brho_{t,\mu}^{\varepsilon_k}(v,\cdot)\Big|_{B^{-\beta+ \Gamma}_{p',1}}\Big)\notag\\
 &&
\times 
\Big| p^{\alpha}_{s-v}\Big|_{B^{-\beta}_{1,1}}
 .
\end{eqnarray*}
Once the integration by parts is performed the analysis involving the singularities corresponds to the previous one for $ \beta=0$.
\end{proof}
}
Let us now prove that the limit point in the previous lemma is a distributional solution to the Fokker-Planck equation \eqref{NL_PDE_FK} and also satisfies the Duhamel representation \eqref{main_DUHAMEL}.
\begin{lemme}\label{lem_ex_duha_SPKRcase_BIS}
Assume \eqref{cond_coeff_SPKR} or \eqref{THE_COND_CI}  is in force. {
Let ${\color{black}S  \le \mathcal T_1}$ with $\mathcal T_1$ as in Lemma \ref{lem_quad_gronv_C1}}. Let $(\varepsilon_k)_{k \ge 1}$ be a decreasing sequence and $\brho_{t,\mu}$ 
be the limit point of $(\brho_{t,\mu}^{\varepsilon_k})_k$ exhibited in Lemma \ref{SECOND_STAB_BIS}, i.e. \eqref{STRONG_CONV_2} holds. Then $\brho_{t,\mu}$ satisfies the Fokker-Planck equation \eqref{NL_PDE_FK} \textcolor{black}{(in the sense of distributions, see \eqref{DEF_DISTR})} and also enjoys the Duhamel type representation \eqref{main_DUHAMEL}.
\end{lemme}
\begin{proof}
The claim can be obtained by replicating the arguments of Lemma 9 in \cite{chau:jabi:meno:22-1}. For the sake of completeness, we just draw the essential points of the demonstration, leaving further details to a line-by-line reading of \cite{chau:jabi:meno:22-1}.  Let $\brho_{t,\mu}$  be the cluster point given by Lemma \ref{SECOND_STAB_BIS}. 


 Starting from \textcolor{black}{the} weak formulation associated with $\brho_{\textcolor{black}{t,\mu}}^{\varepsilon_k}$, $\brho_{t,\mu}$ satisfies, for any $\varphi \in \mathcal D([t,S)\times \R^d)$ and $k\in \mathbb N $,
\begin{eqnarray*}
\int_t^S \int \brho_{t,\mu}(s,x)\big(\partial_s\varphi+\mathcal B_{\brho_{t,\mu}}\cdot \nabla \varphi+\textcolor{black}{L^\alpha}(\varphi)\big)(s,x)\,ds\,dx=-	\int \varphi(\textcolor{black}{t},x)\mu(dx)+\Delta^1_{\brho_{t,\mu},\brho^{\varepsilon_k}_{t,\mu}} \textcolor{black}{(\varphi)}+\Delta^2_{\brho_{t,\mu},\brho^{\varepsilon_k}_{t,\mu}}\textcolor{black}{(\varphi)}
\end{eqnarray*}
for
\begin{eqnarray*}
\Delta^1_{\brho_{t,\mu},\brho^{\varepsilon_k}_{t,\mu}}\textcolor{black}{(\varphi)}&=&\int_t^S \int \big(\brho_{t,\mu}(s,x)-\brho^{\varepsilon_k}_{t,\mu}(s,x)\big)\big(\partial_s\varphi+\textcolor{black}{L^\alpha}(\varphi)\big)(s,x)\,ds\,dx
\end{eqnarray*}
and 
\begin{eqnarray*}
\Delta^2_{\brho_{t,\mu},\brho^{\varepsilon_k}_{t,\mu}}\textcolor{black}{(\varphi)}&=&\int_t^S \int \big(\mathcal B_{\brho_{t,\mu}} \brho_{t,\mu}(s,x)-\mathcal B^{\varepsilon_k}_{\brho^{\varepsilon_k}_{t,\mu}}\brho^{\varepsilon_k}_{t,\mu}(s,x)\big)\cdot \nabla \varphi(s,x)\,ds\,dx\textcolor{black}{.}
\end{eqnarray*}
As it is clear that $(\partial_s+\textcolor{black}{L^\alpha})\varphi\in L^{1}([t,S),L^\infty\textcolor{black}{(\R^d)})$, we readily get from \eqref{STRONG_CONV_2} that $|\Delta_{\brho_{t,\mu},\brho_{t,\mu}^{\varepsilon_k}}^1(\varphi)|\underset{k}{\longrightarrow} 0.$

For the second term $\Delta_{\brho,\brho^{\varepsilon_k}}^2(\varphi)$, we simply have to reproduce the computations of the previous Lemma observing that the heat kernel, which induced time singularity is now replaced by a time-space smooth function. Write indeed, 
\begin{eqnarray*}
&&| \Delta_{\brho_{t,\mu},\brho_{t,\mu}^{\varepsilon_k}}^{2}(\varphi)|\\
 &\le& 
\Big|\int_t^S \int \big(\brho_{t,\mu}-\brho^{\varepsilon_k}_{t,\mu}\big)(s,x)\big(\mathcal B_{\brho_{t,\mu}}\cdot \nabla \varphi\big)(s,x)\,ds\,dx\Big| \\
&&+\Big|\int_t^S \int  \brho^{\varepsilon_k}_{t,\mu}(s,x) \Big(\big(\mathcal B^{\varepsilon_k}_{\brho^{\varepsilon_k}_{t,\mu}}-\mathcal B_{\brho_{t,\mu}}\big)\cdot \nabla \varphi\Big)(s,x)\,ds\,dx\Big|\\
&\le & |\brho_{t,\mu}-\brho^{\varepsilon_k}_{t,\mu}|_{L^\infty((t,S],L^1)}\int_t^S\,ds|\mathcal B_{\brho_{t,\mu}}(s,\cdot)\cdot \nabla \varphi(s,\cdot)|_{L^\infty(\mathbb R^d)}\\
&&+|\brho^{\varepsilon_k}_{t,\mu}|_{L^\infty((t,S],L^1)}\int_t^S\,ds |\big(\big(\mathcal B^{\varepsilon_k}_{\brho^{\varepsilon_k}_{t,\mu}}-\mathcal B_{\brho_{t,\mu}}\big)\cdot \nabla \varphi\big)(s,\cdot)|_{L^\infty(\mathbb R^d)}\\
&\le &|\nabla \varphi|_{L^\infty}\Bigg(|\brho_{t,\mu}-\brho^\varepsilon_{t,\mu}|_{L^\infty((t,S],L^1)}\int_{t}^S ds |b(s,\cdot)|_{B_{p,q}^\beta} |\brho_{t,\mu}(s,\cdot)|_{B_{p',1}^{-\beta}}\\
&&+\int_t^S ds \Big(|(b-b^{\varepsilon_k})(s,\cdot)|_{B_{p,q}^{\beta- \Gamma}}|\brho_{t,\mu}^{\varepsilon_k}(s,\cdot)|_{B_{p',1}^{-\beta+\Gamma}} +|b(s,\cdot)|_{B_{p,q}^\beta}|(\brho_{t,\mu}^{\varepsilon_k}-\brho_{t,\mu})(s,\cdot)|_{B_{p',1}^{-\beta+\vartheta \Gamma}}\Big) \Bigg)\\
&\le &C|\nabla \varphi|_{L^\infty}\Bigg(|\brho_{t,\mu}-\brho^\varepsilon_{t,\mu}|_{L^\infty((t,S],L^1)} |b|_{L^r(B_{p,q}^\beta)}\Big(\int_t^S ds(s-t)^{-\theta r'}\Big)^{\frac 1{r'}}\\
&&+|(b-b^{\varepsilon_k})|_{L^{\bar r}(B_{p,q}^{\beta- \Gamma})}\Big(\int_t^S ds (s-t)^{-\theta \bar r'}\Big)^{\frac 1{\bar r'}} \\
&&+|b|_{L^r(B_{p,q}^\beta)}\Big(\sup_{s\in (t,S]}\w_{\theta}^{\theta_2}(s-t)|(\brho_{t,\mu}^{\varepsilon_k}-\brho_{t,\mu})(s,\cdot)|_{B_{p',1}^{-\beta+\vartheta\Gamma}}\Big)\Big(\int_t^S ds(s-t)^{-\theta \bar r'}\Big)^{\frac 1{r'}} \Bigg),
\end{eqnarray*}
using Lemma \ref{lem_quad_gronv_C1} for the last inequality. We eventually derive from Lemma \ref{SECOND_STAB_BIS} and Proposition \ref{PROP_APPROX}
that $\Delta_{\brho,\brho^{\varepsilon_k}}^{2}(\varphi)\rightarrow 0$ and we conclude  that $\brho_{t,\mu}$ satisfies \eqref{NL_PDE_FK} in a distributional sense.

\textcolor{black}{To establish \eqref{main_DUHAMEL} we can start from the Duhamel representation which holds for the density associated with the mollified coefficients (which is proved e.g.  in Lemma 3 of \cite{chau:jabi:meno:22-1}). Namely, the above equation \eqref{main_MOLL} which we now recall for clarity.
\begin{align*}
\brho^{\varepsilon}_{t,\mu}(s,x)= p^\alpha_{s-t}\star \,\mu(x)-\int_t^s dv\, \Big(\nabla p^\alpha_{s-v}\star \{ \mathcal B_{\brho^{\varepsilon}_{t,\mu}}^{\varepsilon}(v,x) \brho^{\varepsilon}_{t,\mu}(v,x)\}\Big).
\end{align*}
The arguments used in Lemma \ref{SECOND_STAB_BIS} to establish the $L^1$ convergence then give \eqref{main_DUHAMEL} for the limit.}


The final control stated is also a direct consequence of Lemma \ref{SECOND_STAB_BIS}.

\end{proof}

\begin{lemme}\label{lem_uniq_SPKRcase_BIS} Under the assumptions and with the notations of Lemma \ref{SECOND_STAB_BIS}, the equation \eqref{main_DUHAMEL} admits at most one solution  in $L_{{\rm w}_{\theta}}^\infty((t,S],B_{p',1}^{-\beta+\vartheta\Gamma})) $
 for any $S\le\Tspkr$.
\end{lemme}
\begin{proof}  
The proof is here  very close  to the stability analysis performed in Lemma \ref{SECOND_STAB_BIS} (see also Lemma 10 in \cite{chau:jabi:meno:22-1} for similar issues). As in the indicated lemma we present the proof under \eqref{cond_coeff_SPKR}. We refer to 
the end of the proof of Lemma \ref{SECOND_STAB_BIS} for the modifications under \eqref{THE_COND_CI}. 
 
Assume that $\brho^{(1)}_{t,\mu}$ and $\brho^{(2)}_{t,\mu}$ are two possible solutions to \eqref{main_DUHAMEL}. Then, for a.e. $t\le  s\le S$,  $y$ in $\R^d$,
\begin{eqnarray*}
\brho^{(1)}_{t,\mu}(s,y) -\brho^{(2)}_{t,\mu}(s,y)
 = -\int_{t}^s dv \Big[\{{\mathcal B}_{\brho^{(1)}_{t,\mu}}(v,\cdot) \brho^{(1)}_{t,\mu}(v,\cdot)-
 {\mathcal B}_{\brho^{(2)}_{t,\mu}}(v,\cdot) \brho^{(2)}_{t,\mu}(v,\cdot)\} \star \nabla p_{s-v}^{\alpha}\Big] (y).
\end{eqnarray*}
Similarly to \eqref{PREAL_STAB_CAS_FACILE}
write
\begin{eqnarray}
 &&	\Big|\Big(\mathcal B_{\brho_{t,\mu}^{(1)}}(v,\cdot)   \brho_{t,\mu}^{(1)}(v,\cdot)-\mathcal B_{\brho_{t,\mu}^{(2)}}(v,\cdot)   \brho_{t,\mu}^{(2)}(v,\cdot)\Big) \star \nabla p^{\alpha}_{s-v}\Big|_{B^{-\beta+\vartheta\Gamma}_{p',1}} \notag\\
  &\le & C\left( |b(v,\cdot)|_{B_{p,q}^\beta}\Big|(\brho_{t,\mu}^{(1)}-\brho_{t,\mu}^{(2)})(v,\cdot)\Big|_{B^{-\beta+ \vartheta\Gamma}_{p',1}}
  \Big|\brho_{t,\mu}^{(1)}(v,\cdot)\Big|_{B^{ \vartheta\Gamma}_{p',1}}\right.\notag\notag\\
 &&\left.+|b(v,\cdot)|_{B^{\beta}_{p,q}}
 \Big|\brho_{t,\mu}^{(2)}(v,\cdot)\Big|_{B^{ -\beta+\vartheta\Gamma}_{p',q'}} \Big|(\brho_{t,\mu}^{(1)}-\brho_{t,\mu}^{(2)})(v,\cdot)\Big|_{B^{-\beta+\vartheta \Gamma}_{p',1}} \right)\Big|\nabla p^{\alpha}_{s-v}\Big|_{B^{-\beta}_{1,1}},\label{PREAL_TO_UNIQUENESS_FK_BACK}
\end{eqnarray}
where $\vartheta=1 $ if $p,q,r<+\infty$, $\vartheta\in (0,1) $ otherwise. From the H\"older inequality and Lemma \ref{lem_quad_gronv_C1}, we thus derive:
\begin{align*}
|(\brho_{t,\mu}^{(1)}-\brho_{t,\mu}^{(2)})(s,\cdot)|_{B_{p',1}^{-\beta+\vartheta\Gamma}} &\le & C 
|b|_{L^r(B_{p,q}^\beta)}
\left(\int_t^s dv\frac{\Big(\sup_{r\in (t,v]}(r-t)^{\theta}|(\brho_{t,\mu}^{(1)}-\brho_{t,\mu}^{(2)})(r,\cdot)|_{B_{p',1}^{-\beta+\vartheta\Gamma}}\Big)^{r'}}{(v-t)^{2r'\theta}
(s-v)^{r'\frac {1-\beta}\alpha}
}
\right)^{\frac {1}{r'}}.
\end{align*}
The result then follows {\color{black}by} multiplying the above l.h.s. by $(s-t)^{\theta} $, taking the supremum on $(t,S] $ and exploiting the small time assumption (recalling that $\frac 1{\textcolor{black}{  r'}} -\frac{1-\beta}{\alpha}-\theta>0$).
%
%
\end{proof}

\section[Well posedness of the non-linear McKean Vlasov SDE]{Well posedness of the non-linear McKean Vlasov SDE. From the Fokker-Planck equation to the non-linear martingale problem}\label{sec_WP_SDE}
We will here first focus on the integrability properties of the non-linear drift.  
Namely, we have the following result:
\begin{lemme}\label{LEMME_FOR_TIGHTNESS_AND_IDENTIFICATION}
Assume that \eqref{cond_coeff_SPKR} or \eqref{THE_COND_CI} holds. Then, the mollified non-linear drift $\mathcal B^\varepsilon_{\brho_{t,\mu}^\varepsilon}$ in \eqref{main_smoothed} is in $L^{r_0}((t,T],B_{\infty,1}^0) $ with $r_0 \in (\frac{\alpha}{\alpha-\textcolor{black}{(1-\beta\I_{\beta\in (-1,0]})}},\frac{r}{1+r\theta})$. 
\textcolor{black}{Setting $\Xi:=\frac 1{r_0}-\big(\frac 1{r}+\theta\big)>0$}\footnote{\textcolor{black}{Importantly, the positivity of $\Xi $  follows from the definition of $\theta$ in \eqref{def_thetas} and the range in which we take $\textcolor{black}{r_0}$.}}, there exists $C\ge 1$ s.t. for all $\varepsilon>0 $:
\begin{eqnarray}\label{unifestim_drift_gencase_CR}
\forall t\le S\le T,\,|\mathcal B^\varepsilon_{\brho_{t,\mu}^\varepsilon}|_{L^{r_0}((t,S],B^0_{\infty,1})}\le C(S-t)^{\textcolor{black}{\Xi}
}|b|_{L^r(B^{\beta}_{p,q})}.
\end{eqnarray}
\end{lemme}
\begin{proof}
From the Young inequality \eqref{YOUNG}, one gets for all $s\in (t,T] $:
$$|\mathcal B^\varepsilon_{\brho_{t,\mu}^\varepsilon}(s,\cdot)|_{B^0_{\infty,1}}\le \textcolor{black}{c_{\mathbf{Y}}}|b^\varepsilon(s,\cdot)|_{B_{p,q}^\beta} 
|\brho_{t,\mu}^\varepsilon(s,\cdot)|_{B_{p',q'}^{-\beta}}.
$$
Take now $r_0$ as indicated, then $r>r_0 $ and use the H\"older inequality, \textcolor{black}{$L^{r_0}:L^r-L^{(r_0^{-1}-r^{-1})^{-1}} $ (with the usual convention if $r=\infty$)}, to derive:
\begin{align*}
|\mathcal B^\varepsilon_{\brho_{t,\mu}^\varepsilon}|_{L^{r_0}((t,S],B^0_{\infty,1})}\le& \textcolor{black}{c_{\mathbf{Y}}}|b^\varepsilon|_{L^r(B^{\beta}_{p,q})}\Big(\int_t^S ds |\brho_{t,\mu}^\varepsilon(s,\cdot)|_{B_{p',q'}^{-\beta}}^{r_0\frac{r}{r-r_0}}\Big)^{\frac 1{r_0}-\frac {1}r}\\
\le& \textcolor{black}{c_{\mathbf{Y}}}|b^\varepsilon|_{L^r(B^{\beta}_{p,q})} \|\brho_{t,\mu}^\varepsilon\|_{L_{{\rm w}_\theta}^\infty((t,S],B_{p',1}^{-\beta+\vartheta\Gamma})}\Big(\int_t^S ds (s-t)^{-\theta(r_0\frac{r}{r-r_0})}\Big)^{\frac 1{r_0}-\frac {1}r}\\
\le &C |b|_{L^r(B^{\beta}_{p,q})} (S-t)^{\frac 1{r_0}-\frac {1}r-\theta},
\end{align*}
using  Proposition \ref{PROP_APPROX} 
and Lemma \ref{lem_quad_gronv_C1}  for the last but one inequality.

\end{proof}

\paragraph{Existence results}
We first \textcolor{black}{specify} \textcolor{black}{the} canonical space introduced in \textcolor{black}{Section \ref{sec_strategy}}
$$\Omega_\alpha:=\begin{cases} \mathcal C([t,\textcolor{black}{S}];\R^d), \ \alpha=2\textcolor{black}{,}\\
\mathbb D([t,\textcolor{black}{S}];\R^d),\ \alpha \in (1,2).
\end{cases}$$
\textcolor{black}{A probability measure $\mathbf P$ on the canonical space $\Omega_\alpha$} solves the non-linear martingale problem related to \eqref{main} on $[t,\textcolor{black}{S}] $ if:
\begin{itemize}
\item[(i)] $\mathbf P\circ x(t)^{-1}=\mu$;
\item[(ii)] for a.a. $s\in(t,\textcolor{black}{S}]$, $\mathbf P\circ x(s)^{-1}$ is absolutely continuous w.r.t. Lebesgue measure and its density belongs to $L_{{\rm w}_\theta}^{\infty}((t,\textcolor{black}{S}],B_{p',1}^{-\beta})$.
\item[(iii)] for all $f$ in $\mathcal C^1([t,\textcolor{black}{S}],\mathcal C_{\textcolor{black}{0}}^2(\mathbb R^d))$, the process
\begin{equation}\label{MP_NL}
\Bigg\{f(s,x(s))-f(t,x(t))-\int_t^s \Big(\textcolor{black}{\partial_v f(v,x(v))}+\mathcal B_{\mathbf P\circ x(v)^{-1}}(v,x(v)) \cdot {\color{black}\nabla}
 f(v,x(v))+L^\alpha(f)(v,x(v))\Big)\,dv\Bigg\}_{t\le s\le \textcolor{black}{S}},\tag{${\rm MP}_{{\rm NL}}$}
\end{equation}
is a $\mathbf P$ martingale. 
\end{itemize}

\textcolor{black}{We recall that the  smoothness properties required on the marginal laws of the canonical process {\color{black}under} $\mathbf P$  allow to define  almost everywhere the non-linear drift in \eqref{MP_NL}. Anyhow, this latter might still have time singularities, which prevents  from using \textit{standard} results to ensure well-posedness}.

\textcolor{black}{From now on, for} $\mathbf P( x(v)\in dx):=\mathbf P_{t,\mu}(v,dx)=\brho_{t,\mu}(v,x)dx $,  \textcolor{black}{we denote with a slight abuse of notation} $\mathcal B_{\mathbf P\circ x(v)^{-1}}(r,\cdot)=\mathcal B_{\brho_{t,\mu}}(v,\cdot) $.

\begin{prop}\label{prop_ExistenceMain} Let \eqref{cond_coeff_SPKR} or \eqref{THE_COND_CI} be in force.  Let $(\mathbf P^\varepsilon)_{\varepsilon >0}$ denote the solution to the non-linear martingale problem related to \eqref{main_smoothed}. Then any limit point of a converging subsequence $(\mathbf P^{\varepsilon_k})_{k}, \ \varepsilon_k\underset{k}{\rightarrow}0 $, in $\mathcal P(\Omega_\alpha)$ equipped with its weak topology,  solves the non-linear martingale problem related to \eqref{main}. 
\end{prop}

\begin{proof}[Proof of Proposition \ref{prop_ExistenceMain}] We prove tightness and then prove any converging subsequence solves the non-linear martingale.

\noindent\textit{Tightness.} From the Aldous tightness criterion (see e.g. \cite[Theorem 16.10]{Billingsley-99}) if $\alpha<2$ or \textcolor{black}{the Kolmogorov one if $\alpha=2 $} \cite[Theorem 7.3]{Billingsley-99}, in the current additive noise setting,   the tightness of $(\mathbf P^\varepsilon)_{\varepsilon >0}$ follows from the uniform (w.r.t. $\varepsilon$) almost-sure continuity of $s\mapsto \int_t^s \mathcal B^\varepsilon_{\brho^\varepsilon_{t,\mu}}(v,X^{\varepsilon,t,\mu}_v) dv$. Inequality \eqref{unifestim_drift_gencase_CR} readily implies this property.

\noindent\textit{Limit points.} Let $(\mathbf P^{\varepsilon_k})_k$ be a converging subsequence and denote by $\mathbf P$ its limit. Additionally to the weak convergence of $(\mathbf P^{\varepsilon_k})_k$ toward\textcolor{black}{s} $\mathbf P$,  Lemma \ref{SECOND_STAB_BIS} also gives that the 
{\color{black}marginal distributions} $\big(\mathbf P^{\varepsilon_k}_{t,\mu}(s,dx)=\brho^{\varepsilon_k}_{t,\mu}(s,x)\,dx\big)_k$ strongly \textcolor{black}{converge}  toward\textcolor{black}{s} $\mathbf P_{t,\mu}(s,dx)=\brho_{t,\mu}(s,x)\,dx$ in $L_{{\rm w}_\theta}^{\infty}((t,\textcolor{black}{S}],B^{-\beta}_{p',1})$. Following the proof of Lemma \ref{LEMME_FOR_TIGHTNESS_AND_IDENTIFICATION}, this strong convergence also yields the convergence of $(\mathcal B^{\varepsilon_k}_{\brho^{\varepsilon_k}_{t,\mu}})_k$ toward\textcolor{black}{s} $\mathcal B_{\brho_{t,\mu}}$ in $L^{r_0}(\textcolor{black}{(t,S]},L^\infty)$. Indeed, for $r_0$ as in the quoted lemma:
\begin{eqnarray*}
&&|\mathcal B^{\varepsilon_k}_{\brho^{\varepsilon_k}_{t,\mu}}-\mathcal B_{\brho_{t,\mu}}|_{L^{r_0}(\textcolor{black}{(t,S]},L^\infty)}\textcolor{black}{\underset{\eqref{EMBEDDING}}{\le}}\textcolor{black}{C} |\mathcal B^{\varepsilon_k}_{\brho^{\varepsilon_k}_{t,\mu}}-\mathcal B_{\brho_{t,\mu}}|_{L^{r_0}(\textcolor{black}{(t,S]},B_{\infty,1}^0)}\\
&\le& C\Bigg(|\brho^{\varepsilon_k}_{t,\mu}|_{L_{{\rm w}_\theta}^{\infty}((t,\textcolor{black}{S}]B^{-\beta\textcolor{black}{+\vartheta\Gamma}}_{p',1})}|b^{\varepsilon_k}-b|_{L^{\textcolor{black}{\check r}}(B^{\beta\textcolor{black}{-\vartheta\Gamma}}_{p,q})}(T-t)^{\check \Theta}\\
&&+
|b|_{L^r(B^{\beta}_{p,q})}\|\brho_{t,\mu}^{\varepsilon_k}-\brho_{t,\mu}\|_{L_{{\rm w}_\theta}^\infty((t,\textcolor{black}{S}],B_{p',1}^{-\beta+\vartheta\Gamma})}(T-t)^\Theta\Bigg)\underset{k}\to 0,
\end{eqnarray*}
\textcolor{black}{with $\check r=r ,  \textcolor{black}{\check \Xi}=\textcolor{black}{\Xi}$ if $r<+\infty$ and any finite $\check r$ large enough if $r=+\infty $ and $ \textcolor{black}{\check \Xi}=\frac 1{r_0}-\frac {1}{\check r}-\theta, \check r>r_0 $ in that case. Convergence now follows from Proposition \ref{PROP_APPROX} and Lemmas \ref{lem_quad_gronv_C1} and \ref{SECOND_STAB_BIS}}.

As a direct consequence of the above bound, we get that for all $\Phi\in \textcolor{black}{C_0^\infty}((t,T)\times \R^d,\R^d)$,
\[
\lim_{k\rightarrow \infty}\int_t^T \int \mathcal B^{\varepsilon_k}_{\brho^{\varepsilon_k}_{t,\mu}}(s,x)\cdot \Phi(s,x)\,dx\,ds=
\int_t^T \int \mathcal B_{\brho^{}_{t,\textcolor{black}{\mu}}}(s,x)\cdot \Phi(s,x)\,dx\,ds\textcolor{black}{,}
\]
{\color{black} which in turn allows to obtain, together with \eqref{unifestim_drift_gencase_CR}, dominated convergence Theorem and standard computations\footnote{{\color{black}Note indeed that from the strict inequalities in Lemma \ref{LEMME_FOR_TIGHTNESS_AND_IDENTIFICATION}, it can be shown that the drift lies in $L^{r_0}((t,S],B^e_{\infty,1})$, for some $0<e<<1$ and is thus a.e. continuous in space.}} 
that for any $0\le t_1\le \cdots \le t_i\le \cdots \le t_n\le t\le s\le T$, $\Psi_1,\cdots,\Psi_n$ continuous bounded, and 
for any $\phi$ of class $ C^2_{\textcolor{black}{0}} \textcolor{black}{(\R^d,\R)}$, }
\[
\mathbb E_{\mathbf P^{\varepsilon_k}}\!\!\left[\Pi_{i=1}^n\Psi_i(x(t_i))\!\int_t^s\!\mathcal B^{\varepsilon_k}_{\brho^{\varepsilon_k}_{t,\mu}}(v,x(v))\cdot \nabla \phi(x(v))\,dv\right]\to_k
\!\!\mathbb E_{\mathbf P^{}}\left[\Pi_{i=1}^n\Psi_i(x(t_i))\!\int_t^s\!\mathcal B_{\brho^{}_{t,\textcolor{black}{\mu}}}(v,x(v))\cdot \nabla \phi(x(v))\,dv\right].
\]
This \textcolor{black}{exactly means} that $\mathbf P$ solves the non-linear martingale problem related to \eqref{main}.
\end{proof}

\paragraph{Weak uniqueness results}
\begin{prop}[Uniqueness result]\label{prop_UniquenessMain_gencas} Under the assumption \eqref{cond_coeff_SPKR} or \eqref{THE_COND_CI}, for any $\mu$ in $ \mathcal P(\R^d)\cap B_{p_0,q_0}^{\beta_0}(\R^d)$ and \textcolor{black}{$0\le t<S<\mathcal T_1$, with $\mathcal T_1 $ as in Theorem \ref{main_thm_W}},  the SDE \eqref{main} admits at most one weak solution 
 s.t. its marginal laws $\big(\bmu^{t,\mu}_s(\cdot)\big)_{s\in [t,T]} $ have a density for a.e. $s$ in $(t,\textcolor{black}{S}]$, i.e. $\bmu^{t,\mu}_s(dx) = \brho_{t,\mu}(s,x)dx$
and
$\brho_{t,\mu}\in L_{{\rm w}_\theta}^{\infty}(\textcolor{black}{(t,S]},B_{p',1}^{-\beta+\vartheta \Gamma}), \ \vartheta,\ \Gamma$ as in Lemmas \ref{SECOND_STAB_BIS}, \ref{lem_unifesti_gencase2_RELOADED} respectively. 
\end{prop}
\begin{proof}\textcolor{black}{We first recall that from the uniqueness result for the Fokker-Planck equation (Lemma \ref{lem_uniq_SPKRcase_BIS} above), the non-linear drift $\mathcal B $ is uniquely determined}. It then suffices, following Proposition 13 in \cite{chau:jabi:meno:22-1}, to check that $\mathcal B $ viewed as the drift of a \textit{linear} SDE belongs to $L^{\sc}(L^\infty) $ with $1\le \sc \le \infty $ and
$$ \frac \alpha \sc<\alpha-1.$$
This is exactly Lemma 14 in the above reference, \textcolor{black}{which appeals to the results of \cite{CdRM-20} in the linear setting}. Now, taking $r_0\in \left( \frac{\alpha}{\alpha-\textcolor{black}{(1-\beta\I_{\beta\in (-1,0]})}},\frac{r}{1+r\theta}\right)$ as in Lemma \ref{LEMME_FOR_TIGHTNESS_AND_IDENTIFICATION} precisely gives this condition with $\sc=r_0 $. This completes the proof.
\end{proof}

\paragraph{Strong uniqueness results}

\begin{prop}\label{Prop_STRONG_SOL}
\textcolor{black}{Assume \eqref{INIT_DATA}}.
 For any $\mu$ in $ \mathcal P(\R^d)\cap B_{p_0,q_0}^{\beta_0}(\R^d)$ and \textcolor{black}{$0\le t<S<\mathcal T_1$, with $\mathcal T_1 $ as in Theorem \ref{main_thm_W}}, there exists a unique strong solution to \eqref{main} such that its law $\bmu^{t,\mu}$ belongs to $L_{{\rm w}_\theta}^{\infty}(\textcolor{black}{(t,S]},B_{p',1}^{-\beta+\vartheta\Gamma})$ and such that for a.e. $s$ in $(t,\textcolor{black}{S}]$, $\bmu^{t,\mu}_s(dx) = \brho_{t,\mu}(s,x)dx$ whenever 

\begin{itemize}
\item[$\bullet $] \textcolor{black}{if $\beta\in (-1,0] $ the condition \eqref{COND_SPKR_STRONG_INTRO}  of Theorem \ref{main_thm_S} holds}. Namely,
\begin{equation*}
\Bigg(2-\frac 32\alpha+\frac \alpha r+\Big[-\beta+\frac dp  -\textcolor{black}{\zeta_0}
\Big]\Bigg)\vee \Bigg( 1-\alpha+\frac \alpha r+\Big[-\beta+\frac dp-\textcolor{black}{\zeta_0} \Big]_+\Bigg)<\beta.
\end{equation*}
\item[$\bullet $]  \textcolor{black}{if $\beta=-1 $ the condition \eqref{COND_CRITIQUE_INTRO} of Theorem \ref{main_thm_S} holds}. Namely,
\begin{equation*}
\Bigg(2-\frac 32\alpha+\frac \alpha r+\Big[\frac dp  -\textcolor{black}{\zeta_0}\Big]\Bigg) \vee \Bigg(- \alpha + \frac \alpha r +\textcolor{black}{[1 + \frac dp   -\zeta_0]_+}\Bigg)
<\beta=-1.
\end{equation*}
\end{itemize}
\end{prop}

\begin{proof}
Similarly to the weak well-posedness we focus on the integrability properties of the drift viewed as the one of a 
 \emph{linear} version of the McKean-Vlasov SDE \eqref{main} where the law is frozen. This therefore amounts to prove that 
\begin{align}\label{DEF_NL_DR}
\mathfrak b: [0,\textcolor{black}{S}] \times \R^d \ni (s,x) \mapsto \mathfrak b(s,x) =  \int_{\R^d} b(s,x-y) \bmu_s^{t,\mu}(dy)=\int_{\R^d}{\color{black}dy \ }b(\textcolor{black}{s},x-y)\brho_{t,\mu}(s,y)
,
\end{align}
satisfies a Krylov and R\"ockner type condition, see \cite{kryl:rock:05} if $\alpha=2$, or the criterion in  \cite[Theorem 2.4]{xie:zhan:20} if $\alpha \in (1,2)$. 

If $\alpha=2 $ we have already proved in the weak-uniqueness part that $ \mathfrak b\in L^{r_0}(L^\infty)$ so that the Krylov R\"ockner criterion $2/\sc+d/\ell<1=\alpha-1 $ actually holds with $\sc=r_0,\ell=+\infty $. This gives strong uniqueness in the diffusive case under \eqref{cond_coeff_SPKR} or \eqref{THE_COND_CI}. 

Let us turn to $\alpha \in (1,2) $ which actually requires some smoothness properties  additionally to the integrability conditions.
Namely,
\begin{itemize}
\item[-] For $\textcolor{black}{\alpha\in (1,2)}$,  strong well-posedness holds whenever the drift $\mathfrak b$ defined in \eqref{DEF_NL_DR} satisfies $(I-\Delta)^{\gamma / 2}\mathfrak b\in L^{\sc}(L^{\ell})$, \textcolor{black}{or \textcolor{black}{equivalently}, $\mathfrak b \in L^{\sc} \textcolor{black}{(}H^{\gamma,\ell}\textcolor{black}{)} $}, where $ H^{\gamma,\ell}$ stands for the Bessel potential space and $\gamma,\ell$, $\sc$ satisfy
\begin{equation}\label{COND_II}
\gamma \in \Big(1-\frac \alpha 2,1\Big),\quad \ell \in \Big(\frac{2d}\alpha\vee 2,\infty\Big),\quad \sc \in \Big(\textcolor{black}{\frac{\alpha}{\alpha-1}},\infty\Big),\quad \frac{\alpha}{\sc}+\frac{d}{\ell}<\textcolor{black}{\alpha-1}.
\end{equation}
\end{itemize}
{\color{black}
Note that the condition on $\mathfrak b$ in $(ii)$ \textcolor{black}{will actually follow if we manage to prove} that $  \mathfrak b\in L^{\sc}(B_{\ell,\textcolor{black}{1}}^{ \gamma })$.
We indeed recall $B_{\ell,\textcolor{black}{2}}^{\gamma} \hookrightarrow H^{\gamma,\ell}$ when $\ell \ge 2 $ \textcolor{black}{(see e.g. \cite[Th. 2.5.6 p.88]{Triebel-83a}} and  from 
\eqref{BesovEmbedding} for all $\gamma>0$,   
$B_{\ell,1}^{ \gamma } \hookrightarrow B_{\ell,\textcolor{black}{2}}^{ \gamma }$. 

From \eqref{YOUNG} and for $\ell$ meant to be large (but finite) write:
\begin{align*}
|\mathfrak b(s,\cdot)|_{B_{\ell,1}^\gamma}\le |b(s,\cdot)|_{B_{p,q}^{\beta}}|\brho_{t,\mu}(s,\cdot)|_{B_{\textcolor{black}{\ell_2},q'}^{\gamma-\beta}},\text{ with } \textcolor{black}{\ell_2=p'\frac{1}{\frac{p'}{\ell}+1}}. 
\end{align*}
\textcolor{black}{Note that one can choose any $\ell_2<p'$ \textit{close} to $p'$, since again $ \ell$ is arbitrarily  large but finite}. 
Write from the H\"older inequality, similarly to the proof of Lemma \ref{LEMME_FOR_TIGHTNESS_AND_IDENTIFICATION}:
\begin{align*}
\int_t ^{\textcolor{black}{S}}{\color{black} ds \ }|\mathfrak b(s,\cdot)|_{B_{\ell,1}^\gamma}^{ \sc} 
\le \Big(\int_t^{\textcolor{black}{S}} {\color{black} ds \ }|b(s,\cdot)|_{B_{p,q}^{\beta}}^{ \sc a}
\Big)^{\frac 1a}\Big(\int_t^{\textcolor{black}{S}}{\color{black}ds \ }|\brho_{t,\mu}(s,\cdot)|_{B_{\textcolor{black}{\ell_2},q'}^{\gamma-\beta}}^{ \sc a'} 
\Big)^{\frac 1{a'}}, \ a^{-1}+(a')^{-1}=1.
\end{align*}
Again, since $b\in L^{r}((t,{\textcolor{black}{S}}],B_{p,q}^\beta) $ the natural choice consists in taking  $a \sc=r $ \textcolor{black}{giving} $1/{a'}=1-1/{a}=1-{ \sc}/{r} \iff a'= r/({r- \sc})$ and $\gamma = \vartheta\Gamma$. \textcolor{black}{However, pay attention that, since the integrability index $\ell_2$ is slightly smaller than $p'$ we considered for the previous analysis of the density, we need to modify a bit the regularity index and consider a slightly smaller $\gamma $ than the one indicated above}. Namely, we get
\begin{align}
\int_t ^{\textcolor{black}{S}}{\color{black} ds \ }|\mathfrak b(s,\cdot)|_{B_{\ell,1}^\gamma}^{ \sc} 
\le & |b|_{L^r(B_{p,q}^{\beta})}^{ \sc }\Big(\int_t^{\textcolor{black}{S}}{\color{black}ds \ }|\brho_{t,\mu}(s,\cdot)|_{B_{\textcolor{black}{\ell_2},q'}^{\gamma-\beta}}^{ \sc \frac{r}{r-\sc}} 
\Big)^{\frac 1{a'}}
\notag\\
\le& |b|_{L^r(B_{p,q}^{\beta})}^{ \sc } |\brho_{t,\mu}(s,\cdot)|_{L_{{\rm w}_{\bar \theta}}^\infty({\textcolor{black}{(t,S]}},B_{\textcolor{black}{\ell_2},1}^{-\beta+\vartheta\bar\Gamma})}^{\sc} \Big(\int_t^{\textcolor{black}{S}}ds (s-t)^{-{\bar \theta}( \sc \frac{r}{r-\sc})}
\Big)^{\frac 1{a'}}\notag\\
\le& C|b|_{L^r(B_{p,q}^{\beta})}^{ \sc } ({\textcolor{black}{S}}-t)^{\frac 1{a'}-\bar \theta\sc}=C|b|_{L^r(B_{p,q}^{\beta})}^{ \sc } (\textcolor{black}{S}-t)^{\sc\big( \frac 1{\sc}- (\frac 1r+\bar \theta)\big)},
\label{EST_DRIFT_FO_STABLE_SU}
\end{align}
for any $\ell_2<p' $ with $\bar \Gamma=\Gamma-\bar \eta,\ \bar \theta=\theta-\bar \eta,\ \bar \eta:=\bar \eta(p'-\ell_2)>0 $, going to 0 when $\ell_2 $ goes to $p'$ and  with \textcolor{black}{$\Gamma,\theta =\theta$ as in \eqref{def_gamma}, \eqref{def_thetas} respectively}. Indeed, we can reproduce the previous steps, starting from the proof of Lemma \ref{lem_unifesti_gencase2_RELOADED} for the equation with mollified coefficients, in order to take into consideration a slightly smaller integration index.

Write indeed:
\begin{trivlist}
\item[-] Under \eqref{cond_coeff_SPKR}
\begin{eqnarray*}
 	\Big|\Big(\mathcal B_{\brho_{t,\mu}^\varepsilon}^\varepsilon(v,\cdot) \cdot  \brho_{t,\mu}^\varepsilon(v,\cdot)\Big) \star \nabla p^{\alpha}_{s-v}\Big|_{B^{-\beta+\Gamma}_{\ell_2,1}} 
 \le
 C \Big|\Big(\mathcal B_{\brho_{t,\mu}^\varepsilon}^\varepsilon(v,\cdot) \cdot  \brho_{t,\mu}^\varepsilon(v,\cdot)\Big) \Big|_{B^{\Gamma}_{\rho_1,\infty}} \Big|\nabla p^{\alpha}_{s-v}\Big|_{B^{-\beta}_{\rho_2,1}}.
 \end{eqnarray*}
 \item[-] Under \eqref{THE_COND_CI}
 \begin{eqnarray*}
 &&
 	\Big|\Big(\mathcal B_{\brho_{t,\mu}^\varepsilon}^\varepsilon(v,\cdot) \cdot  \brho_{t,\mu}^\varepsilon(v,\cdot)\Big) \star \nabla p^{\alpha}_{s-v}\Big|_{B^{-\beta+\Gamma}_{\ell_2,1}} \\
 &
 \le
 & 
 C \Big[\Big|\Big(\mathcal B_{\brho_{t,\mu}^\varepsilon}^\varepsilon(v,\cdot) \cdot \nabla \brho_{t,\mu}^\varepsilon(v,\cdot)\Big) \Big|_{B^{\Gamma}_{\rho_1,\infty}} +\Big|\Big(\div(\mathcal B_{\brho_{t,\mu}^\varepsilon}^\varepsilon(v,\cdot))  \brho_{t,\mu}^\varepsilon(v,\cdot)\Big) \Big|_{B^{\Gamma}_{\rho_1,\infty}}\Big]\Big| p^{\alpha}_{s-v}\Big|_{B^{-\beta}_{\rho_2,1}},
 \end{eqnarray*}
 with $1+(\ell_2)^{-1} =\rho_1^{-1}+\rho_2^{-1}.$ 

 \end{trivlist}

 Apply now the product rule \eqref{PROD1} from Theorem \ref{thm_paraprod_2} . We get:
 \begin{trivlist}
\item[-] Under \eqref{cond_coeff_SPKR}:
  \begin{eqnarray*}
 \Big|\Big(\mathcal B_{\brho_{t,\mu}^\varepsilon}^\varepsilon(v,\cdot) \cdot  \brho_{t,\mu}^\varepsilon(v,\cdot)\Big) \star \nabla p^{\alpha}_{s-v}\Big|_{B^{-\beta+\Gamma}_{\ell_2,1}} 
 \le
 C\Big|\mathcal B_{\brho_{t,\mu}^\varepsilon}^\varepsilon(v,\cdot)\Big|_{B^{ \Gamma}_{\bar \ell_1,\infty}} \Big|\brho_{t,\mu}^\varepsilon(v,\cdot)\Big|_{B^{ \Gamma}_{\bar \ell_2,1}}
 \Big|\nabla p^{\alpha}_{s-v}\Big|_{B^{-\beta}_{\rho_2,1}}.
 \end{eqnarray*}
\item[-] Under \eqref{THE_COND_CI}
 \begin{eqnarray*}
 &&
 	\Big|\Big(\mathcal B_{\brho_{t,\mu}^\varepsilon}^\varepsilon(v,\cdot) \cdot  \brho_{t,\mu}^\varepsilon(v,\cdot)\Big) \star \nabla p^{\alpha}_{s-v}\Big|_{B^{-\beta+\Gamma}_{\ell_2,1}} \\
 &
 \le
 & 
 C \Big[\Big|\mathcal B_{\brho_{t,\mu}^\varepsilon}^\varepsilon(v,\cdot)\Big|_{B_{\bar \ell_1,\infty}^\Gamma} \Big| \nabla \brho_{t,\mu}^\varepsilon(v,\cdot) \Big|_{B^{\Gamma}_{\bar \ell_2,1}} +\Big|\div(\mathcal B_{\brho_{t,\mu}^\varepsilon}^\varepsilon(v,\cdot))\Big|_{B^\Gamma_{\bar \ell_1,\infty}} \Big| \brho_{t,\mu}^\varepsilon(v,\cdot) \Big|_{B^{\Gamma}_{\textcolor{black}{\bar \ell_2},1}}\Big]\Big| p^{\alpha}_{s-v}\Big|_{B^{-\beta}_{\rho_2,1}},
 \end{eqnarray*}

 \end{trivlist}
  with $\frac 1 {\rho_1}=\frac 1 {\bar \ell_1}+\frac 1 {\bar \ell_2} $. 
  
  In order to repeat the previous procedure, one needs to take $ \bar \ell_2=\ell_2$. On the other hand, the previous choice yields
  $$ 1+\frac 1{\ell_2}-\frac 1{\rho_2}=\frac1{\rho_1}=\frac 1{\ell_2}+\frac 1{\bar \ell_1}\Longrightarrow 1-\frac 1{\rho_2}=\frac 1{\bar \ell_1}.$$
To fit the previous estimates it then remains to apply the Young inequality \eqref{YOUNG}:
\begin{trivlist}
\item[-] Under \eqref{cond_coeff_SPKR}
 \begin{align*}
 	\Big|\Big(\mathcal B_{\brho_{t,\mu}^\varepsilon}^\varepsilon(v,\cdot) \cdot  \brho_{t,\mu}^\varepsilon(v,\cdot)\Big) \star \nabla p^{\alpha}_{s-v}\Big|_{B^{-\beta+\Gamma}_{\ell_2,1}} 
 \le &C |b^\varepsilon(v,\cdot)|_{B^{\beta}_{p,q}}
 \Big|\brho_{t,\mu}^\varepsilon(v,\cdot)\Big|_{B^{ -\beta+\Gamma}_{\ell_2,1}}\Big|\brho_{t,\mu}^\varepsilon(v,\cdot)\Big|_{B^{ \Gamma}_{ \ell_2,1}}
 \Big|\nabla p^{\alpha}_{s-v}\Big|_{B^{-\beta}_{\rho_2,1}}\\
 \le& C|b^\varepsilon(v,\cdot)|_{B^{\beta}_{p,q}}
 \Big|\brho_{t,\mu}^\varepsilon(v,\cdot)\Big|_{B^{ -\beta+\Gamma}_{\ell_2,1}}^2
 \Big|\nabla p^{\alpha}_{s-v}\Big|_{B^{-\beta}_{\rho_2,1}}.
 \end{align*}
 \item[-] Under \eqref{THE_COND_CI}
  \begin{align*}
 &
 	\Big|\Big(\mathcal B_{\brho_{t,\mu}^\varepsilon}^\varepsilon(v,\cdot) \cdot  \brho_{t,\mu}^\varepsilon(v,\cdot)\Big) \star \nabla p^{\alpha}_{s-v}\Big|_{B^{-\beta+\Gamma}_{\ell_2,1}} \\
  \le
 & 
 C 
 \Big(|b^\varepsilon(v,\cdot)|_{B^{\beta}_{p,q}}+\Big|\div(b^\varepsilon(v,\cdot))\Big|_{B^\beta_{p,q}}\Big)
 \Big|\brho_{t,\mu}^\varepsilon(v,\cdot)\Big|_{B^{ -\beta+\Gamma}_{\ell_2,1}}^2
 \Big| p^{\alpha}_{s-v}\Big|_{B^{-\beta}_{\rho_2,1}},
 \end{align*}
with parameters
$1+\frac1{\bar \ell_1}=\frac 1p+\frac 1{\ell_2}$.
 \end{trivlist}
We thus deduce that $\frac 1{\bar \ell_1}=1-\frac 1{\rho_2}=\frac 1p+\frac 1{\ell_2}-1 \iff \frac{1}{\rho_2}=2-(\frac 1p+\frac 1{\ell_2})<1$ since $\ell_2<p' $. The procedure can be summed up in the following way: in order to silghtly decrease the integrability index in the density estimate, we can slightly increase, through $\rho_2 $ the integrability which is asked on the gradient of the heat kernel. Using \eqref{SING_STABLE_HK} we would get similarly to \eqref{CI_EPS_AVANT_GR_QUADRA_PROOF_2} under \eqref{cond_coeff_SPKR} and \eqref{PREAL_MULT_C2} under \eqref{THE_COND_CI}, 
\begin{eqnarray*}
&&|\brho_{t,\mu}^\varepsilon(s,\cdot)|_{B^{ -\beta+\Gamma}_{\ell_2,1}}\\ 
&\le&  c_1 |\mu|_{\textcolor{black}{B^{\tilde \beta_1}_{\tilde p_1,\tilde q_1}}} (s-t)^{-\textcolor{black}{\frac{1}{\alpha}[\Gamma+\bar \vartheta_0]_+}} \\
&&+C\big(|b|_{L^{r}(B^{\beta}_{p,q})}+\textcolor{black}{|{\rm div}(b)|_{L^{r}(B^{\beta}_{p,q})}\I_{\beta=-1}}\big)  \Bigg(\int_t^s \frac{dv}{(s-v)^{r'\big(\frac {1-\beta\I_{\beta\in (-1,0]}}{\alpha}+\frac d\alpha(1-\frac 1{\rho_2}) \big)}}  | \brho_{t,\mu}^\varepsilon(v,\cdot)|_{B^{-\beta+\Gamma}_{\ell_2, 1}}^{2r'} \Bigg)^{\frac 1{r'}}\label{CI_EPS_AVANT_GR_QUADRA_PROOF_RELAXED},
\end{eqnarray*}
with
$$\textcolor{black}{\tilde p_1=\min(p_0,\ell_2),\tilde q_1=q_0\big(1\vee \frac{\ell_2'}{p_0'}\big),\tilde \beta_1=\beta_0\big(1\wedge \frac{p_0'}{\ell_2'}\big)},\ \textcolor{black}{\bar \vartheta_0}:= \Big[-\beta+\frac{d}{ \ell_2'} - \big(\beta_0+\frac{d}{p_0'}\big)\textcolor{black}{\Big(1\wedge \frac{p_0'}{\ell_2'} \Big)}\Big], \ \frac 1{\ell_2}+\frac 1{\ell_2'}=1.$$
Following, up to the previous modifications, the arguments of Lemmas \ref{lem_unifesti_gencase2_RELOADED} and \ref{SECOND_STAB_BIS} then yields to \eqref{EST_DRIFT_FO_STABLE_SU}.

The previous computations give that we can actually take $\ell_2 $ as close as we want to $p'$ but in order to have $\ell<+\infty$. Also, for $\ell_2 $ close to $p'$, $ \rho_2$ is close to one.
Similarly, we can indeed take $\gamma=\bar \vartheta\Gamma $ with $\bar \vartheta\in (0,1) $, hence as close to 1 as desired.

\textcolor{black}{To guarantee that strong well-posedness holds, it thus remains to establish the first constraint on $\gamma $ in \eqref{COND_II} holds}. To this end it suffices to check when the inequality
\begin{equation}\label{LOW_BOUND_ON_GAMMA}
\Gamma>1-\frac \alpha 2 
\end{equation}
is true.

\textcolor{black}{Set now $ \vartheta_0:= \frac{1}\alpha\Big[-\beta+\frac{d}{ p} - \big(\beta_0+\frac{d}{p_0'}\big)\textcolor{black}{\Big(1\wedge \frac{p_0'}p\Big)}\Big]=\frac{1}\alpha\Big[-\beta+\frac{d}{ p} - \zeta_0\Big]$}. We recall that $\Gamma  $ defined in \eqref{def_gamma} can be equivalently rewritten as
\begin{equation*}
\Gamma=
\alpha+\beta\I_{\beta\in (-1,0]}-1-\frac \alpha r -\Big(-\beta+\frac dp- \textcolor{black}{\zeta_0}
\Big)-\bar \eta=\alpha+\beta\I_{\beta\in (-1,0]}-1-\frac \alpha r-\alpha\textcolor{black}{\vartheta_0}-\bar \eta, 
\end{equation*}
for some $\bar \eta>0$ meant to be small.
%
The previous condition thus rewrites
$$\beta(1+\I_{\beta\in (-1,0]}) >2-\frac 3 2 \alpha+\frac \alpha r +\frac dp-\textcolor{black}{\zeta_0}
,$$
}
which is precisely the condition appearing in the statement.
\end{proof}

\section{Connection with some physical and biological models}
\label{CONNEC_WITH_MODELS}
\textcolor{black}{We discuss in this section some specific applications of Theorems \ref{main_thm_W} and \ref{main_thm_S}. We consider  models related  to turbulence theory and particle methods in Computational Fluid Dynamics 
as well as systems arising from recent trends in Biology. We particularly focus on the  three following equations}:
\begin{trivlist}
\item[$\bullet$] The (scalar) Burgers equation,
\item[$\bullet$] The two dimensional vortex equation for the incompressible Navier-Stokes equations\textcolor{black}{,}
\item[$\bullet$] The {\color{black} truncated} parabolic-elliptic Keller-Segel equations for chemotaxis. 
\end{trivlist}

\textcolor{black}{The common feature of those three equations is that they can be written as scalar valued singular non-linear transport-diffusion PDEs of the following form}:
\begin{equation}\label{QuasiLin}
\partial_{{\color{black}s}}u(s,\cdot)+{\rm{div}}\Big (u(s,x)F(u(s,\cdot))\Big)-L^\alpha u(s,\cdot)=0,\,u(0,\cdot)=u^0,
\end{equation}
where the driving 
$F$ \textcolor{black}{writes} $F({\color{black}s},u)=K\star u(x)$\textcolor{black}{,} for $K$ a time-homogeneous strongly concentrating kernel. 
Importantly the systematic interpretation of the solution $u$, understood in a distributional sense, as the time marginal distributions $\bmu^{t,\mu}$ of \eqref{main} (taking therein  $t=0$) requires some preliminary considerations which \textcolor{black}{start} 
with the initial condition $u^0$. Due to the conservative form of \eqref{QuasiLin}, to ensure that $u(s,\cdot)$ can be viewed as a probability measure, \textcolor{black}{this property needs to be fulfilled by the initial condition $u(0,\cdot) $}. While this naturally restricts the physical interpretation of the model involved, beyond this situation, the McKean-Vlasov interpretation of \eqref{QuasiLin}  become\textcolor{black}{s trickier}. We refer to \cite{Jourdain-00} or \cite{MarPul-82} for related issues, see also Section \ref{VORTEX} below.

{\color{black}The focus on the Burgers and Navier-Stokes \textcolor{black}{equations} will allow us to briefly revisit some predominant literature from the eighties and nineties.} For the sake of clarity, points of comparison with the literature will be essentially focused \textcolor{black}{on} probabilistic models, leaving purposely aside a more complete survey on the \textcolor{black}{PDE} analysis of \eqref{QuasiLin}. For the same reasons, precise comparisons on the smallness of the initial condition and possible range of the time-horizons $\mathcal T_1$ will be left aside.

\textcolor{black}{Importantly, we will consider for the Keller-Segel equations, for simplicity and to enter the setting of the previous sections, a \textit{truncated} version of the singular drifts involved in this model. This is done in order to focus on the singularity (at the origin) leaving aside the behavior at infinity, which anyhow does not lead to additional specific difficulties. In the case of the Navier-Stokes equation, we will propose two approches: the first one consists in using Lebesgue-Besov embeddings in order to write the (truncated) interaction kernel $K$ as an element of a suitable Lebesgue space (whose spatial integrability index violates the Krylov-R\"ockner condition); the second one fully uses our analysis as the interaction kernel $K$ is viewed as an element of a Besov space with regularity index $\beta=-1$ and free  divergence. We emphasize that, in comparison with the first approach, the second one allows to remove the truncation. This shows that the dichotomy in our conditions \textbf{(C1)} and \textbf{(C2)} really matters, as part of the structure can obviously not be captured by embeddings.}

\subsection{The Burgers equation} In its most 
popular formulation, the Burgers equation corresponds to the scalar non-linear PDE:
\begin{equation}\label{Burgers1D}
\partial_su(\textcolor{black}{s},\cdot)+\frac{1}{2}\partial_x(u(\textcolor{black}{s},\cdot))^2-{\color{black}\nu}\triangle u(\textcolor{black}{s},\cdot)=0,\ s>0,\,u(0,\cdot)=u^0(\cdot),
\end{equation}
where the solution $u$ models the speed motion of a viscous fluid evolving on the real line $\mathbb R$ under the joint action of a nonlinear transport operator $\frac{1}{2}\partial_x(u(s,\cdot))^2=u(s,\cdot)\partial_xu(s,\cdot)$ and the viscous dissipation $\triangle u(s,\cdot)$  - for consistency with \eqref{main}, the kinematic viscosity $\nu$ has to be set to ${\color{black}1/2}$.  While \eqref{Burgers1D} initially depicts a one-dimensional pressure-less model of Navier-Stokes equation, the Burgers equation nowadays applies in various disciplines such as aerodynamics, molecular biology, cosmology and traffic modelling.

Its \textcolor{black}{fractional} version\footnote{also called fractal in some related papers, see \cite{FunWoy-98}}:
\begin{equation}\label{FractalBurgers1D}
\partial_su(s,\cdot)+\frac{1}{2}\partial_x(u(s,\cdot))^2-L^\alpha u(s,\cdot)=0,\ s>0,\,u(0,\cdot)=u^0(\cdot),
\end{equation}
which substitutes the characteristic \textcolor{black}{\textit{heat} dissipation operator $\Delta $ with} 
the fractional \textcolor{black}{one} $L^\alpha$, 
presents a particular interest for hydrodynamics and statistical turbulence 
(we refer the interested reader to \cite{BiFuWo-98} and references therein for a brief exposure of the physical interest of \eqref{FractalBurgers1D} and to \cite{BaFiLeBr-79} for the impact of modified \textcolor{black}{fractional} dissipativity on recovering some characteristic scaling laws in turbulence).

\textcolor{black}{From \eqref{NL_PDE_FK}, it is easily seen that the McKean-Vlasov model related to the Burgers corresponds to} an interaction kernel given by $\frac{1}{2}\delta_{\{0\}}$ where $\delta_{\{0\}}$ denotes the Dirac mass at 0. 

Properly, the resulting McKean-Vlasov model formulates as
\begin{equation}
\label{McKeanBurgers1D}
X_{{\color{black}s}}^{\textcolor{black}{0,\mu}}=\textcolor{black}{\xi}+\frac{1}{2}\int_0^{{\color{black}s}} \brho(v,X_v^{\textcolor{black}{0,\mu}})\,dv+{\mathcal W}_{\textcolor{black}{s}},\,\,\textcolor{black}{\xi}\sim u^0,\,\brho(v,\cdot)=\,\text{p.d.f of }\text{Law}(X_v),
\end{equation}
or as
\begin{equation}
\label{McKeanBurgers1D-Alt}
X_{{\color{black}s}}^{\textcolor{black}{0,\mu}}=\textcolor{black}{\xi}+\frac{1}{2}\int_0^{{\color{black}s}} \tilde{\mathbb E}[\delta_{\{X_v^{\textcolor{black}{0,\mu}}-\tilde X_v^{\textcolor{black}{0,\mu}}\}}]\,dv+{\mathcal W}_{{\color{black}s}},
\end{equation}
\textcolor{black}{where $(\tilde X_s^{0,\mu})_{s\ge 0} $ has under $\tilde {\mathbb P} $ the same law as $(X_s^{0,\mu})_{s\ge 0} $},
 provided \textcolor{black}{the law of} $X_s^{0,\mu}$ is absolutely continuous at all time $s\in(0,T]$.
 For the case $\alpha=2$, these formulations have been formally introduced in the seminal paper \cite{McKean-66} (together with the interpretation of a model of the Boltzmann equation).
 
While \cite{GutKac-83}, \cite{CalPul-83} and \cite{OsaKot-85} focused on the particle approximation and associated propagation of chaos properties through analytic techniques, existence and uniqueness of a solution to \eqref{McKeanBurgers1D} was, to the best of our knowledge, firstly established in \cite{Sznitman-86}.The author obtained therein the existence and uniqueness of a strong solution, with $\brho$ in $L^2(\textcolor{black}{(}0,T]\times \R^d)$ for any arbitrary time horizon $T$ (see Theorems 2.5 and 4.1 \textcolor{black}{therein}
) under the condition $u^0\in L^1(\R)\cap L^\infty(\R)$. 

  In the \textcolor{black}{fractional} case $\alpha\in(1,2)$, similar weak existence result\textcolor{black}{s} have been successively established in \cite{FunWoy-98} in the case of a symmetric stable noise with $1<\alpha<2$ and $u^0$ lying in the Sobolev space $H^1(\mathbb R)=W^{1,2}(\mathbb R)$. Uniqueness is only established  for one-time marginal distributions in the class $\brho\in L^\infty(\textcolor{black}{(}0,T]\times L^2(\mathbb R))\cap  L^2(\textcolor{black}{(}0,T]\times H^1(\mathbb R))$ (see Theorems 2.1 and 3.1 of the indicated reference).
  
 Recall from \eqref{lem_proba_in besov} that a Dirac measure \textcolor{black}{belongs to} the Besov $ B^{-d/p'}_{p,\infty}$ for $p\in[1,\infty]$. \textcolor{black}{In particular}, the interaction kernel in \eqref{McKeanBurgers1D}  lies in the space $L^\infty(B^{-1/p'}_{p,\infty})$. 
 
\textcolor{black}{On the one hand, we thus derive}, \textcolor{black}{ from condition \eqref{cond_coeff_SPKR} \textcolor{black}{and the definition in \eqref{def_zetra}}, that  weak well-posedness holds if}
\begin{equation}\label{cond_coeff_SPKR_BURGERS}
  1-\alpha +\Big[1-\zeta_0\Big]_+<-\frac 1{p'},\ \zeta_0=\bigg(\beta_0+\frac 1{p_0'}\bigg)\textcolor{black}{\left( 1 \wedge \frac{p_0'}{p} \right)}\textcolor{black}{.}
\end{equation}
{\color{black} As this bound holds for any $p$ in $[1,\infty]$, one may optimize it w.r.t. this parameter and thus choose $p'=+\infty$ and $p=1$ which in turn yields that $\beta=0$. We obtain that weak well-posedness holds as soon as
\begin{equation}\label{cond_coeff_SPKR_BURGERS}
\Big[1-\zeta_0\Big]_+< \alpha-1,\ \zeta_0=\bigg(\beta_0+\frac 1{p_0'}\bigg)\textcolor{black}{.}
\end{equation}
This allows to obtain weak well-posedness as soon as $\beta_0 -1/p_0>1-\alpha$ and gives, in the Brownian case $\alpha=2$ e.g. $\beta_0=0,\, p_0>1$ or $\beta_0>0,\, p_0=1$. Note that this gives precisely the condition \eqref{cond_gencase} for the initial condition.}

\textcolor{black}{On the other hand}, 
\eqref{COND_SPKR_STRONG_INTRO} becomes
 \begin{equation}\label{COND_SPKR_STRONG_BURGERS}
 \textcolor{black}{\Big(2-\frac 3{2}\alpha+1-\zeta_0\Big)\vee (1-\alpha+[1-\zeta_0]_+)}<-\frac 1{p'}\textcolor{black}{,}
 \end{equation}
{\color{black}and we derive from 
\eqref{COND_SPKR_STRONG_INTRO}, \eqref{COND_SPKR_STRONG_BURGERS} that strong uniqueness will hold as soon as $\beta_0+1/p_0'>3(1-\alpha/2) \Leftrightarrow \beta_0>2-3\alpha/2 + 1/p_0$ \textcolor{black}{in the considered case}, i.e. as soon as the initial condition satisfies \eqref{cond_gencase_S}}. This situation specifically reduces in the Brownian case, to $\beta_0+1>1/p_0$.

In particular, we recover the result in \cite{Sznitman-86} where strong existence and uniqueness for the non linear equation were obtained for $u^0\in L^1\cap L^\infty  $. Indeed, from \eqref{EMBEDDING} which implies $L^1\cap L^\infty\hookrightarrow L^{p_0}\hookrightarrow B^0_{p_0,\infty}$ for $p_0\in[1,\infty]$, the result follows taking $p_0>1$
.

Also, {\color{black}i}n the fractal case, $\alpha\in(1,2)$, taking $p=2$ and  identifying $W^{1,2}=B^1_{2,2}$ (see e.g. \cite{Triebel-83a}, Theorem 2.5.6, p. 88, and Theorem 2.3.9, p. 61) which yields $\beta_0=1,p_0=q_0=2 $, we recover the existence result of \cite{FunWoy-98}, and add to this result, weak uniqueness of the corresponding SDE from  \eqref{cond_coeff_SPKR_BURGERS} and \eqref{COND_SPKR_STRONG_BURGERS} provided $\frac 3{2}<\alpha<2$ (noting that for the considered parameters $\zeta_0=(1+\frac 12)=\frac 32$ so that the additional constraint $\Big(2-\frac 3{2}\alpha+1-\zeta_0\Big)<-\frac 12\iff \frac 43<\alpha$ is less stringent than the one needed for weak uniqueness).
\subsection{The 
vortex equation in dimension 2}\label{VORTEX}
The vortex equations (or vorticity equations) model the rotational properties of an incompressible Newtonian turbulent 
fluid flow. The motions of such \textcolor{black}{a} fluid are described, at each time ${\color{black}s}$, and each point $x$ of $\R^d$, $d\in \{2,3\} $, through their macroscopic velocity $u({\color{black}s},x)=(u^{(1)}({\color{black}s},x),\cdots,u^{(d)}({\color{black}s},x))$, and their evolution, characterized by the incompressible Navier-Stokes equations:
\begin{align*}
&\partial_{{\color{black}s}}u({\color{black}s},x)+(u({\color{black}s},x)\cdot \nabla)u({\color{black}s},x)=-\nabla p+\frac 1{2}\triangle u({\color{black}s},x),\,\,\nabla \cdot u({\color{black}s},x)=0,\,\,\,{\color{black}s}\ge  0,\,x\in\R^d,\\
&u(0,x)=u^0(x).
\end{align*}
For simplicity, we purposely focus the presentation of the equations \textcolor{black}{on \textit{function}-solutions} $u$ rather than distributional solution\textcolor{black}{s}, and remain in the classical dissipative case $\alpha=2$.  We also, again, set the kinematic viscosity to $1/2$. The divergence free constraint $\nabla\cdot u({\color{black}s},x)=0$ reflects the incompressibility of the \textcolor{black}{flow}
and  $\nabla p$ stands for the gradient of the pressure
 acting on the fluid.  

In the case of a two-dimensional $(d=2)$ flow, the vorticity $w({\color{black}s},x):=\text{curl}(u)({\color{black}s},x)=\nabla\times u({\color{black}s},x)$ is a scalar function driven by the equation
\begin{equation}\label{Vortex2D}
\begin{aligned}
&\partial_{{\color{black}s}}w({\color{black}s},x)+(u({\color{black}s},x)\cdot \nabla_x)w({\color{black}s},x)=\frac 1{2}\triangle w({\color{black}s},x)\,{\color{black},\,s}\ge 0,\,x\in\R^2,\\
&w(0,x)=\nabla\times u^0(x).
\end{aligned}
\end{equation}
This case is thus simpler than the {\color{black}three  dimensional} 
case,  \textcolor{black}{$d=3 $}, for which the vorticity is a field, i.e. the previous equation must then be understood as a system of $3$ equations.
The vorticity equation somehow decouples the non-linearity and allows to directly handle the pressure term (the curl of a gradient is zero).  The original velocity $u$ can be recovered, up to an additive constant, from the vorticity $w$ using the identity $\Delta u=\begin{cases}
\nabla^\bot w=\left( \begin{array}{c} -\partial_2 w\\
\partial_1 w\end{array}\right),\ d=2,\\
-\nabla \times w, \ d=3
\end{cases}$ 
which  follows from the incompressibility property and \textit{formally} leads to the identity
\begin{equation*}
u({\color{black}s},x) =K* w({\color{black}s},x),
\end{equation*}
where $K$ stands for the Biot-Savart kernel. Its expression actually appears applying the adjoint of the operator $-\nabla ^\bot $ for $d=2$, resp. $\nabla \times  $ for $d=3$, to the Poisson kernel $\mathfrak P $, i.e. $u=\begin{cases}(-\Delta)^{-1} (-\nabla^\bot w),\\
(-\Delta)^{-1}\nabla \times w,\ d=3 
\end{cases}$. Namely, 
\begin{equation}
K(x)=\begin{cases}-\nabla^\bot \mathfrak P(x),\ d=2,\\
\nabla \mathfrak P (x)\times,\ d=3
\end{cases}\ {\rm with}\ \ \ \ \mathfrak P(x)=\left\{
\begin{aligned}
&-\frac{\log(|x|)}{\color{black}(2\pi)}\,\text{if}\,d=2,\\
&\frac{1}{{\color{black}4\pi}|x|^{d-2}}\,\text{if}\,d=3,
\end{aligned}
\right. \label{GEN_BIOT_SAVART}
\end{equation}
using as well the convention $K*w(x)=\int_{\R^3} \nabla \mathfrak P(x-y) \times w(y) dy $ when $d=3$. As for the Burgers equation, the \textcolor{black}{\textit{fractal/fractional}} version of the Navier-Stokes equation\textcolor{black}{s}, and by extension of \eqref{Vortex2D}:
\begin{equation}\label{FractalVortex2D}
\begin{aligned}
&\partial_{{\color{black}s}}w({\color{black}s},x)+(u({\color{black}s},x)\cdot \nabla_x)w({\color{black}s},x)=L^\alpha w({\color{black}s},x)\,{\color{black},\,s}\ge 0,\,x\in\R^2,\\
&w(0,x)=\nabla\times u^0(x),
\end{aligned}
\end{equation}
 presents a particular physical interest - we again refer to \cite{BaFiLeBr-79}, and to the exhaustive presentation in \cite{KavErc-22}.


{\color{black} Within the diffusive setting $
\alpha=2$,   Chorin exploited in \cite{Chorin-94}} the vorticity equations to develop particle methods -commonly known today as vortex methods- for the simulation of turbulent fluid flows.
In \cite{MarPul-82}, Marchioro and Pulvirenti addressed the link between the two-dimensional vortex equation and the McKean-Vlasov model. 
They later exploited this link to validate Chorin's particle method, introducing a smoothed {\color{black}mean-field} particle
 approximation of {\color{black}\eqref{Vortex2D}} 
where the Poisson kernel  $\mathfrak P$ is regularized at the neighborhood of $0$. {\color{black}Then,} 
the authors established that the time marginal empirical measures propagate chaos toward the {\color{black}solution to \eqref{Vortex2D}}\textcolor{black}{,} 
even in the zero viscosity limit. Osada \cite{Osada-86} established a similar result for a non vanishing viscosity {\color{black}and without any smoothing of $K$.} {\color{black}While these results apply to a peculiar probabilistic interpretation of \eqref{Vortex2D}, M\'el\'eard \cite{Meleard-00,Meleard-01} considered a McKean-Vlasov representation of \eqref{Vortex2D} \textcolor{black}{of} the form:}
\begin{equation}
\label{NS2D}
X^{0,\mu}_s=\xi+\int_0^{\textcolor{black}{s}} \textcolor{black}{\tilde{\mathbb E}[h_0(\tilde X_0^{0,\mu})K(X_r^{0,\mu}-\tilde X^{0,\mu}_r)]}\,dr+\mathcal W_s,
\end{equation}
\textcolor{black}{where again $(\tilde X_s^{0,\mu})_{s\ge 0} $ has under $\tilde {\mathbb P} $ the same law as $(X_s^{0,\mu})_{s\ge 0} $, and  $h_0$ is a pre-factor which derives from the formulation of the problem in the McKean setting, i.e. solutions are sought as density functions whereas the initial condition $w(0,\cdot)$ is not necessarily one. We again refer to \cite{Jourdain-00} for additional related details}.

{\color{black}The well-posedness of \eqref{NS2D}, along a quantitative particle approximation,}
was first established in terms of a nonlinear martingale problem in \cite{Meleard-00} (see Theorems 1.2 and 2.4 therein), for a non-negative initial condition $w_0$ in $L^1\cap L^\infty$ (see Theorems 1.2 and 2.4 of the same reference), and later extended  in \cite{Meleard-01} (Theorems 1.2 and 3.4), to the case of  $w(0,\cdot)$ being a Radon measure of the form $\nabla\times u^0$ where $u^0$ belongs to a suitable Lorentz space.


Let us now discuss what can be derived from the approach developed in the current work in this setting. Observe first from \eqref{GEN_BIOT_SAVART} that for the two-dimensional vortex equation
$$K(x)=\frac{(-x_2,x_1)}{2\pi(x_1^2+x_2^2)}.$$
Observe that in the distributional sense ${\rm div} K=0 $. Since $|K(x)|\le C/|x| $,  $K$ can be viewed as an element of  $L_{{\rm loc}}^{2-\epsilon}$ for $\epsilon>0$ arbitrary. Essentially the singularity of $K$ is localized in the neighborhood of $0$, with $K$ being smooth and bounded outside this region. 

This kernel can be viewed in two different ways: one leaving aside the divergence free property of $K$, exploiting the bound $|K(x)|\le C/|x| $ and embedding the kernel into a Lebesgue space; the second one embedding $K$ into $B^{-1}_{\infty,\infty}$ through a specific representation of $K$ (at play in e.g. \cite{jabi:wang:18}). Interestingly, each view leads to the same range of admissible initial distributions for weak and strong well-posedness. Although the second view enables to avoid truncation, this is somehow a superficial difference.

\begin{trivlist}
\item[-] {\color{black}\textbf{Truncated kernel and $\beta=0 $}.}
 We first focus our attention on a truncated version of the kernel given by the drift $b(x)=K(x)\I_{B(0,R)}^*(x) $, where the cut-off $\I_{B(0,R)}^* $ stands for a mollification of the indicator function of the ball $ B(0,R)$ with support in $B(0,R+1) $ for a given radius $R>0$. 

As $b\in L^{2-\epsilon} $ for any $\varepsilon>0$,  from the embedding \eqref{EMBEDDING} $L^{2-\epsilon}\hookrightarrow B_{2-\epsilon,\infty}^0$, taking $p=2-\epsilon,q=\infty,\beta=0 $, we can enter  the setting of assumption \eqref{cond_coeff_SPKR},
provided that:
\begin{equation}\label{InitialCond-Vortex1}
1-\alpha+\left[\frac{2}{2-\epsilon}-\zeta_0\right]_+<0,
\end{equation}
{\color{black}for some $\varepsilon>0$. Note that for fixed $p,d$ and $\beta_0\ge 0$, $\zeta_0$ is maximal at $p_0'=p$ and is our case then equals to $\beta_0+2/(2-\varepsilon)$, still with an arbitrary choice of $\varepsilon$. As such, we derive that weak well posedness holds for any $\beta_0\ge 0$. Without optimizing in $p_0'$, we have no constraint for $\zeta_0\ge 1$ and the constraint $\zeta_0>2-\alpha$ when $\zeta_0 < 1$. The latter gives in turn $\beta_0 + 2/p_0' >1-\alpha \Leftrightarrow \beta_0> 2/p_0-\alpha$. }

On the other hand, from \eqref{COND_SPKR_STRONG_INTRO}, it holds that the non-linear SDE is strongly well posed if additionally:\\
\begin{equation}\label{InitialCond-Vortex2}
2-\frac 32\alpha+\frac{2}{2-\epsilon}-\zeta_0< 0\Longleftarrow \zeta_0>3-\frac 32\alpha .
\end{equation}
{\color{black} This gives that strong well-posedness hold for $\beta_0>1+2/p_0 - 3\alpha/2$.}
In the case $\alpha=2$,
weak and strong well-posedness hold for e.g. $\beta_0>0,p_0=1 $ or $\beta_0=0,p_0>1 $.\\
\item[-] {\color{black}\textbf{Full kernel and $\beta=-1 $}.}
				Alternatively, consider the interpretation (see e.g. \cite{jabi:wang:18}) $b(x)=K(x) $ with the vector field $K:\mathbb R^d\rightarrow \mathbb R^d$  written as
				$$
				K_i(x)=\frac 1{2\pi}\sum_{j=1}^2\partial_{x_j}V_{i,j}(x),\quad i=1,2,
				$$
				where $V$ is the anti-symmetric $2\times 2$-matrix field:
				\begin{equation*}
					V(x)=\begin{pmatrix}
						 0& -\arctan(x_1/x_2)\\
						\arctan(x_2/x_1) &0 
					\end{pmatrix}.
				\end{equation*}
			Immediately $\div(K)=0$
			. Indeed
			, for the inhomogeneous part of $|K|_{B^{-1}_{\infty,\infty}}$ we have directly from \eqref{HEAT_CAR},
			\begin{align*}
			|\mathcal F^{-1}(\phi\mathcal F(K))|_{L^\infty}&
			=\sum_{i,j}|\mathcal F^{-1}(\phi\mathcal F(\partial_{x_j}V_{i,j}))|_{L^\infty}=\sum_{i,j}|\mathcal F^{-1}(\xi_j\phi\mathcal F(V_{i,j}))|_{L^\infty}\\
			&=\sum_{i,j}|\mathcal F^{-1}(\xi_j\phi)\star  V_{i,j}|_{L^\infty}\le c(d)\big|\mathcal F^{-1}(\xi_j\phi)\big|_{L^1}|V|_{L^\infty}\le C|V|_{L^\infty}<\infty.
			\end{align*}
			For the thermic part, choosing $n=0$ in \eqref{HEAT_CAR}, we have
			\begin{align*}
				\mathcal T_{\infty,\infty}^{-1}(K):&=\sup_{v\in(0,1]}\left\{v^{(1/\textcolor{black}{\alpha})}|\tilde p_\alpha(v,\cdot)*K|_{L^\infty}\right\}\le \sum_{i,j}\sup_{v\in(0,1]}\left\{v^{(1/\textcolor{black}{\alpha})}|\partial_{x_j}\tilde p_\alpha(v,\cdot)*V_{i,j}|_{L^\infty}\right\}\\
				&\le  c(d)|V|_{L^\infty}\sup_{v\in(0,1]}\left\{v^{(1/\textcolor{black}{\alpha})}|\nabla\tilde p_\alpha(v,\cdot)|_{L^1}\right\}\le c(d)\ch|V|_{L^\infty}<\infty,
			\end{align*}
			using for the last step follows from \eqref{SING_STABLE_HK} (which naturally holds for the isotropic density $\tilde p$ and gives $|\nabla\tilde p_\alpha(v,\cdot)|_{L^1}\le \frac{\ch}{(v)^{{\frac {1}\alpha}}}$ for $v\in(0,1]$ along with the embedding $B^{0}_{1,q}\hookrightarrow L^1$).
			In this view,
			the weak wellposedness condition \eqref{THE_COND_CI}, for $p=\infty=r$ reduces to
				$$
				1-\alpha+\big[1-\zeta_0\big]_+<0\Leftrightarrow \zeta_0\ge 1\quad \text{or}\quad 2-\alpha<\zeta_0<1.
				$$
				Indeed, as in the former interpretation, we have no constraint for $\zeta_0\ge 1$ and the constraint $\zeta_0>2-\alpha$ when $\zeta_0 < 1$ and the range for the weak well-posedness is here similar to the one exhibited in \eqref{InitialCond-Vortex1}.
				
				Regarding strong wellposedness, \eqref{COND_CRITIQUE_INTRO} becomes: 
\begin{equation*}
	\Bigg(2-\frac 32\alpha+\Big[-\zeta_0 \Big]\Bigg)\vee \Bigg(- \alpha +[1   -\zeta_0]_+\Bigg)  <-1,
	\end{equation*}
	that is 
	$$
	\zeta_0>3-3\alpha/2\quad \text{and}\quad\Big\{\zeta_0\ge 1\quad\text{or} \quad 2-\alpha<\zeta_0<1\Big\}\Rightarrow \zeta_0>(2-\alpha)\vee (3(1-\alpha/2))=3(1-\alpha/2),
	$$
	which brings us back to \eqref{InitialCond-Vortex2}.

\end{trivlist}

\subsection{The \textcolor{black}{(truncated)} Keller-Segel model}\label{KS}

The Keller-Segel equations are  a system of second order PDEs describing the joint evolution of the distribution $\nu_{{\color{black}s}}(dx)$ of cells (e.g. bacteria) and the concentration of chemo-attractant $c=c({\color{black}s},x)$, which induces a \textcolor{black}{significant} force field in the cell evolution. In its \textcolor{black}{parabolic-elliptic} form, and assuming that the cell distribution has a density, i.e. $\nu_{{\color{black}s}}(dx) =u({\color{black}s},x)dx$, the equations write as
\begin{align}
&\partial_{{\color{black}s}}u({\color{black}s},x)+\chi\nabla\cdot (u({\color{black}s},x)\nabla c({\color{black}s},x))- \frac 1{2}\triangle u({\color{black}s},x)=0,\notag\\
&-\triangle c({\color{black}s},x)= u({\color{black}s},x),\notag\\
&u(0,x)=u^0(x),\,\,c(0,x)=c^0(x)\,\,\,\text{given}.\label{KS_EQ}
\end{align}
The coefficient $\chi$ modulates the intensity of the action of the concentration. Formally, writing again $c({\color{black}s},x)=(-\Delta)^{-1}u ({\color{black}s},x)$ we have $ c({\color{black}s},x)=\int_{\R^d}\mathfrak P(x-y)u({\color{black}s},y)dy$ so that
$\nabla c({\color{black}s},x)=\int_{\R^d}\nabla \mathfrak P(x-y)u({\color{black}s},y)dy$
$:=(K* u({\color{black}s},\cdot))(x) $ for $K(z)=-z/|z|^d c_d$ (where the constant $c_d$ depends on the considered dimension. This leads to consider a kernel of the form $b(x)= \chi  K (x)$ to derive the corresponding McKean-Vlasov interpretation. The kernel is strongly attractive, and compared to the vortex equations \textcolor{black}{is \textit{not}} divergence-free. In particular this can \textcolor{black}{lead} to blow-up phenomena, i.e. the \textcolor{black}{cells aggregate at 0}, leading to a degeneracy of $\nu_{{\color{black}s}}$ to the Dirac measure $\delta_{\{0\}}$. The way \textcolor{black}{these} blow-up phenomena \textcolor{black}{emerge} is inherently related to the smallness of $u^0$ and $\chi$\textcolor{black}{.} In dimension $2$, it is known, see e.g. the monograph of Biler \cite{bile:19}, that global well-posedness will hold provided $\chi< 8\pi$. The motivation to model diffusion through a (non-local) fractional diffusion comes from the fact that organisms may adopt L\'evy flight search strategies for their nutriment. \textcolor{black}{In that setting},   dispersal is \textcolor{black}{then}  better \textcolor{black}{modeled} by non-local operators (\cite{Escudero-06}, \cite{BouCal-10}).

To show how \eqref{KS_EQ} enters in the framework of the assumptions of Theorems \ref{main_thm_W} and \ref{main_thm_S}, we may proceed as in the case of the vortex equation,  observe first that the kernel $K$ belongs to $L_{{\rm loc }}^{\mathfrak p}$ for $\mathfrak p<\frac d{d-1} $. Hence, setting $b=\chi K\I_{B(0,R)}^*, R>0 $ where as above $\I_{B(0,R)}^* $ stands for a mollification of the indicator function of the ball $B(0,R)$, we get $ b\in L^p$ for $p<d/d-1$.  

We can again, from \eqref{EMBEDDING}, enter the setting of assumption \eqref{cond_coeff_SPKR} for $p\in [1,\frac d{d-1}),q=\infty,\beta=0 $, as soon as
\begin{equation}\label{C_WEAK_KS}
1-\alpha+[d-1-\zeta_0]_+ <0,
\end{equation}
which gives weak well-posedness for $\zeta_0\ge d-1 $ for any $\alpha\in (1,2] $ or as soon as $-\alpha+d<\zeta_0$ for smaller $\zeta_0 $. Thus $\zeta_0>d-\alpha $ gives weak well-posedness. To obtain strong well-posedness, from \eqref{COND_SPKR_STRONG_INTRO}, it is also required to fulfill
$$2-\frac 3{2}\alpha+d<\zeta_0.$$
In particular, in the Brownian regime, the condition for the (local) \textit{strong} uniqueness  reads as
\[
d-1<(\beta_0+\frac{d}{p_0'})(1\wedge \frac{p_0'}{p}). 
\]
This will be fulfilled if e.g. $\beta_0=0 $ (no \textit{a priori} smoothness of the initial data) and $p_0>d\iff \frac{1}{p_0}<\frac 1d\iff \frac{1}{p_0'}>\frac{d-1}{d}\iff p_0'<\frac{d}{d-1} $ as long as $p\le p_0'$ (which can be assumed w.l.o.g. since if $p_0'<\frac{d}{d-1} $, one can always view $b$ as an element of $B_{p_0',\infty}^0 $) or at the other extreme, (imposing no specific integrability properties) provided $\beta_0>d-1$ for $p_0=1$. Observe from \eqref{C_WEAK_KS} that the condition for weak uniqueness is actually weaker and reads $\zeta_0>d-2 $. In particular, for $d=2$, this will be fulfilled as soon as e.g. $\beta_0>0,p_0=1 $ or $\beta_0=0, p_0>1 $.


As for the vortex case, despite the localization of the kernel, the above conditions should be the ones under which \textit{weak and} \textit{strong} well-posedness for the McKean-Vlasov SDE associated with the Keller-Segel system \eqref{KS_EQ} \textcolor{black}{hold}.

\appendix

\appendix
{\color{black}

{\color{black}
\section{A continuity result with mollified coefficients}}

{\color{black}
We recall that we defined from the previous analysis, see \eqref{def_gamma} and \eqref{def_thetas},
\begin{equation*}
\Gamma {\color{black}:= 
\eta\left\{ \alpha-1+\beta-\frac \alpha r + \beta \ind_{\beta >-1}-\frac dp + \zeta_0 \right\},\quad \eta \in (0,1)},
\end{equation*}
and
\begin{equation*}
\theta
\textcolor{black}{=\frac 1\alpha\left\{-\beta + \frac dp - \zeta_0 + \left(\frac{1+\eta}{2\eta}\right) \Gamma\right\}>0}.
\end{equation*}}

\begin{lemme}[Continuity for the supremum of the time normalized Besov norm] 
\label{lem_cont0}
Set, for $\varepsilon >0$,
$$f_t^{{\varepsilon}}:s\in (t,T)\mapsto f_t^{{\varepsilon}}(s):=\sup_{v\in (t,s]}({v}-t)^{\theta} |\brho_{t,\mu}^\varepsilon(v,\cdot)|_{B^{ -\beta+\Gamma}_{p',{\color{black}1}}},$$
with $\theta $ as in Lemma \ref{lem_unifesti_gencase2_RELOADED} (see equation \eqref{def_thetas}).

For every $\varepsilon >0$, 
the map  $s\in (t,T]\mapsto f_t^\varepsilon(s) $ can be extended in $t$ by continuity in $s$ setting $f_t^\varepsilon(t)=0$.
\end{lemme}
\begin{proof}  
Recall first that for every $\varepsilon >0$, $b^\varepsilon$ is a smooth function from $(t,T) \times \R^d \to \R^d$. In particular, the non-linear drift is now a smooth function \textcolor{black}{and} for all $r,\, \gamma,\, \ell,\, m$ in $ [1,\infty] \times \R^+ \times [1,\infty]^2$, by \eqref{YOUNG} and \eqref{EMBEDDING}:
\begin{equation}\label{EXPLO_DRIFT}
\left|\mathcal B_{\brho_{t,\mu}^\varepsilon}^\varepsilon\right|_{L^{r}(B^{\gamma}_{\ell, m})} \le  \left| b^\varepsilon\right|_{L^{r}(B^{\gamma}_{\ell, m})} \left|\brho_{t,\mu}^\varepsilon\right|_{L^\infty(L^1)} \le \textcolor{black}{K_{\varepsilon,r,\gamma,\ell,m}}.
\end{equation}


We will use some heat kernel estimates on the mollified law.
To this end, the point is to use  the representation of the density. The idea underneath is that for all $y\in \R^d $:
\begin{equation}\label{DESINT}
\brho_{t,\mu}^\varepsilon(s,y)=\int_{\R^d}\mu(dx)\tilde \rho_{t,x,\mu}^\varepsilon(s,y),
\end{equation}
where $\tilde \rho_{t,x,\mu}^\varepsilon(s,y) $ stands for the density of the SDE:
\begin{equation}\label{EST_BESOV_NORM}
\tilde X_s^{\varepsilon,t,x,\mu}=x+\int_t^s \mathcal B_{\rho_{t,\mu}^\varepsilon}^{\varepsilon}(v,\tilde X_v^{\varepsilon,t,x,\mu})dv +\mW_s-\mW_t.
\end{equation}
{\color{black}We will actually derive a heat kernel type estimate on $\tilde \rho_{t,x,\mu}^\varepsilon(s,y)$ which somehow guarantees that the same computations performed in Lemma \ref{controle_CI} for the initial condition actually apply for $\brho_{t,\mu}^\varepsilon(s,\cdot)=\mu\star \tilde \rho_{t,x,\mu}^\varepsilon(s,\cdot) $. Namely, we have similarly to  \eqref{HK_PARTIAL_1} in Lemma \ref{controle_CI} that
\begin{align*}
 \Big|\mu \star \tilde \rho_{t,x,\mu}^\varepsilon(s,\cdot)\Big|_{B^{ -\beta+\Gamma}_{p',1}} \le 
 C|\mu|_{B_{p_1,q_1}^{\beta_1}} | \tilde \rho_{t,x,\mu}^\varepsilon(s,\cdot)|_{B_{\mathfrak p(p_1,p'),1}^{-\beta+{\color{black}\gamma}-\beta_1}}.
\end{align*} 
We will hereafter derive that 
\begin{align*}
 \Big|\mu \star \tilde \rho_{t,x,\mu}^\varepsilon(s,\cdot)\Big|_{B^{ -\beta+\Gamma}_{p',1}} \le 
 C_\varepsilon|\mu|_{B_{p_1,q_1}^{\beta_1}} (s-t)^{- \textcolor{black}{\frac{1}{\alpha}\left[\Gamma- \beta + \frac dp - \zeta_0\right]_+}},
\end{align*} 
since $ \tilde \rho_{t,x,\mu}^\varepsilon(s,\cdot)$ satisfies usual heat kernel estimates (with constants depending \textit{a priori} on the mollification parameter).} 

Let us now turn to the derivation  of the heat kernel type bound on $|\tilde \rho_{t,x,\mu}^\varepsilon(s,\cdot)|_{B_{\mathfrak p(p_1,p'),1}^{-\beta+\Gamma-\beta_1}} $ where we recall that
\begin{align*}
[\mathfrak p(p_1,p')]^{-1} &= 1 +  (p')^{-1} - (p_1)^{-1}=\begin{cases}
1+\frac{1}{p'}-\frac{1}{p_0},\ p_1=p_0, \\
1,\ p_1=\min(p_0,p')=p'.
\end{cases}\\
\beta_1&=\beta_0\left( 1\wedge \frac{p_0'}{p}\right). 
 \end{align*}
 
  The point again consists in starting from the Duhamel representation of the density. Similarly to \eqref{main_MOLL} one gets
\begin{eqnarray}\label{main_MOLL_DESINT}
\tilde \rho^\varepsilon_{t,x,\mu}(s,y) &=& p^\alpha_{s-t}(y-x) +\int_t^s dv \Big[\{\mathfrak b_{t,\mu}^\varepsilon(v,\cdot) \tilde \rho^\varepsilon_{t,x,\mu}(v,\cdot)\} \star \nabla p_{s-v}^{\alpha}\Big] (y),\\
\mathfrak b_{t,\mu}^\varepsilon(v,\cdot)&:=&\mathcal B_{\rho_{t,\mu}^\varepsilon}^{\varepsilon}(v,\cdot) \notag,
\end{eqnarray}
where we here exploited the backward Kolmogorov equation associated with the density.

The proof will proceed in two steps. If $ \mathfrak p(p_1,p')>1$ we will first perform a bootstrap argument in the integrability parameter and then another bootstrap type argument concerning the regularity parameter, whereas if ${\color{black}\mathfrak p(p_1,p')=1 }$ only the second step is needed. Hence, w.l.o.g. we handle the case $ \mathfrak p(p_1,p')>1 $.

Take now $\ell_0$ s.t. $1/\alpha+d/( \alpha \ell_0)<1 \iff \ell_0>d/(\alpha-1)  $. Then for any $\gamma>0 $ s.t.  $1/\alpha+d/ (\alpha \ell_0)+\gamma/\alpha<1$, from \eqref{SING_STABLE_HK}, \eqref{YOUNG} and \eqref{EMBEDDING}, we derive:
\begin{align*}
\|\tilde \rho^\varepsilon_{t,x,\mu}(s,\cdot)\|_{B_{\ell_0',1}^\gamma} &\le \|p^\alpha_{s-t}(\cdot-x)\|_{_{B_{\ell_0',1}^\gamma}} +\int_t^s dv \Big|\Big[\{\mathfrak b_{t,\mu}^\varepsilon(v,\cdot) \tilde \rho^\varepsilon_{t,x,\mu}(v,\cdot)\} \star \nabla p_{s-v}^{\alpha}\Big] \Big|_{_{B_{\ell_0',1}^\gamma}}\\
&\le C(s-t)^{-\frac 1\alpha(\gamma+\frac{d}{\ell_0})}+\int_t^s dv \Big|\{\mathfrak b_{t,\mu}^\varepsilon(v,\cdot) \tilde \rho^\varepsilon_{t,x,\mu}(v,\cdot)\}\Big|_{B_{1,\infty}^0} |\nabla p_{s-v}^{\alpha} \Big|_{_{B_{\ell_0',1}^\gamma}}\\
 &\le C(s-t)^{-\frac 1\alpha(\gamma+\frac{d}{\ell_0})}+\int_t^s dv \Big|\mathfrak b_{t,\mu}^\varepsilon(v,\cdot)\Big|_{L^\infty} |\tilde \rho^\varepsilon_{t,x,\mu}(v,\cdot)|_{L^1} (v-t)^{-\frac 1\alpha(1+\gamma+\frac{d}{\ell_0})}\\
 &\le C(s-t)^{-\frac 1\alpha(\gamma+\frac{d}{\ell_0})}+C_\varepsilon(s-t)^{\frac{\alpha-(1+\gamma+\frac{d}{\ell_0})}{\alpha}}\le C_\varepsilon(s-t)^{-\frac 1\alpha(\gamma+\frac{d}{\ell_0})},
\end{align*}
where we also used \eqref{EXPLO_DRIFT} for the last but one inequality. 
If $\mathfrak p(p_1,p')'>d/(\alpha-1) $ we can therefore get the required integrability parameter as well as some residual regularity $ \gamma>0$ smaller than the remaining margin. 
If now $\mathfrak p(p_1,p')'<\ell_0 $, some integrability gain is still to be obtained. For some $\delta>0$ to specify we now want to exploit the former control to obtain a bound on the $B_{\ell_0'+\delta,1}^\gamma $ norm of the density. Precisely, 
\begin{align*}
\|\tilde \rho^\varepsilon_{t,x,\mu}(s,\cdot)\|_{B_{\ell_0'+\delta,1}^\gamma} &\le \|p^\alpha_{s-t}(\cdot-x)\|_{_{B_{\ell_0'+\delta,1}^\gamma}} +\int_t^s dv \Big|\Big[\{\mathfrak b_{t,\mu}^\varepsilon(v,\cdot) \tilde \rho^\varepsilon_{t,x,\mu}(v,\cdot)\} \star \nabla p_{s-v}^{\alpha}\Big] \Big|_{_{B_{\ell_0'+\delta,1}^\gamma}}\\
&\le C(s-t)^{-\frac 1\alpha(\gamma+\frac{d}{(\ell_0'+\delta)'})}+\int_t^s dv \Big|\{\mathfrak b_{t,\mu}^\varepsilon(v,\cdot) \tilde \rho^\varepsilon_{t,x,\mu}(v,\cdot)\}\Big|_{B_{\ell_0',1}^\gamma} |\nabla p_{s-v}^{\alpha} \Big|_{_{B_{r(\delta),1}^0}},
\end{align*}
where $(\ell_0'+\delta)' $ denotes the conjugate exponent of $\ell_0'+\delta $ and from \eqref{YOUNG}, $1+\frac{1}{\ell_0'+\delta}=\frac{1}{\ell_0'}+\frac {1}{r(\delta)} \iff \frac{1}{r(\delta)}=1+\frac{1}{\ell_0'+\delta}-\frac{1}{\ell_0'}=1-\frac{\delta}{\ell_0'(\ell_0'+\delta)}\iff \frac{1}{(r(\delta))'}=\frac{\delta}{\ell_0'(\ell_0'+\delta)}$. Hence to keep the above integral finite one needs to take $\delta >0$ s.t. $\frac{1}{\alpha}+\frac{d}{\alpha (r(\delta))'}<1\iff (r(\delta))'>\frac{d}{\alpha-1}\iff \frac{\ell_0'(\ell_0'+\delta)}{\delta}>\frac{d}{\alpha-1}\iff \delta \Big(\frac d{\alpha-1} -\ell_0'\Big) <(\ell_0')^2$. For such a $\delta>0 $ it then holds from \eqref{EXPLO_DRIFT} that
\begin{align*}
\|\tilde \rho^\varepsilon_{t,x,\mu}(s,\cdot)\|_{B_{p_0'+\delta,1}^\gamma} 
&\le C(s-t)^{-\frac 1\alpha(\gamma+\frac{d}{(\ell_0'+\delta)'})}+C_{\varepsilon}(s-t)^{1-(\frac{\gamma}{\alpha}+\frac{d}{\alpha \ell_0})-(\frac{1}{\alpha}+\frac{d}{\alpha (r(\delta))'})}\\
&\le C_{\varepsilon,1}(s-t)^{-\frac 1\alpha(\gamma+\frac{d}{(\ell_0'+\delta)'})}.
\end{align*}
Indeed $ (\alpha-1-\frac{d}{ (r(\delta))'})-\frac{d}{ \ell_0}+\frac{d}{(\ell_0'+\delta)'}\ge 0\iff \alpha\ge 1$. It can then be seen by induction that one obtains for $\delta$ small enough and s.t. $\ell_0'+n\delta=\mathfrak p(p_1,p') $, for all $j\in \{ 1,\cdots,n\} $:
\begin{align}
\|\tilde \rho^\varepsilon_{t,x,\mu}(s,\cdot)\|_{B_{\ell_0'+j\delta,1}^\gamma} &\le \|p^\alpha_{s-t}(\cdot-x)\|_{_{B_{\ell_0'+j\delta,1}^\gamma}} +\int_t^s dv \Big|\Big[\{\mathfrak b_{t,\mu}^\varepsilon(v,\cdot) \tilde \rho^\varepsilon_{t,x,\mu}(v,\cdot)\} \star \nabla p_{s-v}^{\alpha}\Big] \Big|_{_{B_{\ell_0'+j\delta,1}^\gamma}}\notag\\
&\le C(s-t)^{-\frac 1\alpha(\gamma+\frac{d}{(\ell_0'+j\delta)'})}+\int_t^s dv \Big|\{\mathfrak b_{t,\mu}^\varepsilon(v,\cdot) \tilde \rho^\varepsilon_{t,x,\mu}(v,\cdot)\}\Big|_{B_{\ell_0'+(j-1)\delta,1}^\gamma} |\nabla p_{s-v}^{\alpha} \Big|_{_{B_{r_j(\delta),1}^0}},
\label{THE_PREAL_APRIORI_HK}
\end{align}
where $\frac{1}{r_j(\delta)}=1+\frac{1}{\ell_0'+j\delta}-\frac{1}{p_0'+(j-1)\delta}=1-\frac{\delta}{(\ell_0'+(j-1)\delta)(\ell_0'+j\delta)}\iff  \frac{1}{r_j(\delta)'}=\frac{\delta}{(\ell_0'+(j-1)\delta)(\ell_0'+j\delta)}$. Hence, there are no integrability issues, since $\frac{1}{r_j(\delta)'} $ decreases with $j$. This gives:
\begin{align*}
\|\tilde \rho^\varepsilon_{t,x,\mu}(s,\cdot)\|_{B_{\ell_0'+j\delta,1}^\gamma} 
&\le C(s-t)^{-\frac 1\alpha(\gamma+\frac{d}{(\ell_0'+j\delta)'})}+C_{\varepsilon,j-1}(s-t)^{1-(\frac{\gamma}{\alpha}+\frac{d}{\alpha (\ell_0'+(j-1)\delta)'})-(\frac{1}{\alpha}+\frac{d}{\alpha (r_j(\delta))'})}\\
&\le C_{\varepsilon,j}(s-t)^{-\frac 1\alpha(\gamma+\frac{d}{(\ell_0'+j\delta)'})},
\end{align*}
since again 
$(\alpha-1-\frac{d}{ (r_j(\delta))'})-\frac{d}{ (\ell_0'+(j-1)\delta )'}+\frac{d}{(\ell_0'+j\delta)'}\ge 0$.
Thus, the estimate with the right integrability index follows from \eqref{THE_PREAL_APRIORI_HK} taking $j=n$.

To obtain the required estimate, namely a control on $|\brho_{t,\mu}^\varepsilon(s,\cdot)|_{B_{\mathfrak p(p_1,p'),1}^{-\beta+\Gamma-\beta_1}}$, it suffices to iterate the previous bootstrap type approach on the regularity parameter. Namely, for a small parameter $\gamma_0>0$, and $j\in \mathbb N^* $, write:
\begin{align}
\|\tilde \rho^\varepsilon_{t,x,\mu}(s,\cdot)\|_{B_{\mathfrak p(p_1,p'),1}^{\gamma+j \gamma_0}} &\le \|p^\alpha_{s-t}(\cdot-x)\|_{_{B_{\mathfrak p(p_1,p'),1}^{\gamma+j\gamma_0}}} +\int_t^s dv \Big|\Big[\{\mathfrak b_{t,\mu}^\varepsilon(v,\cdot) \tilde \rho^\varepsilon_{t,x,\mu}(v,\cdot)\} \star \nabla p_{s-v}^{\alpha}\Big] \Big|_{_{B_{\mathfrak p(p_1,p'),1}^{\gamma+j\gamma_0}}}\notag\\
&\le C(s-t)^{-\frac 1\alpha(\gamma+ j\gamma_0+\frac{d}{\mathfrak p(p_1,p')'})}+\int_t^s dv \Big|\{\mathfrak b_{t,\mu}^\varepsilon(v,\cdot) \tilde \rho^\varepsilon_{t,x,\mu}(v,\cdot)\}\Big|_{B_{\mathfrak p(p_1,p'),1}^{\gamma+(j-1)\gamma_0}} |\nabla p_{s-v}^{\alpha}|_{B_{1,1}^{\gamma_0}},
\label{THE_PREAL_APRIORI_HK}
\end{align}
provided $(\gamma+(j-1)\gamma_0+\frac d{\mathfrak p(p_1,p')'})<\alpha $ and $\gamma_0<\alpha-1 $. The point is then to choose $\gamma_0  $ small enough and $m\in \mathbb N $ s.t. $\gamma+\gamma_0 m
=-\beta+\Gamma-\beta_1$, observing that $ -\beta+\Gamma-\beta_1+\frac{d}{\mathfrak p(p_1,p')'}<\alpha$ under the considered assumptions. 
Hence, there are feasible parameters for which the procedure yields 
 \begin{equation}\label{HK_PARTIAL_DENS}
 |\tilde \rho^\varepsilon_{t,x,\mu}(s,\cdot)|_{B_{\mathfrak p(p_1,p'),1}^{-\beta+\Gamma-\beta_1}}\le C_\varepsilon(s-t)^{-\frac 1\alpha(-\beta+\Gamma+\frac dp-\zeta_0)},
  \end{equation}
  which in turn gives, recalling \eqref{DESINT} and similarly to Lemma \ref{controle_CI}
  $$|\brho_{t,\mu}^\varepsilon(s,\cdot)|_{B^{ -\beta+\Gamma}_{p',{\color{black}1}}}\le  C_\varepsilon(s-t)^{-\frac 1\alpha(-\beta+\Gamma+\frac dp-\zeta_0)}.$$
The statement then follows multiplying by $(s-t)^{\theta} $.

\end{proof}
}

{\color{black}\section{About the extension to global well-posedness}}
\label{EXT_TO_LONG_TIME}

We here aim at giving some elements that would allow to extend the previous results {\color{black}stated in Section \ref{ESTI_FOR_FK}} to an arbitrary time horizon, provided the initial norm is sufficiently small in an appropriate norm. To this end we need to introduce a slightly different class of weighted spaces. Namely, two normalizing scales appear. One for the \textit{short time} (as above) and another one for the \textit{long time}.

\paragraph{Extension of Weighted Lebesgue-Besov spaces} 
Introduce for 
$\sc, \ell,m\in [1,+\infty], \gamma, \lambda_1,\lambda_2\in \R $, {\color{black}$t\le S\le T$,}
\begin{align*}
&L_{\w_{\lambda_1}^{\lambda_2}}^{\sc}((t,S],B_{\ell,m}^\gamma):=\Bigg\{f:s\in[t,S]\mapsto f(s,\cdot)\in B^{ \gamma}_{ \ell, m}\,\text{measurable, s.t.}\, \int_t^S  | f(s,\cdot)|_{B^{ \gamma}_{ \ell, m}}^{\textcolor{black}{\sc}} {\w_{\lambda_1}^{\lambda_2}}_{}(s-t)ds<+\infty \Bigg\},
\end{align*}
if $\sc<+\infty $ and
\begin{align}
&L_{{\text w}_{\lambda_1}^{\lambda_2}}^\infty((t,S],B_{\ell,m}^\gamma)\notag\\
:=&\Bigg\{f:s\in[t,S]\mapsto f(s,\cdot)\in B^{ \gamma}_{ \ell, m}\,\text{measurable, s.t.}\, \text{ess sup}_{s\in(t,S]}  \Big({\w}_{\lambda_1}^{\lambda_2}(s-t)| f(s,\cdot)|_{B^{ \gamma}_{ \ell, m}}\Big)<+\infty \Bigg\},
\end{align}
where the weight function ${\w}_{\lambda_1}^{\lambda_2}$ is given by 
\begin{equation}
\label{DEF_WEIGHT}\tag{\textbf{W}}
{\w}_{\lambda_1}^{\lambda_2}(s-t)=[(s-t)^{\lambda_1} \wedge 1] \times [(s-t)^{\lambda_2} \vee 1].
\end{equation}
In other words, the subscript of the weight corresponds to exponent of the \textit{short time renormalization} while the superscript stands for the exponent related to the \textit{long time renormalization}.

Endowed with the metric
\begin{equation*}
|f|_{L_{{\w}_{\lambda_1}^{\lambda_2}}^{\textcolor{black}{\sc}}((t,S],B^\gamma_{\ell,m})}=\Bigg( \int_t^S ds  | f(s,\cdot)|_{B^{\gamma}_{\ell,m}}^{\textcolor{black}{\sc}}{\w}_{\lambda_1}^{\lambda_2}(s-t)\Bigg)^{\frac 1{\textcolor{black}{\sc}}} ,
\end{equation*}
 with the usual modification if $ \textcolor{black}{\sc}=+\infty $, the normed space $(L_{{\w}_{\lambda_1}^{\lambda_2}}^{\textcolor{black}{\sc}}((t,S],B_{\ell,m}^\gamma, \w_{\lambda_1}^{\lambda_2}),|\cdot|_{L_{{\w}_{\lambda_1}^{\lambda_2}}^{\textcolor{black}{\sc}}((t,S],B^\gamma_{\ell,m})})$ is also a Banach space (see e.g. again \cite[Chapter 1]{HyNeVeWe-16}).
 
 We are going to give some key controls and conditions to derive the following extension of Lemmas \ref{lem_unifesti_gencase2_RELOADED}
and \eqref{lem_quad_gronv_C1} which are really the key to derive the previous results established in small time. Namely, introducing the condition:
 \begin{align}\label{LT}\tag{\textbf{LT}}
\textcolor{black}{1-\alpha + \frac \alpha r + \frac dp \ge 0,\ \textcolor{black}{1<}\alpha(1-\frac 1r)<2}{\color{black},}
\end{align}

 it holds that
 \begin{lemme}[A priori estimates on the mollified density, long time regime]\label{lem_unifesti_gencase2_RELOADED_LT} 
Assume $T-t\ge 1$ and that 
\textbf{(C1)} or \textbf{(C2)},  \eqref{LT} hold. Then, defining 
\begin{equation}\label{def_thetas_LT}
\theta_1 := \textcolor{black}{\frac 1\alpha}\left\{-\beta + \frac dp - \zeta_0 + \left(\frac{1+\eta}{2\eta}\right) \Gamma\right\}, \quad \theta_2 := \frac 1{r'} - \frac 1\alpha,
\end{equation}
there exists $C := C(\Theta)>0$ such that for all ${S \le T}$,
\begin{eqnarray}
&&\sup_{s\in (t,S]}\Big\{ {\rm w}_{\theta_1}^{\theta_2}(s-t) |\brho_{t,\mu}^\varepsilon(s,\cdot)|_{B_{p',1}^{-\beta+ \Gamma}}\Big\} \le  \textcolor{black}{C\Bigg\{} \textcolor{black}{|\mu|_{B^{\beta_1}_{p_1,q_1}}}  {\rm w}_{\frac{1-\eta}{2\eta} \frac \Gamma{\alpha}}^{0}(S-t)\notag
\\
&&\qquad  +\left(|b|_{L^{r}(B^{\beta}_{p,q})} + |\div(b)|_{L^{r}(B^{\beta}_{p,q})}\ind_{\beta =-1} \right)  {\rm w}_{\frac{1-\eta}{2\eta}\frac \Gamma{\alpha}}^{0}(S-t)\bigg(\sup_{s\in (t,S]}\Big\{ \w_{\theta_1}^{\theta_2}(s-t) |\brho_{t,\mu}^\varepsilon(s,\cdot)|_{B_{p',1}^{-\beta+ \Gamma}}\Big\} \bigg)^2\textcolor{black}{\Bigg\}},\notag
\end{eqnarray}
\textcolor{black}  {for $\Gamma $ as in \eqref{def_gamma} and
${\color{black}{\rm w}^{\lambda_2}_{\lambda_1}}$ defined in \eqref{DEF_WEIGHT}}.
\end{lemme}

\begin{lemme}[A priori control through a Gronwall type inequality with quadratic growth]\label{lem_quad_gronv_C1_LT} Under \textbf{(C1)} or \textbf{(C2)} and \eqref{LT}, 
and for $\theta_1,\theta_2, \Gamma$ as in Lemma \ref{lem_unifesti_gencase2_RELOADED_LT} and Theorem \ref{main_thm_W}, there exist
 $\mathcal C_0:=\textcolor{black}{\mathcal C_0(\Theta)}>0$ and $C_{\ref{lem_quad_gronv_C1_LT}} $ such that   for any $T>0$,  $S  \le T$, uniformly in $\varepsilon>0 $, 
 \begin{eqnarray}\label{Unifestim_quadrgronwall_C1_LT}
 \sup_{s\in (t,S]}[ \w_{\theta_1}^{\theta_2}(s-t) |\brho_{t,\mu}^\varepsilon(s,\cdot)|_{B^{ \textcolor{black}{-\beta}+\Gamma}_{p',{\color{black}1}}}]
\le 
C_{\ref{lem_quad_gronv_C1_LT}},
\end{eqnarray} 
whenever $|\mu|_{\textcolor{black}{B_{p_1,q_1}^{\beta_1}}} \le  \mathcal C_0$. If now  $|\mu|_{\textcolor{black}{B_{p_1,q_1}^{\beta_1}}} >  \mathcal C_0$, there exists $\mathcal T_2 := \mathcal T_2\big(\Theta\big)>t$ such that, for all $S \le \mathcal T_2$, estimate \eqref{Unifestim_quadrgronwall_C1_LT} holds.

\end{lemme}

 To prove the above results, the starting point is again the Duhamel representation \eqref{TRIANG_KR}. To handle the convolution of the initial condition with the heat semi-group we first need t{\color{black}o} extend Lemma \ref{controle_CI}. 
\begin{lemme}[Besov controls for the convolution of the initial condition and the stable heat kernel in possibly long time]\label{controle_CI_LT}
Recall that 
\begin{equation*}
 p_1=\min(p_0,p'),\ q_1=q_0(1\vee \frac{p}{p_0'}),\ \beta_1=\beta_0(1\wedge \frac{p_0'}p).
\end{equation*}
Then, for any $s \ge t$ it holds that for any $\Gamma>0$:
 \begin{align}
 \Big|\mu \star p^\alpha_{s-t}\Big|_{B^{ -\beta+\Gamma}_{p',1}} \le 
 C|\mu|_{B_{p_1,q_1}^{\beta_1}} [(s-t)\wedge 1]^{- \frac{1}{\alpha}\textcolor{black}{\left[\Gamma - \beta + \frac dp - \zeta_0\right]_+}}
 [(s-t)\vee 1]^{- \frac{d}{\alpha p}}
 ,
 \label{NEW_CTR_CI_LT}
\end{align} 
where $\zeta_0=\left(\beta_0 + \frac{d}{p_0'}\right)\left( 1 \wedge \frac{p_0'}{p} \right)$ is the \textit{regularity gain from the initial condition} defined in \eqref{def_zetra}.
\end{lemme}

 \begin{proof}
To establish \eqref{NEW_CTR_CI_LT}, it only remains to  handle the long time estimate, \textit{i.e.} we assume that $s-t >1$. In such case, using  \eqref{lem_proba_in besov} we write from \eqref{YOUNG} and then \eqref{SING_STABLE_HK}
 \begin{eqnarray}
\Big|\mu \star p^\alpha_{s-t}\Big|_{B^{ -\beta+\Gamma}_{p',1}} &\le& \cv|\mu|_{B^{0}_{1,\infty}} | p^\alpha_{s-t}|_{\textcolor{black}{B^{ -\beta+\Gamma}_{p', 1}}} \le C \cv|\mu|_{B^{0}_{1,\infty}}  [(s-t)\vee 1]^{- \frac{d}{\alpha p}},\quad s >t+1. \label{NEW_CTR_CI_LT_BIS}
\end{eqnarray}

Recalling now from \eqref{def_zetra} that $\zeta_0 = \left(\beta_0 + d/(p_0')\right)\left( 1 \wedge [p_0'/p] \right),$ we conclude from \eqref{NEW_CTR_CI_ST} and \eqref{NEW_CTR_CI_LT_BIS} that for any $s>t$,
 \begin{align*}
 \Big|\mu \star p^\alpha_{s-t}\Big|_{B^{ -\beta+\Gamma}_{p',1}} \le 
 C|\mu|_{B_{p_1,q_1}^{\beta_1}} [(s-t)\wedge 1]^{- \frac{1}{\alpha}\left[\Gamma - \beta + \frac dp - \zeta_0\right]}[(s-t)\vee 1]^{- \frac{d}{\alpha p}}.
\end{align*} 
\textit{i.e.} \eqref{NEW_CTR_CI} holds. This concludes the proof.
\end{proof}

\begin{proof}[Proof of Lemma \ref{lem_unifesti_gencase2_RELOADED_LT}.]
Now, similarly to the proof of Lemma \ref{lem_unifesti_gencase2_RELOADED} we get under \eqref{cond_coeff_SPKR}:
\begin{eqnarray}
|\brho_{t,\mu}^\varepsilon(s,\cdot)|_{B^{-\beta+ \Gamma}_{p',1}} &\le&  \textcolor{black}{C \Bigg\{} \textcolor{black}{|\mu|_{B^{\beta_1}_{p_1,q_1}}}  [(s-t)\wedge 1]^{- \frac{1}{\alpha}\left[\Gamma - \beta + \frac dp - \zeta_0\right]}[(s-t)\vee 1]^{- \frac{d}{p\alpha}}\notag\\
&& +|b|_{L^{r}(B^{\beta}_{p,q})}  \Bigg(\int_t^s \frac{dv}{[(s-v)\wedge 1]^{\frac {-\beta+1}\alpha r'}[(s-v)\vee 1]^{\frac {r'}\alpha}}  | \brho_{t,\mu}^\varepsilon(v,\cdot)|_{B^{-\beta+\Gamma}_{p', 1}}^{2r'} \Bigg)^{\frac 1{r'}}\textcolor{black}{\Bigg\}}\label{CI_EPS_AVANT_GR_QUADRA_PROOF_2_LT},
\end{eqnarray}
where we have used the heat kernel controls \eqref{SING_STABLE_HK} in order to take into account the long time behavior.

In order to have time integrable singularities in the former integral we assume that:
\begin{equation}\label{C1ST_a_LT}\tag{a${}_1$}
\frac 1{r'}-\frac{1-\beta}\alpha >0.
\end{equation}
\textcolor{black}{We now introduce two parameters $\theta_1$ and $\theta_2$, meant to be non negative, which we are going to calibrate.} 
With the notations of  \eqref{DEF_WEIGHT},
we write from \eqref{CI_EPS_AVANT_GR_QUADRA_PROOF_2_LT}:
\begin{eqnarray}
&&\sup_{s\in (t,S]}\Big\{ \w_{\theta_1}^{\theta_2}(s-t) |\brho_{t,\mu}^\varepsilon(s,\cdot)|_{B_{p',1}^{-\beta+ \Gamma}}\Big\} \le \textcolor{black}{C\Bigg\{} \textcolor{black}{|\mu|_{B^{\beta_1}_{p_1,q_1}}} \max_{s\in(t,S]} \Big\{\w_{\theta_1- \frac{1}{\alpha}\left[\Gamma - \beta + \frac dp - \zeta_0\right]}^{\theta_2- \frac{d}{p\alpha}}(s-t)\Big\} \label{CI_EPS_AVANT_GR_QUADRA_PROOF_2_TER}\\
&&\qquad  +|b|_{L^{r}(B^{\beta}_{p,q})}  \max_{s\in(t,S]}\Bigg\{\w_{\theta_1}^{\theta_2}(s-t)\Bigg(\int_t^s \frac{dv}{\w_{\frac{-\beta+1}{\alpha} r'}^{\frac{r'}{\alpha}}(s-v)} | \brho_{t,\mu}^\varepsilon(v,\cdot)|_{B^{ -\beta+\Gamma}_{p', 1}}^{2r'}  \Bigg)^{\frac 1{r'}}\Bigg\}\textcolor{black}{\Bigg\}}.\notag
\end{eqnarray}
To obtain homogeneous quantities, we now have to singularize the integrand in the above r.h.s. \textcolor{black}{To do so, we need the exponents $\theta_1$ to be so that $(v-t)^{-2\theta_1 r'}$ is integrable around $t$. For technical reasons ({\color{black}relative to the }application of Lemma \ref{lem_gestion_sing}) we suppose for $i=1,2$ that}, 
\begin{equation}\label{C1ST_b_LT}\tag{b${}_1 $}
\theta_i < \frac 1{2r'},\quad i=1,2.
\end{equation}
We now get
\begin{eqnarray*}
&&\int_t^s \frac{dv}{\w_{\frac{-\beta+1}{\alpha} r'}^{\frac{r'}{\alpha}}(s-v)}  | \brho_{t,\mu}^\varepsilon(v,\cdot)|_{B^{ -\beta+\Gamma}_{p', 1}}^{2r'}\\
&=&\int_t^s \frac{dv}{\w_{\frac{-\beta+1}{\alpha} r'}^{\frac{r'}{\alpha}}(s-v) \,  \w_{2\theta_1 r'}^{2\theta_2r'}(v-t)} \Big([\w_{\theta_1}^{\theta_2}(v-t)] | \brho_{t,\mu}^\varepsilon(v,\cdot)|_{B^{ -\beta+\Gamma}_{p', 1}}\Big)^{2r'}\\
&\le& \sup_{v\in (t,s]}\Big\{\w_{2\theta_1 r'}^{2\theta_2r'}(v-t)|\brho_{t,\mu}^\varepsilon(v,\cdot)|^{2r'}_{B^{-\beta+ \Gamma}_{p',1}}\Big\} \int_t^s \frac{dv}{\w_{\frac{-\beta+1}{\alpha} r'}^{\frac{r'}{\alpha}}(s-v) \,  \w_{2\theta_1 r'}^{2\theta_2r'}(v-t)}.
\end{eqnarray*}
\textcolor{black}{Under the conditions \eqref{C1ST_a_LT} and \eqref{C1ST_b_LT} we have assumed, we are in position  to apply Lemma  \ref{lem_gestion_sing} below to integrate the time singularities. This yields:}
\begin{eqnarray*}
&&\int_t^s \frac{dv}{\w_{\frac{-\beta+1}{\alpha} r'}^{\frac{r'}{\alpha}}(s-v)}  | \brho_{t,\mu}^\varepsilon(v,\cdot)|_{B^{ -\beta+\Gamma}_{p', 1}}^{2r'}\\
&\le& \sup_{v\in [t,s)}\Big\{ \big[\w_{\theta_1}^{\theta_2}(s-t)\big]^{2r'} |\brho_{t,\mu}^\varepsilon(v,\cdot)|^{2r'}_{B^{-\beta+\Gamma}_{p',1}}\Big\}[(s-t)\wedge1]^{1-\frac{1-\beta}{\alpha}r'-2\theta_1 r'} [(s-t)\vee 1]^{1-\frac{1}{\alpha}r'-2 \theta_2r'}\\
&=&\sup_{v\in [t,s)}\Big\{ \big[\w_{\theta_1}^{\theta_2}(s-t)\big]^{2r'} |\brho_{t,\mu}^\varepsilon(v,\cdot)|^{2r'}_{B^{-\beta+\Gamma}_{p',1}}\Big\}\bigg(\w_{\frac 1{r'}-\frac{1-\beta}{\alpha}-2\theta_1}^{\frac 1{r'}-\frac{1}{\alpha}-2\theta_2}(s-t)\bigg)^{r'}.
\end{eqnarray*}

Plugging the above estimate \textcolor{black}{into} \eqref{CI_EPS_AVANT_GR_QUADRA_PROOF_2_TER} yields
\begin{eqnarray}
&&\sup_{s\in (t,S]}\Big\{ \w_{\theta_1}^{\theta_2}(s-t) |\brho_{t,\mu}^\varepsilon(s,\cdot)|_{B_{p',1}^{-\beta+ \Gamma}}\Big\} \le  \textcolor{black}{C\Bigg\{} \textcolor{black}{|\mu|_{B^{\beta_1}_{p_1,q_1}}} \sup_{s\in (t,S]} \Big\{\w_{\theta_1- \frac{1}{\alpha}\left[\Gamma - \beta + \frac dp - \zeta_0\right]}^{\theta_2- \frac{d}{p\alpha}}(s-t)\Big\} \label{AVANT_EQUILIBRE_BIS}\\
&&\qquad  +|b|_{L^{r}(B^{\beta}_{p,q})}   \sup_{s\in (t,S]} \left\{\w_{\frac 1{r'}-\frac{1-\beta}{\alpha}-\theta_1}^{\frac 1{r'}-\frac{1}{\alpha}-\theta_2}(s-t)\right\}\bigg(\sup_{s\in (t,S]}\Big\{ \w_{\theta_1}^{\theta_2}(s-t) |\brho_{t,\mu}^\varepsilon(s,\cdot)|_{B_{p',1}^{-\beta+ \Gamma}}\Big\} \bigg)^2\textcolor{black}{\Bigg\}}.\notag
\end{eqnarray}
Our objective now consists in: (i) equilibrating and removing the singularities of the initial condition, $[\Gamma - \beta +  d/p - \zeta_0]/\alpha$, and of the integral term, $1/r'-(1-\beta)/\alpha$, in the above in \textit{small time}; (ii) removing the time dependence in \textit{large time}. 

For the \textit{small time regime} (i), 
the previous choice of $\theta_1=\theta $ in Lemma \ref{lem_unifesti_gencase2_RELOADED} fits. Namely, we can take:
\begin{equation*}
\Gamma = \eta \left\{\alpha-1+ 2\beta - \frac \alpha r - \frac dp + \zeta_0 \right\},\quad \alpha \theta = \left\{-\beta + \frac dp -\zeta_0 \right\} + \frac 1{2\eta}\left(1+\eta\right) \Gamma.
\end{equation*}
Concerning the \textit{large time regime} (ii), we want to remove the time dependence in \eqref{AVANT_EQUILIBRE_BIS}. The conditions can be summarized as:
$$\exists \theta_2 \ge 0 \text{ s.t. }\eqref{C1ST_a_LT}:\, \frac 1{r'}- \frac{1-\beta}\alpha >0,\quad \eqref{C1ST_b_LT}:\, \theta_2 < \frac{1}{2r'}\, \text{ and }\frac d{p\alpha} \ge \theta_2 \ge \frac 1{r'}- \frac{1}\alpha.$$
Note that the third condition acts precisely the other way than the condition of the \textit{small time regime} \eqref{cond_coeff_SPKR} and is reminiscent from the integrability of $v \mapsto v^{-\zeta},\, \zeta>0$ on $\R^+$. {\color{black}This third condition} is non-empty thanks to the second condition in \eqref{LT}. This suggests to choose $\theta_2$ as 
\begin{equation}\label{def_theta2}
\theta_2 := \frac 1{r'}- \frac1\alpha.
\end{equation}
As such, $\eqref{LT} \Rightarrow \alpha(1-1/r) < 2 \Rightarrow \eqref{C1ST_b_LT}$,  $\eqref{LT} \Rightarrow  \eqref{C1ST_a_LT}$. We also emphasize that the condition \eqref{C1ST_b_LT} excludes the case $(\alpha,r) = (2,\infty)$.\\

We can now plug the parameters $\Gamma$, $\theta_1$ and $\theta_2$ we chose in \eqref{AVANT_EQUILIBRE_BIS} to obtain the following estimate:
\begin{eqnarray}
&&\sup_{s\in (t,S]}\Big\{ \w_{\theta_1}^{\theta_2}(s-t) |\brho_{t,\mu}^\varepsilon(s,\cdot)|_{B_{p',1}^{-\beta+ \Gamma}}\Big\} \le  \textcolor{black}{C\Bigg\{} \textcolor{black}{|\mu|_{B^{\beta_1}_{p_1,q_1}}} \sup_{s\in (t,S]} \Big\{\w_{\frac{1-\eta}{2\eta}\frac \Gamma{\alpha}}^{ \frac 1{r'} - \frac 1\alpha - \frac d{p\alpha}}(s-t)\Big\} \notag 
\\
&&\qquad  +C_1|b|_{L^{r}(B^{\beta}_{p,q})}   \sup_{s\in (t,S]} \left\{\w_{\frac{1-\eta}{2\eta}\frac \Gamma{\alpha}}^{ 0}(s-t)\right\}\bigg(\sup_{s\in (t,S]}\Big\{ \w_{\theta_1}^{\theta_2}(s-t) |\brho_{t,\mu}^\varepsilon(s,\cdot)|_{B_{p',1}^{-\beta+ \Gamma}}\Big\} \bigg)^2\textcolor{black}{\Bigg\}}.\notag
\end{eqnarray}
From \eqref{LT}, we have 
\begin{eqnarray}
&&\sup_{s\in (t,S]}\Big\{ \w_{\theta_1}^{\theta_2}(s-t) |\brho_{t,\mu}^\varepsilon(s,\cdot)|_{B_{p',1}^{-\beta+ \Gamma}}\Big\} \le  \textcolor{black}{C\Bigg\{} \textcolor{black}{|\mu|_{B^{\beta_1}_{p_1,q_1}}}  \w_{\frac{1-\eta}{2\eta} \frac \Gamma{\alpha}}^{0}(S-t)\notag 
\\
&&\qquad  +C_1|b|_{L^{r}(B^{\beta}_{p,q})}   \w_{\frac{1-\eta}{2\eta}\frac \Gamma{\alpha}}^{0}(S-t)\bigg(\sup_{s\in (t,S]}\Big\{ \w_{\theta_1}^{\theta_2}(s-t) |\brho_{t,\mu}^\varepsilon(s,\cdot)|_{B_{p',1}^{-\beta+ \Gamma}}\Big\} \bigg)^2\textcolor{black}{\Bigg\}}.\notag
\end{eqnarray}
This concludes the proof under \A{C1}.\\

Let us now restart from the Duhamel formulation \eqref{main_MOLL} under \A{C2}. We precisely rebalance the gradient through an integration by parts to alleviate the time singularity on the heat kernel, \textcolor{black}{when the integration variable is \textit{close} to the time upper bound, using to this end the structural condition on the drift in  \A{C2}. 
{\color{black}This strategy applies} in \textit{short time}, if e.g. $s-t\le1$}.

\textcolor{black}{On the other hand, if $s-t>1$, we will actually reproduce the computations performed in \eqref{first_ctr_lem6_c1} putting therein $\beta=-1 $ on the time integration interval $[t,s-1]$}. 
Assuming w.l.o.g. that $ s-t>1$, we get,
\begin{eqnarray}
|\brho_{t,\mu}^\varepsilon(s,\cdot)|_{B^{1+\Gamma}_{p',1}} &\le&   |\mu \star p^\alpha_{s-t}|_{B^{1+\Gamma}_{p',1}}+\int_t^{s-1} dv\Big|\mathcal B_{\brho_{t,\mu}^\varepsilon}^\varepsilon(v,\cdot) \brho_{t,\mu}^\varepsilon(v,\cdot)\Big) \star  \nabla p^{\alpha}_{s-v}\Big|_{B^{1+\Gamma}_{p',1}} \notag\\
&& + \int_{s-1}^s dv\Big|\Big(\div\big(\mathcal B_{\brho_{t,\mu}^\varepsilon}^\varepsilon(v,\cdot)\big) \brho_{t,\mu}^\varepsilon(v,\cdot)\Big) \star  p^{\alpha}_{s-v}\Big|_{B^{1+\Gamma}_{p',1}}\notag\\
&&+ \int_{s-1}^s dv \Big|\Big(\mathcal B_{\brho_{t,\mu}^\varepsilon}^\varepsilon(v,\cdot) \cdot \nabla \brho_{t,\mu}^\varepsilon(v,\cdot)\Big) \star   p^{\alpha}_{s-v}(\cdot)\Big|_{B^{1+\Gamma}_{p',1}}\label{norm_IPP_BIS}.
\end{eqnarray}

As indicated above, for the time interval $v\in [t,s-1] $, we proceed as in \eqref{first_ctr_lem6_c1}. On the other hand, for $v\in [s-1,s]$,
applying successively \eqref{YOUNG} (with $m_1=1,m_2=\infty$), \eqref{PROD1}, \textcolor{black}{\eqref{YOUNG} again} and finally \eqref{BesovEmbedding} yields
\begin{eqnarray*}
 &&	\Big|\textcolor{black}{\Big(}\div\Big(\mathcal B_{\brho_{t,\mu}^\varepsilon}^\varepsilon(v,\cdot) \Big)   \brho_{t,\mu}^\varepsilon(v,\cdot)\textcolor{black}{\Big)} \star  p^{\alpha}_{s-v}\Big|_{B^{1+\Gamma}_{p',1}} \\
 &\le& C \Big|\div\Big(\mathcal B_{\brho_{t,\mu}^\varepsilon}^\varepsilon(v,\cdot)\Big)   \brho_{t,\mu}^\varepsilon(v,\cdot) \Big|_{B^{\Gamma}_{p',\infty}} \Big| p^{\alpha}_{s-v}\Big|_{B^{1}_{1,1}}\\
 &\le & C\Big|\div \Big(\mathcal B_{\brho_{t,\mu}^\varepsilon}^\varepsilon(v,\cdot)\Big)\Big|_{B^{ \Gamma}_{\infty,\infty}} \Big|\brho_{t,\mu}^\varepsilon(v,\cdot)\Big|_{B^{ \Gamma}_{p',1}} \Big| p^{\alpha}_{s-v}\Big|_{B^{1}_{1,1}}\\
 &\le & C|\div(b^\varepsilon(v,\cdot))|_{B^{-1}_{p,q}}
 \Big|\brho_{t,\mu}^\varepsilon(v,\cdot)\Big|_{B^{ 1+\Gamma}_{p',q'}} \Big|\brho_{t,\mu}^\varepsilon(v,\cdot)\Big|_{B^{ 1+\Gamma}_{p',1}} \Big| p^{\alpha}_{s-v}\Big|_{B^{1}_{1,1}}\le C|\div(b^\varepsilon(v,\cdot))|_{B^{-1}_{p,q}}\Big|\brho_{t,\mu}^\varepsilon(v,\cdot)\Big|_{B^{ 1+\Gamma}_{p',1}}^2 \Big| p^{\alpha}_{s-v}\Big|_{B^{1}_{1,1}}.\nonumber 
\end{eqnarray*}
Similarly,
\begin{eqnarray*}
 &&	\Big|\mathcal B_{\brho_{t,\mu}^\varepsilon}^\varepsilon(v,\cdot) \cdot\nabla   \brho_{t,\mu}^\varepsilon(v,\cdot) \star  p^{\alpha}_{s-v}\Big|_{B^{1+\Gamma}_{p',1}} \\
 &\le& C \Big|\mathcal B_{\brho_{t,\mu}^\varepsilon}^\varepsilon(v,\cdot)\cdot \nabla   \brho_{t,\mu}^\varepsilon(v,\cdot) \Big|_{B^{\Gamma}_{p',\infty}} \Big| p^{\alpha}_{s-v}\Big|_{B^{1}_{1,1}}\\
 &\le & C\Big|\mathcal B_{\brho_{t,\mu}^\varepsilon}^\varepsilon(v,\cdot)\Big|_{B^{ \Gamma}_{\infty,\infty}} \Big|\nabla \brho_{t,\mu}^\varepsilon(v,\cdot)\Big|_{B^{ \Gamma}_{p',1}} \Big| p^{\alpha}_{s-v}\Big|_{B^{1}_{1,1}}\\
 &\underset{\eqref{LO}}{\le} & C|b^\varepsilon\textcolor{black}{(v,\cdot)}|_{B^{-1}_{p,q}}
 \Big|\brho_{t,\mu}^\varepsilon(v,\cdot)\Big|_{B^{ 1+\Gamma}_{p',q'}} \Big|\brho_{t,\mu}^\varepsilon(v,\cdot)\Big|_{B^{ 1+\Gamma}_{p',1}} \Big| p^{\alpha}_{s-v}\Big|_{B^{1}_{1,1}}\le C|b^\varepsilon(v,\cdot)|_{B^{-1}_{p,q}}\Big|\brho_{t,\mu}^\varepsilon(v,\cdot)\Big|_{B^{ 1+\Gamma}_{p',1}}^2 \Big| p^{\alpha}_{s-v}\Big|_{B^{1}_{1,1}}.\nonumber 
\end{eqnarray*}
\textcolor{black}{Note from the above bounds that the terms $b^\varepsilon, {\rm div}(b^\varepsilon)$ naturally appear with same norm.} \textcolor{black}{Using again \eqref{NEW_CTR_CI_LT} for the initial condition}, the two above estimates and \eqref{BEFORE_NOR} 
with $\beta=-1 $ in \eqref{norm_IPP_BIS},  \textcolor{black}{we obtain} thanks to \eqref{SING_STABLE_HK} that
\begin{eqnarray*}
&&| \brho_{t,\mu}^\varepsilon(s,\cdot)|_{B^{1+\Gamma}_{p',1}} \notag\\
&\le&  \textcolor{black}{C\Bigg\{}|\mu|_{\textcolor{black}{B^{\beta_1}_{p_1,q_1}}} \w_{- \frac{1}{\alpha}\left[\Gamma + 1 + \frac dp - \zeta_0\right]}^{- \frac{d}{\alpha p}}+\textcolor{black}{\int_t^{s-1} \frac{dv}{[(s-v)\wedge 1]^{\frac {2}\alpha}[(s-v)\vee 1]^{\frac {1}\alpha}} | b^\varepsilon(v,\cdot)|_{B^{-1}_{p,q}} | \brho_{t,\mu}^\varepsilon(v,\cdot)|_{B^{ 1+\Gamma}_{p',1}}^2}\\
&&+\int_{\textcolor{black}{s-1}}^s \frac{dv}{(s-v)^{\frac 1 \alpha}} \left(|  b^\varepsilon(v,\cdot) |_{B^{-1}_{p,q}} +|  \div(b^\varepsilon(v,\cdot)) |_{B^{-1}_{p,q}}\right)  |\brho_{t,\mu}^\varepsilon(v,\cdot) |_{B^{1+\Gamma}_{p', 1}}^2 \textcolor{black}{\Bigg\}}\\
&\le & \textcolor{black}{C\Bigg\{|\mu|_{\textcolor{black}{B^{\beta_1}_{p_1,q_1}}} \w_{- \frac{1}{\alpha}\left[\Gamma + 1 + \frac dp - \zeta_0\right]}^{- \frac{d}{\alpha p}}+\int_{\textcolor{black}{t}}^s \frac{dv}{(s-v)^{\frac 1 \alpha}} \left(|  b^\varepsilon(v,\cdot) |_{B^{-1}_{p,q}} +|  \div(b^\varepsilon(v,\cdot)) |_{B^{-1}_{p,q}}\right)  |\brho_{t,\mu}^\varepsilon(v,\cdot) |_{B^{1+\Gamma}_{p', 1}}^2 \Bigg\}}.
\end{eqnarray*}
\textcolor{black}{Applying} the $L^r-L^{r'}$ H\"older inequality in time in the above equation, we get
\begin{eqnarray*}
|\brho_{t,\mu}^\varepsilon(s,\cdot)|_{B^{1+\Gamma}_{p',1}} &\le& \textcolor{black}{C\Bigg\{}|\mu|_{\textcolor{black}{B^{\beta_1}_{p_1,q_1}}} \w_{- \frac{1}{\alpha}\left[\Gamma + 1 + \frac dp - \zeta_0\right]}^{- \frac{d}{\alpha p}} \\
&&+\Big(\big| b^\varepsilon \big|_{L^{r}(B^{-1}_{p,q})} +|  \div(b^\varepsilon) |_{L^r(B^{-1}_{p,q})}\Big) \Bigg(\int_t^s \frac{dv}{\textcolor{black}{(s-v)^{\frac {r'}\alpha}}}  | \brho_{t,\mu}^\varepsilon(v,\cdot)|_{B^{1+\Gamma}_{p', 1}}^{2r'} \Bigg)^{\frac 1{r'}}\textcolor{black}{\Bigg\}}. 
\end{eqnarray*}
We now multiply both sides by $\w_{\theta_1}^{\theta_2}(s-t)$ {\color{black}and} argue as we did to pass from \eqref{CI_EPS_AVANT_GR_QUADRA_PROOF_2} to \eqref{AVANT_EQUILIBRE} to deduce that
\begin{eqnarray}
&&\sup_{s \in (t,S]}\Big\{\w_{\theta_1}^{\theta_2}(s-t)|\brho_{t,\mu}^\varepsilon(s,\cdot)|_{B^{1+\Gamma}_{p',1}}\Big\} \le  \textcolor{black}{C\Bigg\{}|\mu|_{\textcolor{black}{B^{\beta_1}_{p_1,q_1}}} \max_{s\in(t,S]} \Big\{\w_{\theta_1- \frac{1}{\alpha}\left[\Gamma + 1 + \frac dp - \zeta_0\right]}^{\theta_2- \frac{d}{p\alpha}}(s-t)\Big\}\notag\\
&&+\Big(\big|  b \big|_{L^{r}(B^{-1}_{p,q})}+|{\div}(b)|_{L^r(B_{\textcolor{black}{p,q}}^{-1})}\Big)\sup_{s\in (t,S]} \left\{\w_{\frac 1{r'}-\frac{1}{\alpha}-\theta_1}^{\frac 1{r'}-\textcolor{black}{\frac 1\alpha}-\theta_2}(s-t)\right\}\bigg(\sup_{s\in (t,S]}\Big\{ \w_{\theta_1}^{\theta_2}(s-t) |\brho_{t,\mu}^\varepsilon(s,\cdot)|_{B_{p',1}^{1 + \Gamma}}\Big\} \bigg)^2, \notag 
\end{eqnarray}
where we implicitly assumed that the following condition hold to apply Lemma \ref{lem_gestion_sing}:
$$2\theta_i r'<1\iff \theta_i <\frac {1}{2r'},\quad i=1,2 \text{ and } \frac 1{r'} - \frac 1\alpha >0.$$
To equilibrate the singularities in small time regime, we require, in addition \textcolor{black}{to} the above constraints:
$$\exists \Gamma >0,\, \theta_1 \ge 0,\, {\rm s.t. }\quad \frac{\Gamma + 1+  \frac dp- \zeta_0}{\alpha} < \theta_1 < \frac{1}{r'}-\frac{1}{\alpha}.$$
As $1/r' - 1/\alpha \le 1/(2r') \Leftrightarrow \alpha(1-1/r) \le 2$ is always satisfied, conditions reduce to
$$\left\{\exists \Gamma >0,\, \theta_1 \ge 0,\, {\rm s.t. }\quad \frac{\Gamma + 1+  \frac dp- \zeta_0}{\alpha} < \theta_1 < \frac{1}{r'}-\frac{1}{\alpha} \text{ and }  \frac 1{r'} - \frac 1\alpha >0\right\} \Leftrightarrow \eqref{THE_COND_CI}.$$
Reasoning as we did for the proof under \A{C1} gives
\begin{equation*}
\Gamma = \eta \left\{ \alpha - 2 - \frac \alpha r - \frac dp   +\zeta_0 \right\},\quad \alpha \theta_1 = \left\{1  + \frac dp -\zeta_0 \right\} + \frac {1+\eta}{2\eta} \Gamma,\ \eta \in(0,1).
\end{equation*}
We now move to the long time regime. Our constraints are
$$ \frac 1{r'}-\textcolor{black}{\frac 1\alpha} \le \theta_2 \le \frac d{p\alpha},\quad \theta_2 < \frac 1{2r'} \text{ and } \frac 1{r'} - \frac 1\alpha >0.$$

The condition \eqref{LT} then allows to set, as under \A{C1} above, $\theta_2:=1/r'-1/\alpha $, {\color{black}and} it then follows that:
\begin{eqnarray}
&&\sup_{s\in (t,S]}\Big\{ \w_{\theta_1}^{\theta_2}(s-t) |\brho_{t,\mu}^\varepsilon(s,\cdot)|_{B_{p',1}^{-\beta+ \Gamma}}\Big\} \le  \textcolor{black}{C\Bigg\{} \textcolor{black}{|\mu|_{B^{\beta_1}_{p_1,q_1}}}  \w_{\frac{1-\eta}{2\eta}\frac \Gamma{\alpha}}^{0}(S-t)\notag 
\\
&&\qquad  +C_1\Big(|b|_{L^{r}(B^{-1}_{p,q})}+|\div(b)|_{L^{r}(B^{-1}_{p,q})}\Big)   \w_{\frac{1-\eta}{2\eta} \frac \Gamma{\alpha}}^{0}(S-t)\bigg(\sup_{s\in (t,S]}\Big\{ \w_{\theta}^{0}(s-t) |\brho_{t,\mu}^\varepsilon(s,\cdot)|_{B_{p',1}^{-\beta+ \Gamma}}\Big\} \bigg)^2\textcolor{black}{\Bigg\}},\notag
\end{eqnarray}
provided again that $ 1/r'-1/\alpha<1/(2r')\iff 1/r'<2\alpha\iff \alpha(1-\frac 1r)<2$. The proof of Lemma \ref{lem_unifesti_gencase2_RELOADED_LT} is complete. 
\end{proof}

\begin{proof}[Proof of Lemma \ref{lem_quad_gronv_C1_LT}.]
The proof is rather similar to the one of Lemma \ref{lem_quad_gronv_C1}, except that we now use time weights which have a different behavior in short and long time (and are actually, from the choice of the exponents in Lemma \ref{lem_quad_gronv_C1} bounded by 1).

 {\color{black} From  Lemma \ref{lem_unifesti_gencase2_RELOADED_LT}, for any $\textcolor{magenta}{S\le T}$, the mapping $$f_t^{\textcolor{black}{\varepsilon}}:s\in \textcolor{black}{(}t,S]\mapsto f_t^{\textcolor{black}{\varepsilon}}(s):=\sup_{v\in (t,s]}\w_{\theta_1}^{\theta_2}(v-t) |\brho_{t,\mu}^\varepsilon(v,\cdot)|_{B^{ -\beta+\Gamma}_{p',{\color{black}1}}}$$ satisfies an inequality of the form:
\begin{equation}\label{PREAL_TRINOME_LT}
0 \le  a_t(s) - f_t^{\textcolor{black}{\varepsilon}}(s)+c_t(s)(f_t^{\textcolor{black}{\varepsilon}}(s))^2,\,t< s\le S,
\end{equation}
where, 
\begin{eqnarray*}
a_t(s)&=& C_{\ref{lem_unifesti_gencase2_RELOADED_LT}} |\mu|_{\textcolor{black}{B^{\beta_1}_{p_1,q_1}}} 
\w_{\frac{1-\eta}{2\eta} \frac \Gamma{\alpha}}^{0} (s-t)
= : C_{\ref{lem_unifesti_gencase2_RELOADED_LT}}  c_0 \w_{\frac{1-\eta}{2\eta} \frac \Gamma{\alpha}}^0 (\textcolor{black}{s}-t),\\
c_t(s)&=& C_{\ref{lem_unifesti_gencase2_RELOADED_LT}} \textcolor{black}{\Big(}| b|_{L^{r}(B^{\beta}_{p,q})}+\textcolor{black}{| {\rm div}(b)|_{L^{r}(B^{\beta}_{p,q})}\ind_{\beta=-1}\Big)} 
\w_{\frac{1-\eta}{2\eta} \frac \Gamma{\alpha}}^0 (\textcolor{black}{s}-t)
= : C_{\ref{lem_unifesti_gencase2_RELOADED_LT}} c_b \w_{\frac{1-\eta}{2\eta} \frac \Gamma{\alpha}}^0 (\textcolor{black}{s}-t)
.
\end{eqnarray*}

Define, for {\color{black}$\tau \in (0,T] $}, the polynomial $P_{\tau}(z) = a_t(\tau) -z+c_t(\tau)z^2$. \textcolor{black}{{\color{black} Note that} the time dependent coefficients $s\in (t,S]\mapsto a_t(s),c_t(s) $ are increasing {\color{black} and bounded by $C_{\ref{lem_unifesti_gencase2_RELOADED_LT}}  c_0 + C_{\ref{lem_unifesti_gencase2_RELOADED_LT}} c_b $ (as the times weights are uniformly bounded by 1).  Thus, we have from \eqref{PREAL_TRINOME_LT} that}} $P_{\tau}(f_t^{\textcolor{black}{\varepsilon}}(s)) \ge 0$, $t< s\le S$. Moreover, as soon as $c_t(\tau)a_t(\tau) < 1/4$, \textcolor{black}{which will always be the case provided that $T$ is small enough} or $|\mu|_{B_{p_1,q_1}^{\beta_1}} $ is smaller than some $\mathcal C_0>0$, this polynomial admits two positive roots and since from Lemma \ref{lem_cont0} we have that for every fixed $\varepsilon >0$, $s \mapsto f_t^\varepsilon (s)$ is continuous and $ f_t^\varepsilon (s) \to 0$ as $s \to t$, we obtain that $f_t^\varepsilon(s)$ is bounded by the smaller root of the polynomial, namely,
$$\forall t< s \le S,\, f_t^\varepsilon(s) \le \textcolor{black}{\frac{1-\sqrt{1-4c_t(S)a_t(S)}}{2c_t(S)}}.$$
Setting then
{\color{black}
$$\mathcal C_0 := [2C_{\ref{lem_unifesti_gencase2_RELOADED}}^2c_b]^{-1} \Longrightarrow \forall \tau \in (0,T],\quad 4 c_t(\tau)a_t(\tau) \le \frac 12,$$
and recalling
$$
\mathcal T_1 =( t + \left[8C_{\ref{lem_unifesti_gencase2_RELOADED}}^2 c_0c_b\right]^{-[\eta/(1-\eta)]\, [\alpha/\Gamma]})\wedge T \Longrightarrow 4 c_t(\mathcal T_1)a_t(\mathcal T_1) \le \frac 12,$$
we obtain, defining 
$$\mathcal T_2 : = \mathcal T_1 \mathbf 1_{c_0 > \mathcal C_0} + T \mathbf{1}_{c_0 \le \mathcal C_0},$$
that
\begin{align*}
f_t^\varepsilon(s)
&\le a_t(\mathcal T_2)(1+2^{\frac 52}c_t(\mathcal T_2)a_t(\mathcal T_2)):=C_{\ref{lem_quad_gronv_C1_LT}}.
\end{align*}
This gives the claim.
}
}

\end{proof}

{\color{black}Let us point out that t}he stability analysis of Lemma \ref{SECOND_STAB_BIS} could then be performed rather similarly, from the control of Lemma \ref{lem_quad_gronv_C1_LT} with less stringent time integrability conditions, i.e. in $L_{{\ww}_{\theta}}^{\tilde r}((t,S], B^{-\textcolor{black}{\beta}+ \vartheta\Gamma}_{p',1}), \tilde r<+\infty$ or possibly in $L_{{\ww}_{\theta}}^{\infty}((t,S], B^{-\textcolor{black}{\beta}+ \vartheta\Gamma}_{p',1})$ establishing an appropriate Gronwall-Volterra type lemma following e.g. the approach of \cite{webb:19}. 

{\color{black}\section{Integration of the time singularities}}

\begin{lemme}
	\label{lem_gestion_sing}
For $0\le a_1,a_2,b_1,b_2<1$, for $t<s$ in $(\R^{+})^2$ define, with our notations {\color{black}\eqref{DEF_WEIGHT}},
\begin{eqnarray*}
I_{t,s} := \int_t^s {\rm w}_{\textcolor{black}{-}a_1}^{\textcolor{black}{-}a_2}(s-v) {\rm w}_{\textcolor{black}{-}b_1}^{\textcolor{black}{-}b_2}(v-t)dv.
\end{eqnarray*}
Then, there exists $C(a,b_1,b_2)>0$ s.t. for $t<s$ in $(\R^{+})^2$
\begin{eqnarray*}
I_{t,s}\le C(a,b_1,b_2) \w_{1-b_1-a_1}^{1-b_2-a_2}(s-t).
\end{eqnarray*}
\end{lemme}

\begin{proof}[Proof of Lemma \ref{lem_gestion_sing}]
Assume first that $s-t \le 2$, we have, for any $t<v<s$,  $[(v-t) \vee 1]^{-b_2},\, [(s-v) \vee 1]^{-a_2}  \le 1$ and 
\begin{eqnarray*}
 [(v-t)\wedge 1]^{-b_1} \le 2^{b_1} (v-t)^{-b_1},\,  [(s-v)\wedge 1]^{-a_1} \le 2^{a_1} (s-v)^{-a_1}.
\end{eqnarray*}
so that, as $a_1,b_1<1$,
\begin{eqnarray*}
I_{t,s} &\le& 2^{b_1+a_1}\int_t^s (s-v)^{-a_1}(v-t)^{-b_1} dv\\
&\le& 2^{b_1+a_1} B(1-a_1,1-b_1) (s-t)^{1-b_1-a_1}.
\end{eqnarray*}
Thus,
\begin{eqnarray*}
I_{t,s} &\le&  2^{b_1+a_1} B(1-a_1,1-b_1) (s-t)^{1-b_1-a_1}.
\end{eqnarray*}
Assume now that $s-t > 2$. \textcolor{black}{Introduce for $0\le t\le t'\le  s'\le s $,
$$\tilde I_{t',s'}(t,s)=\int_{t'}^{s'}\w_{-{a_1}}^{-a_2}(s-v)\w_{-{b_1}}^{-b_2}(t-v) dv,$$
so that $I_{t,s}=\tilde I_{t,s}(t,s) $ and
\begin{eqnarray*}
I_{t,s} = \tilde I_{t,t+1}(t,s) + \tilde I_{t+1,s-1}(t,s)+ \tilde I_{s-1,s}(t,s).
\end{eqnarray*}
}
We have, recalling $0 \le a_1,a_2,b_1,b_2<1$
\begin{eqnarray*}
\textcolor{black}{\tilde I_{t,t+1}(s,t)}  &=& \int_{t}^{t+1}(s-v)^{-a_2} [(v-t)\wedge 1]^{-b_1}  dv\\
&\le& \frac 1{1-b_1} (s-t-1)^{-a_2} \le \frac{2^{a_2}}{1-b_1}(s-t)^{-a_2},\\
\textcolor{black}{\tilde I_{t+1,s-1}(s,t)}  &=& \int_{t+1}^{s-1}(s-v)^{-a_2}  (v-t)^{-b_2}dv\\
& \le& B(1-a_2,1-b_2)(s-t)^{1-a_2-b_2},\\
\textcolor{black}{\tilde I_{s-1,s}(s,t)}  &=& \int_{s-1}^{s}(s-v)^{-a_1}  (v-t)^{-b_2}dv\\
&\le& \frac 1{1-a_1} (s-t-1)^{-b_2} \le \frac{2^{b_2}}{1-a_1}(s-t)^{-b_2},\\
\end{eqnarray*}
from which we deduce that
\begin{eqnarray*}
I_{t,s}\le \Big\{ 2^{a_2}(1-b_1)^{-1}+ 2^{b_2}(1-a_1)^{-1} + B(1-a_2,1-b_2)\Big\}(s-t)^{1-a_2-b_2}.
\end{eqnarray*}
It therefore holds that there exists $C(a,b_1,b_2)>0$ s.t. for $t<s$ in $(\R^{+})^2$
\begin{eqnarray*}
I_{t,s}\le C(a,b_1,b_2) [(s-t)\wedge1]^{1-b_1-a_1} [(s-t)\vee 1]^{1- b_2-a_2}.
\end{eqnarray*}
\end{proof}

{\color{black}\section{Estimates in Besov norm for the stable heat-kernel}}
\label{HK_APPENDIX}
We consider here the proof of \eqref{SING_STABLE_HK} for a negative regularity index $\gamma<0 $. As a preliminary, we recall the $L^\ell$-estimate of $p_{\alpha}$ from \cite{CdRM-20} (see again Lemma 11 and (3.19) therein):
$$
|\partial^{\mathbf a}p_{\alpha}(s-t,\cdot)|_{L^\ell}\le C(s-t)^{-(\frac d{\alpha \ell'}+\frac{|\mathbf a|}\alpha)},\qquad |\mathbf a|\le 1,\,\ell\in[1,\infty], \ell^{-1}+{\ell'}^{-1}=1.
$$
Let us start with the long time estimates, i.e. we assume that $s-t\ge 1 $. In that case write from the definition in \eqref{HEAT_CAR}:
\begin{align*}
|\mathcal F^{-1}(\phi\mathcal F(\partial^{\mathbf a}p_{\alpha}(s-t,\cdot)))|_{\textcolor{black}{L^\ell}}&=|\mathcal F^{-1}(\phi) * \partial^{\mathbf a} p_{\alpha}(s-t,\cdot)|_{L^\ell}\le |\mathcal F^{-1}(\phi)|_{L^1} | \partial^{\mathbf a}p_{\alpha}(s-t,\cdot)|_{L^\ell}\\
&\le C(s-t)^{-(\frac d{\alpha \ell'}+\frac{|\mathbf a|}\alpha)}.
\end{align*}
On the other hand, for $\gamma<0 $ (and $m<+\infty$):
\begin{align*}
\mathcal T_{\ell,m}^\gamma(\partial^{\mathbf a}p_{\alpha}(s-t,\cdot)))&=\left(\int_0^1\,\frac{dv}{v}v^{(n-\gamma/\textcolor{black}{\alpha})m}|\partial^n_v \tilde p_\alpha(v,\cdot)*\partial^{\mathbf a} p_{\alpha}(s-t,\cdot)|^m_{L^\ell}\right)^{\frac 1m}\\
&\le \left(\int_0^1\,\frac{dv}{v}v^{(n-\gamma/\textcolor{black}{\alpha})m}|\partial^n_v \tilde p_\alpha(v,\cdot)|_{L^1}|^m |\partial^{\mathbf a} p_{\alpha}(s-t,\cdot)|^m_{L^\ell}\right)^{\frac 1m}\\
&\le C(s-t)^{-\frac{d}{\alpha \ell'}}\left(\int_0^1\,\frac{dv}{v}v^{(n-\gamma/\textcolor{black}{\alpha})m}v^{-nm}\right)^{\frac 1m}\le C(s-t)^{-\frac{d}{\alpha \ell'}}.
\end{align*}
The case $m=+\infty$ is handled similarly. This completes the proof of \eqref{SING_STABLE_HK} for $\gamma<0 $ and $s-t\ge 1 $. 

Let us now turn to the \textit{short time} estimate $s-t\le 1 $. Write then, still from the definition in \eqref{HEAT_CAR}:
\begin{align*}
|\mathcal F^{-1}(\phi\mathcal F(\partial^{\mathbf a}p_{\alpha}(s-t,\cdot)))|_{\textcolor{black}{L^\ell}}&=|\mathcal F^{-1}(\phi) * \partial^{\mathbf a} p_{\alpha}(s-t,\cdot)|_\ell\le |\mathcal F^{-1}(\phi)|_{L^\ell} | \partial^{\mathbf a}p_{\alpha}(s-t,\cdot)|_{L^1}\\
&\le C(s-t)^{-\frac{|\mathbf a|}\alpha},
\end{align*}
and no other time singularity than the one induced by the derivatives (when $|\mathbf a|\neq 0 $) appear for this part of the norm. Turning to the thermic part yields (considering again w.l.o.g $m<+\infty $):
\begin{align*}
\mathcal T_{\ell,m}^\gamma(\partial^{\mathbf a}p_{\alpha}(s-t,\cdot)))&=\left(\int_0^1\,\frac{dv}{v}v^{(n-\gamma/\textcolor{black}{\alpha})m}|\partial^n_v \tilde p_\alpha(v,\cdot)*\partial^{\mathbf a} p_{\alpha}(s-t,\cdot)|^m_{L^\ell}\right)^{\frac 1m}\\
\le& \left(\int_0^{s-t}\,\frac{dv}{v}v^{(n-\gamma/\textcolor{black}{\alpha})m}|\partial^n_v \tilde p_\alpha(v,\cdot)*\partial^{\mathbf a} p_{\alpha}(s-t,\cdot)|^m_{L^\ell}\right)^{\frac 1m}\\
&+\left(\int_{s-t}^1\,\frac{dv}{v}v^{(n-\gamma/\textcolor{black}{\alpha})m}|\partial^n_v \tilde p_\alpha(v,\cdot)*\partial^{\mathbf a} p_{\alpha}(s-t,\cdot)|^m_{L^\ell}\right)^{\frac 1m}\\
\le& \left(\int_0^{s-t}\,\frac{dv}{v}v^{(n-\gamma/\textcolor{black}{\alpha})m}|\partial^n_v \tilde p_\alpha(v,\cdot)|_{L^1}^m|\partial^{\mathbf a} p_{\alpha}(s-t,\cdot)|^m_{L^\ell}\right)^{\frac 1m}\\
&+\left(\int_{s-t}^1\,\frac{dv}{v}v^{(n-\gamma/\textcolor{black}{\alpha})m}|\partial^n_v \tilde p_\alpha(v,\cdot)|_{L^\ell}^m|\partial^{\mathbf a} p_{\alpha}(s-t,\cdot)|^m_{L^1}\right)^{\frac 1m}\\
\le& C\Big( (s-t)^{-(\frac{d}{\alpha \ell'}+\frac{|\mathbf a|}\alpha)}\Big(\int_0^{s-t}\frac{dv}{v}v^{-\frac \gamma \alpha m}\Big)^{\frac 1m}+(s-t)^{-\frac{|\mathbf a|}{\alpha}}\left(\int_{(s-t)}^1\,\frac{dv}{v}v^{(n-\gamma/\textcolor{black}{\alpha})m}v^{-nm-\frac{d}{\alpha \ell '}m}\right)^{\frac 1m}\Big)\\
\le& C\left( (s-t)^{-(\frac{d}{\alpha \ell'}+ \frac{|\mathbf a|}{\alpha}+\frac \gamma\alpha)}+(s-t)^{-\frac{|\mathbf a|}{\alpha}}[(s-t)^{-(\frac{\gamma}{\alpha}+\frac{d}{\alpha \ell'})\I_{-(\frac{\gamma}{\alpha}+\frac{d}{\alpha \ell'})m<0}}+|\ln(s-t)|\I_{\gamma=-\frac d{\ell'}}]\right)\\
\le &C (s-t)^{-(\frac{|\mathbf a|}{\alpha}+[\frac{d}{\alpha \ell'} +\frac \gamma \alpha]_+])}(1+|\ln(s-t)|\I_{\gamma=-\frac d{\ell'}}]\I_{m<+\infty}).
\end{align*}
Observe that the logarithmic correction vanishes does not appear when $m=\infty $. This therefore concludes the proof of \eqref{SING_STABLE_HK}.

\vspace{.3cm}
\subsection*{Data availability}
Data sharing not applicable to this article as no datasets were generated or analysed during the current study.

{\color{black}
	\subsection*{Acknowledgements} For the second author, the paper was prepared within the framework of the Basic Research Program at HSE University and the RSF project No. 24-11-00123. The first author thanks the Centre Henri Lebesgue ANR-11-LABX-0020-01 for creating
	an attractive mathematical environment. The authors are also thankful to the anonymous referees who carefully review the first version of the paper and helped us to improve it.}

\end{document}